\documentclass[11pt]{article}

\usepackage{amssymb}
\usepackage{amsmath}
\usepackage{amsthm}
\usepackage{mathdots}
\usepackage{xcolor}
\usepackage{graphicx}
\usepackage{tikz-cd}
\usepackage{pgf,acadpgf}
\usepackage{enumerate}
\usepackage[normalem]{ulem}

\long\def\commentout#1{}

\numberwithin{equation}{section}

\newtheoremstyle{slplain}
  {\topsep}
  {\topsep}
  {\slshape}
  {0pt}
  {\bfseries}
  {.}
  {0.5em}
  {}

\theoremstyle{slplain}
  \newtheorem{THM}{Theorem}[section]
  \newtheorem{LEM}[THM]{Lemma}
  \newtheorem{PROP}[THM]{Proposition}
  \newtheorem{COR}[THM]{Corollary}
  
  \newtheorem*{THMC}{Theorem C}
  \newtheorem*{THMD}{Theorem D}
  \newtheorem*{THME}{Theorem E}
  \newtheorem*{LEMA}{Lemma A}
  \newtheorem*{LEMB}{Lemma B}

\theoremstyle{definition}
  \newtheorem{DEF}[THM]{Definition}
  
  \newtheorem{REM}[THM]{Remark}

\newcommand{\Fraisse}{Fra\"\i ss\'e}

\newcommand{\lex}{\le_{\mathit{lex}}}
\newcommand{\alex}{\le_{\mathit{alex}}}
\newcommand{\refl}{{\mathit{refl}}}
\newcommand\nlongrightarrow{\longrightarrow\kern -1.45em/\kern 0.9em}

\renewcommand{\preceq}{\preccurlyeq}

\newcommand{\UNION}{\bigcup}

\renewcommand{\le}{\leqslant}
\renewcommand{\ge}{\geqslant}

\newcommand{\0}{\varnothing}

\renewcommand{\phi}{\varphi}
\renewcommand{\epsilon}{\varepsilon}
\newcommand{\AAA}{\mathbf{A}}
\newcommand{\BB}{\mathbf{B}}
\newcommand{\CC}{\mathbf{C}}
\newcommand{\DD}{\mathbf{D}}

\newcommand{\KK}{\mathbf{K}}

\newcommand{\NN}{\mathbb{N}}

\newcommand{\QQ}{\mathbb{Q}}

\newcommand{\ZZ}{\mathbb{Z}}
\newcommand{\union}{\cup}
\newcommand{\restr}[2]{\hbox{$#1$}\hbox{$\upharpoonright$}_{#2}}

\newcommand{\Boxed}[1]{\mbox{$#1$}}
\newcommand{\id}{\mathrm{id}}

\newcommand{\Ob}{\mathrm{Ob}}

\newcommand{\high}{\mathrm{ht}}

\newcommand{\LO}{\mathit{LO}}
\newcommand{\Age}{\mathrm{Age}}

\newcommand{\frakA}{\mathfrak{A}}
\newcommand{\frakS}{\mathfrak{S}}

\newcommand{\calA}{\mathcal{A}}
\newcommand{\calB}{\mathcal{B}}
\newcommand{\calC}{\mathcal{C}}
\newcommand{\calD}{\mathcal{D}}
\newcommand{\calE}{\mathcal{E}}

\newcommand{\calH}{\mathcal{H}}

\newcommand{\calK}{\mathcal{K}}

\newcommand{\calQ}{\mathcal{Q}}

\newcommand{\calS}{\mathcal{S}}

\newcommand{\calU}{\mathcal{U}}

\newcommand{\ChEmb}{{\mathbf{Ch}_{\mathit{emb}}}}
\newcommand{\Perm}{{\mathbf{Perm}_{\mathit{emb}}}}
\newcommand{\PowChEmb}[1]{{\mathbf{Ch}^{#1}_{\mathit{emb}}}}
\newcommand{\PQ}{{\mathbf{P}_{\QQ}}}
\newcommand{\PowPQ}[1]{{\mathbf{P}^{#1}_{\QQ}}}

\newcommand{\Aut}{\mathrm{Aut}}
\newcommand{\tp}{\mathrm{tp}}

\newcommand{\Emb}{\mathrm{Emb}}
\newcommand{\im}{\mathrm{im}}

\newcommand{\spec}{\mathrm{spec}}
\newcommand{\tree}[1]{\langle#1\rangle}
\newcommand{\dom}{\mathrm{dom}}
\newcommand{\Br}{\mathrm{Br}}

\title{\Fraisse's Conjecture and big Ramsey degrees\\of structures admitting\\finite monomorphic decomposition\\\textcolor{blue}{(CORRECTED VERSION)}}
\author{%
  Dragan Ma\v sulovi\'c (corresponding author)\\
  University of Novi Sad, Faculty of Sciences\\
  Department of Mathematics and Informatics\\
  Trg Dositeja Obradovi\'ca 3, 21000 Novi Sad, Serbia\\
  e-mail: dragan.masulovic@dmi.uns.ac.rs
\and
  Veljko Tolji\'c\\
  University of Novi Sad, Faculty of Sciences\\
  Department of Mathematics and Informatics\\
  Trg Dositeja Obradovi\'ca 3, 21000 Novi Sad, Serbia\\
  e-mail: veljko.toljic@dmi.uns.ac.rs
}

\begin{document}
\maketitle

\begin{abstract}
  Monomorphic structures (structures with only one kind of $n$-element substructures, for each $n$)
  were introduced and studied by R.\ \Fraisse\ as natural generalizations of chains ($=$ linear orders).
  This notion was later generalized by Pouzet and Thier\'y to structures admitting a finite monomorphic decomposition.
  In this paper we characterize countable structures admitting a finite monomorphic decomposition which
  have finite big Ramsey degrees. The necessary prerequisite for that is the characterization of monomorphic structures with finite big Ramsey degrees.
  Interestingly, both characterizations require deep structural properties of chains.
  \Fraisse's Conjecture (actually, its positive resolution due to Laver) is instrumental in
  the characterization of monomorphic structures with finite big Ramsey degrees, while
  the analysis of big Ramsey combinatorics of structures admitting a finite monomorphic decomposition
  requires a product Ramsey theorem for big Ramsey degrees of chains. We find this last result particularly intriguing because
  big Ramsey degrees are known to exhibit irregular behavior when it comes to general product statements. As a spin-off of the product Ramsey theorem,
  we \sout{provide an alternative proof of Hubi\v cka's result} \textcolor{blue}{prove} that
  the generic \textcolor{blue}{2-dimensional} partial order has finite big Ramsey degrees.

  \textcolor{blue}{%
      In the addendum, we combine a recent result by Oudrar and Pouzet with our
      analysis of finite big Ramsey degrees for structures admitting finite monomorphic decomposition
      to characterize the existence of finite Big Ramsey degrees for all countable relational structures
      whose language has a linear order and age has polynomial growth.
  }

  \bigskip

  \textbf{Key Words and Phrases:} big Ramsey degrees, countable chains, monomorphic structures, structures admitting finite monomorphic decomposition

  \bigskip

  \noindent
  \textbf{AMS Subj. Classification (2020):} 06A05, 05C55
\end{abstract}

\section{Introduction}

Motivated by the Infinite Ramsey Theorem, and prompted by Galvin and Laver, in 1979 Devlin started
the analysis of big Ramsey combinatorics of countable chains ($=$ linearly ordered sets) more complex than $\omega$
by showing that finite chains have finite big Ramsey degrees in $\QQ$ -- the chain of the rationals \cite{devlin}.
This result takes care of all non-scattered countable chains since it is easy to show that bi-embeddable
countable relational structures have the same big Ramsey combinatorics.

Big Ramsey combinatorics of scattered countable chains proved to be challenging in a different manner.
It was shown in \cite{masul-sobot} that a countable ordinal $\alpha$ has finite big Ramsey degrees if and only if
$\alpha < \omega^\omega$. This result was then upgraded to arbitrary countable scattered chains by
Ma\v sulovi\'c in~\cite{masul-fbrd-chains} and Dasilva Barbosa, Ma\v sulovi\'c and Nenadov in~\cite{masul-fbrd-finale},
where countable scattered chains having finite big Ramsey degrees were characterized as precisely those having finite Hausdorff rank.
(All the necessary notions are introduced in Section~\ref{mnmrf.sec.prelim}.)

The analysis of big Ramsey degrees of countable chains naturally generalizes to the class of
monomorphic structures introduced by \Fraisse\ in~\cite{Fraisse-ThRel}.
An infinite relational structure is \emph{monomorphic} if
it has, up to isomorphism, only one $n$-element substructure for each $n \in \NN$.
The paper~\cite{masul-fbrd-finale} shows that a monomorphic structure chainable by a countable chain
with finite big Ramsey degrees has itself finite big Ramsey degrees.
In Section~\ref{mnmrf.sec.mnmrf} we complete the characterization of countable
monomorphic structures with finite big Ramsey degrees by showing that this is also a necessary condition.
Interestingly, Laver's positive resolution of \Fraisse's Conjecture was instrumental in this characterization.

These results extend further to structures admitting a finite monomorphic decomposition,
which were introduced by Pouzet and Thi\'ery in~\cite{pouzet-thiery-I}.
We show in Section~\ref{mnmrf.sec.fmd} that a countable structure admitting a finite monomorphic decomposition
has finite big Ramsey degrees if and only if so does every monomorphic part in its minimal monomorphic decomposition.

The analysis of big Ramsey combinatorics of structures admitting a finite monomorphic decomposition
requires a product Ramsey theorem for big Ramsey degrees for chains, which we prove in Section~\ref{mnmrf.sec.prod-thm}.
Although of technical nature, we find this product Ramsey result particularly intriguing because
big Ramsey degrees are known to exhibit irregular behavior when it comes to general product statements.

We conclude the paper with a spin-off of the product Ramsey theorem for big Ramsey degrees for chains:
in Section~\ref{mnmrf.sec.brd-gen-poset} we \sout{provide an alternative proof of Hubi\v cka's result}
\textcolor{blue}{prove} that the generic \textcolor{blue}{2-dimensional} partial order has finite big Ramsey
degrees~\sout{\cite{hubicka-param-spaces}}.

\section{Preliminaries}
\label{mnmrf.sec.prelim}

\paragraph{Relational structures.}
A \emph{relational language} is a set $L$ of \emph{relation symbols}, each of which comes with its \emph{arity}.
An \emph{$L$-structure} $\calA = (A, L^A)$ is a set $A$ together with a set $L^A$ of
relations on $A$ which are interpretations of the corresponding symbols in $L$.
The underlying set of a structure $\calA$, $\calA_1$, $\calA^*$, \ldots\ will always be denoted by its roman
letter $A$, $A_1$, $A^*$, \ldots\ respectively.
A structure $\calA = (A, L^A)$ is \emph{finite} if $A$ is a finite set.

Let $\calA$ and $\calB$ be $L$-structures and let $f : A \to B$ be a mapping.
We say that $f$ is a \emph{homomorphism} if
      $R^A(a_1, \ldots, a_n) \Rightarrow R^B(f(a_1), \ldots, f(a_n))$
      for all $R \in L$ and $a_1, \ldots, a_n \in A$, where $n$ is the arity of~$R$.
The mapping $f$ is an \emph{embedding}, in symbols $f : \calA \hookrightarrow \calB$, if it is injective and
      $R^A(a_1, \ldots, a_n) \Leftrightarrow R^B(f(a_1), \ldots, f(a_n))$
      for all $R \in L$ and $a_1, \ldots, a_n \in A$, where $n$ is the arity of~$R$.
We say that $\calA$ and $\calB$ are \emph{bi-embeddable} if there exist embeddings $\calA \hookrightarrow \calB$
and $\calB \hookrightarrow \calA$.

Surjective embeddings are \emph{isomorphisms}. We write $\calA \cong \calB$ to denote that $\calA$ and $\calB$ are isomorphic.
An \emph{automorphism} of an $L$-structure $\calA$ is an isomorphism $\calA \to \calA$. Let $\Aut(\calA)$ denote the
\emph{automorphism group} of~$\calA$. A structure $\calA$ is \emph{rigid} if $\Aut(\calA) = \{\id_A\}$,
where $\id_A$ denotes the \emph{identity mapping $A \to A$}.

An $L$-structure $\calA$ is a \emph{substructure} of an $L$-structure
$\calB$, in symbols $\calA \le \calB$, if the identity map is an embedding of $\calA$ into $\calB$.
Let $\calA$ be a structure and $\0 \ne B \subseteq A$. Then $\calA[B] = (B, \restr {L^A} B)$ denotes
the \emph{substructure of $\calA$ induced by~$B$}, where $\restr {L^A} B$ denotes the restriction of
$L^A$ to~$B$.

For a homomorphism $f : \calA \to \calB$ let $\im(f) = \{f(a) : a \in A \} \subseteq B$ denote the
\emph{image of $\calA$ under $f$}. If $f$ is an embedding then $\calB[\im(f)] \cong \calA$.

\paragraph{Big Ramsey degrees.}
Let $L$ be a relational language. 
For $L$-structures $\calA$ and $\calB$ let
$\Emb(\calA, \calB)$ denote the set of all the embeddings $\calA \hookrightarrow \calB$.
For $L$-structures $\calA$, $\calB$, $\calC$ and positive integers $k, t \in \NN$
we write $\calC \longrightarrow (\calB)^{\calA}_{k, t}$ to denote that for every $k$-coloring $\chi : \Emb(\calA, \calC) \to k$
there is an embedding $w \in \Emb(\calB, \calC)$ such that $|\chi(w \circ \Emb(\calA, \calB))| \le t$.
We say that $\calA$ has a \emph{finite embedding big Ramsey degree in $\calC$}
if there exists a positive integer $t$ such that for each $k \in \NN$ we have that
$\calC \longrightarrow (\calC)^{\calA}_{k, t}$.
The least such $t$ is then denoted by $T(\calA, \calC)$. If such a $t$ does not exist
we say that $\calA$ \emph{does not have a finite embedding big Ramsey degree in $\calC$} and write
$T(\calA, \calC) = \infty$. Finally, we say that an infinite $L$-structure $\calC$ \emph{has finite embedding big Ramsey degrees}
if $T(\calA, \calC) < \infty$ for every finite substructure $\calA$ of~$\calC$.

Analogously, for $L$-structures $\calA$ and $\calB$ let
$\binom \calB \calA$ denote the set of all the substructures of $\calB$ that are isomorphic to $\calA$.
For $L$-structures $\calA$, $\calB$, $\calC$ and positive integers $k, t \in \NN$
we write $\calC \overset\sim\longrightarrow (\calB)^{\calA}_{k, t}$ to denote that for every $k$-coloring $\chi : \binom \calC \calA \to k$
there is a $\calB' \in \binom \calC \calB$ such that $|\chi(\binom{\calB'}{\calA})| \le t$.
We say that $\calA$ has a \emph{finite structural big Ramsey degree in $\calC$}
if there exists a positive integer $t$ such that for each $k \in \NN$ we have that
$\calC \overset\sim\longrightarrow (\calC)^{\calA}_{k, t}$.
The least such $t$ is then denoted by $\tilde T(\calA, \calC)$. If such a $t$ does not exist
we say that $\calA$ \emph{does not have a finite structural big Ramsey degree in $\calC$} and write
$\tilde T(\calA, \calC) = \infty$. Finally, we say that an infinite $L$-structure $\calC$ \emph{has finite structural big Ramsey degrees}
if $\tilde T(\calA, \calC) < \infty$ for every finite substructure $\calA$ of~$\calC$.

The two kinds of big Ramsey degrees are closely related:

\begin{THM}\cite{zucker-brd}\label{mnmrf.thm.zucker-bigT}
  Let $L$ be a relational language, let $\calC$ be a countably infinite $L$-structure and $\calA$ a finite $L$-structure such that $\calA \le \calC$.
  Then $T(\calA, \calC) = |\Aut(\calA)| \cdot \tilde T(\calA, \calC)$.
\end{THM}

\paragraph{Chains.} A \emph{chain} is a pair $(A, \Boxed<)$ where $<$ is a strict linear order on $A$.
As usual, $\NN = \{1, 2, 3, \ldots \}$ is the chain of all the positive integers with the usual ordering,
$\omega = \{0, 1, 2, \ldots \}$ is the chain of all the non-negative integers with the usual ordering,
$\ZZ = \{\ldots, -2, -1, 0, 1, 2, \ldots\}$ is the chain of all the integers with the usual ordering,
and $\QQ$ is the chain of all the rationals with the usual ordering.
Every integer $n \in \NN$ can be thought as a finite chain $0 < 1 < \ldots < n-1$.

Let $(A, \Boxed<)$ be a chain and assume that for each $a \in A$ we have a chain $(B_a, \Boxed{<_a})$.
Then the \emph{(indexed) sum of chains} $\sum_{a \in A} B_a$ is the chain on
$\bigcup_{a \in A} (\{a\} \times B_a)$ where the linear order $\prec$ is defined \emph{lexicographically}:
$(a, b) \prec (a', b')$ iff $a < a'$, or $a = a'$ and $b \mathrel{<_a} b'$. Multiplying a chain $B$ by a chain $A$
consists of replacing each element of $A$ by a copy of $B$: $B \cdot A = \sum_{a \in A} B$.
We also say that $B \cdot A$ is the \emph{product} of $B$ and $A$.
Instead of $\sum_{i \in n} B_i$ we shall write $B_0 + B_1 + \ldots + B_{n-1}$.

The class $\LO$ of all countable chains (linear orders) can be preordered by the embeddability relation in a usual way:
write $\calA \preccurlyeq \calB$ if there is an embedding $\calA \hookrightarrow \calB$.
\Fraisse's Conjecture (now a theorem) expresses a deep structural property of the class $\LO$:

\begin{THM}[\Fraisse's Conjecture \cite{Fraisse-conj}]
  $\LO$ is well-quasi-ordered by embeddability.
\end{THM}

In other words, there are no infinite descending chains and no infinite antichains with respect to~$\preccurlyeq$ in $\LO$.
Some twenty years after the publication of \cite{Fraisse-conj} Laver proved \Fraisse's Conjecture
in~\cite{laver-fraisse-conj} by showing a stronger statement:

\begin{THM}[Laver's Theorem \cite{laver-fraisse-conj}] 
  $LO$ is better-quasi-ordered (and hence, well-quasi-ordered) by embeddability.
\end{THM}

We shall also need another deep structural property of countable chains.
A chain $\calA$ is \emph{scattered} if $\QQ \not\hookrightarrow \calA$; otherwise it is \emph{non-scattered}.
In 1908 Hausdorff published a structural characterization of scattered chains~\cite{hausdorff-scat}, which was
rediscovered by Erd\H os and Hajnal in their~1962 paper~\cite{erdos-hajnal-scat}.
Define a sequence $\calH_\alpha$ of sets of chains indexed by ordinals as follows:
\begin{itemize}
\item
  $\calH_0 = \{0, 1\}$ -- the empty chain $\0$ and the 1-element chain 1;
\item
  for an ordinal $\alpha > 0$ let
  $
    \calH_\alpha = \{\sum_{i \in \ZZ} \calS_i : \calS_i \in \UNION_{\beta < \alpha} \calH_\beta \text{ for all } i \in \ZZ \}
  $.
\end{itemize}
Hausdorff then shows in~\cite{hausdorff-scat} that
  for each ordinal $\alpha$ the elements of $\calH_\alpha$ are countable scattered chains; and
  for every countable scattered chain $\calS$ there is an ordinal $\alpha$ such that $\calS$ is order-isomorphic
  to some chain in $\calH_\alpha$.
The least ordinal $\alpha$ such that $\calH_\alpha$ contains a chain order-isomorphic to a
countable scattered chain $\calS$ is referred to as the \emph{Hausdorff rank of $\calS$} and denoted by $r_H(\calS)$.
A countable scattered chain $\calS$ has \emph{finite Hausdorff rank} if $r_H(\calS) < \omega$;
otherwise it has \emph{infinite Hausdorff rank}.

For any chain $\calC$ there is, up to isomorphism, only one $n$-element substructure,
so it is convenient to consider the \emph{big Ramsey spectrum of $\calC$}:
$$
  \spec(\calC) = (T(1, \calC), T(2, \calC), T(3, \calC), \ldots, T(n, \calC), \ldots) \in (\NN \union \{\infty\})^\NN,
$$
where $n$ is the prototypical $n$-element chain $0 < 1 < \ldots < n-1$. We then say that $\calC$
\emph{has finite big Ramsey spectrum}, or simply that \emph{$\spec(\calC)$ is finite}, if $T(n, \calC) < \infty$ for all $n \ge 1$.

\begin{THM}\label{fbrd-finale.thm.MAIN}\cite{masul-fbrd-finale}
  Let $\calC$ be a countable chain. Then $\spec(\calC)$ is finite if and only if $\calC$ is non-scattered, or
  $\calC$ is a scattered chain of finite Hausdorff rank.
\end{THM}

\section{Monomorphic structures}
\label{mnmrf.sec.mnmrf}

An infinite relational structure $\calA$ is \emph{monomorphic}~\cite{Fraisse-ThRel} if, for each $n \in \NN$,
all the $n$-element substructures of $\calA$ are isomorphic.
For any monomorphic structure $\calS$ there is, up to isomorphism, only one $n$-element substructure,
so it is convenient to consider the \emph{big Ramsey spectrum of $\calS$}:
$$
  \spec(\calS) = (T(\calA_1, \calS), T(\calA_2, \calS), T(\calA_3, \calS), \ldots) \in (\NN \union \{\infty\})^\NN,
$$
where $\calA_n$ is the unique $n$-element substructure of $\calS$, up to isomorphism. We then say that $\calS$
\emph{has finite big Ramsey spectrum}, or simply that \emph{$\spec(\calS)$ is finite}, if $T(\calA_n, \calS) < \infty$ for all $n \ge 1$.

As demonstrated by \Fraisse\ in~\cite{Fraisse-ThRel}, and then generalized by Pouzet in~\cite{pouzet-mnmf=chainable},
monomorphic first-order structures are closely related to chains.
Let $L = \{ R_i : i \in I \}$ and $M = \{ S_j : j \in J \}$ be first-order relational languages. An $M$-structure
$\calA = (A, S_j^{\calA})_{j \in J}$ is a \emph{reduct} of an $L$-structure $\calA^*  = (A, R_i^{\calA^*})_{i \in I}$
if there exists a set $\Phi = \{ \phi_j : j \in J \}$ of $L$-formulas such that
for each $j \in J$:
$$
  \calA \models S_j[\overline a] \text{ if and only if } \calA^* \models \phi_j[\overline a],
$$
(where $\overline a$ denotes a tuple of elements of the appropriate length).
We then say that $\calA$ is \emph{definable in $\calA^*$ by $\Phi$}, and that it is
\emph{quantifier-free definable in $\calA^*$} if there is a set of quantifier-free formulas $\Phi$
such that $\calA$ is definable in $\calA^*$ by $\Phi$.

A relational structure $\calA = (A, L^\calA)$ is \emph{chainable}~\cite{Fraisse-ThRel}
if there exists a linear order $<$ on $A$ such that $\calA$ is quantifier-free definable in~$(A, \Boxed<)$.
We then say that the linear order~$<$ \emph{chains}~$\calA$.
The following theorem was  proved by
\Fraisse\ for finite relational languages~\cite{Fraisse-ThRel} and for arbitrary relational languages by Pouzet~\cite{pouzet-mnmf=chainable}.

\begin{THM} \cite{Fraisse-ThRel,pouzet-mnmf=chainable}
  An infinite relational structure is monomorphic if and only if it is chainable.
\end{THM}

\begin{THM} \cite{masul-fbrd-finale}
  Let $L$ be a finite relational language and let let $\calS = (S, L^S)$ be a countable monomorphic structure.
  If $\calS$ is chainable by a linear order $<$ on $S$ such that $\spec(S, \Boxed<)$ is finite
  then $\spec(\calS)$ is finite.
\end{THM}

\begin{REM}\label{mnmrf.rem.2}
Let $L$ be a relational language and $\calS = (S, L^S)$ a countable monomorphic $L$-structure.
In the proof of the main result of this section, Theorem~\ref{mnmrf.thm.mnmrf},
we shall focus on structural big Ramsey degrees $\tilde T(\calA, \calS)$ rather than embedding big Ramsey degrees
$T(\calA, \calS)$. Namely, if $\calA_n$ is the (unique up to isomorphism) $n$-element substructure of $\calS$, then
the sets
$$
  \binom \calS {\calA_n} = \{\calA \le \calS : \calA \cong \calA_n\}
  \text{\quad and \quad}
  \binom S n = \{ A \subseteq S : |A| = n\}
$$
are in an obvious bijective correspondence. Therefore, the following is a convenient reformulation of the notion of structural
big Ramsey degrees for monomorphic structures:
\begin{quote}
    Let $n, t \in \NN$. Then $\tilde T(\calA_n, \calS) \le t$ if for every $k \in \NN$
    and every coloring $\chi : \binom Sn \to k$ there is a substructure $\calS' \le \calS$ such that $\calS' \cong \calS$ and
    $\Big|\big\{\chi(A) : A \in \binom{S'}{n}\big\}\Big| \le t$.
\end{quote}
\end{REM}

\begin{THM}\label{mnmrf.thm.mnmrf}
  Let $L$ be a relational language and let let $\calS = (S, L^S)$ be a countable monomorphic structure.
  For each $n \in \NN$ let $\calA_n$ denote the (up to isomorphism) unique $n$-element substructure of $\calS$.
  Let $\sqsubset$ be a minimal chain (up to bi-embeddability) which chains $\calS$ and let $S_\sqsubset = (S, \Boxed \sqsubset)$.
  Then
  $$
    \tilde T(n, S_\sqsubset) = \tilde T(\calA_n, \calS).
  $$
  Consequently, $T(\calA_n, \calS) = |\Aut(\calA_n)| \cdot T(n, S_\sqsubset)$.
\end{THM}
\begin{proof}
  Let us first establish the equality of structural big Ramsey degrees.

  $(\ge)$
  Let $t = \tilde T(n, S_\sqsubset) \in \NN$. Take any coloring $\chi : \binom Sn \to k$, $k \in \NN$,
  of all the $n$-element substructures of $\calS$. Note that $\chi$ can also be thought of as a coloring
  of all the $n$-element subchains of $S_\sqsubset$ (Remark~\ref{mnmrf.rem.2}). Since $t = \tilde T(n, S_\sqsubset)$,
  there is a subchain $S'_\sqsubset \le S_\sqsubset$ such that $S'_\sqsubset \cong S_\sqsubset$ and
  \begin{equation}\label{mnmrf.eq.0-0}
    \textstyle
    \Big|\big\{\chi(A) : A \in \binom{S'}{n}\big\}\Big| \le t.
  \end{equation}
  Let $f : S_\sqsubset \to S'_\sqsubset$ be an isomorphism and let
  $\calS'$ be the substructure of $\calS$ induced by $S'$. Clearly, $f$ is an isomorphism $\calS \to \calS'$
  because $\calS$ is quantifier-free definable in $S_\sqsubset$, and $\calS'$ is quantifier-free definable in $S'_\sqsubset$.
  Therefore, $\calS'$ is an isomorphic copy of $\calS$ satisfying~\eqref{mnmrf.eq.0-0}.

  $(\le)$
  Let $t = \tilde T(\calA_n, \calS) \in \NN$. Take any coloring $\chi : \binom Sn \to k$, $k \in \NN$,
  of all the $n$-element subchains of $S_\sqsubset$. Note that $\chi$ can also be thought of as a coloring of 
  of all the $n$-element substructures of $\calS$ (Remark~\ref{mnmrf.rem.2}). Since $t = \tilde T(\calA_n, \calS)$, there is a substructure
  $\calS' \le \calS$ such that $\calS' \cong \calS$ and
  \begin{equation}\label{mnmrf.eq.0}
    \textstyle
    \Big|\big\{\chi(A) : A \in \binom{S'}{n}\big\}\Big| \le t.
  \end{equation}
  Let $\calS' = (S', L^{S'})$ and let $f : \calS \hookrightarrow \calS$ be an embedding such that $\im(f) = S'$.
  Define $<_f$ on $S$ as follows:
  \begin{equation}\label{mnmrf.eq.1}
    a \mathrel{<_f} b \text{ if and only if } f(a) \sqsubset f(b).
  \end{equation}

  \medskip

  Claim. $<_f$ chains $\calS$.

  Proof. Let $R \in L$ be a relational symbol of arity $h$. Since $\sqsubset$ chains $\calS$ there is a quantifier-free formula
  $\phi(x_1, \ldots, x_h)$ in the language $\{\sqsubset\}$ such that for every $b_1, \ldots, b_h \in S$:
  $$
    \calS \models R[b_1, \ldots, b_h] \text{ iff } S_\sqsubset \models \phi[b_1, \ldots, b_h].
  $$
  Then
  \begin{align*}
    \calS \models R[a_1, \ldots, a_h]
    &\text{ iff } \calS \models R[f(a_1), \ldots, f(a_h)] && [\text{$f$ is an embedding}]\\
    &\text{ iff } S_\sqsubset \models \phi[f(a_1), \ldots, f(a_h)] && [\sqsubset \text{ chains } \calS]\\
    &\text{ iff } S_{<_f} \models \phi[a_1, \ldots, a_h] && [\text{induction and \eqref{mnmrf.eq.1}}]
  \end{align*}
  Therefore, $<_f$ chains $\calS$ using the same quantifier-free formulas. This proves the Claim.

  \medskip

  So, $<_f$ chains $\calS$ and $f : S_{<_f} \hookrightarrow S_\sqsubset$ is an embedding of chains.
  Since $\sqsubset$ is a minimal chain (up to bi-embeddability) which chains $\calS$, it follows that 
  $S_\sqsubset$ and $S_{<_f}$ are bi-embeddable, so there is an embedding
  $g : S_\sqsubset \hookrightarrow S_{<_f}$. Note that
  $$
    S_\sqsubset \quad \overset g \hookrightarrow \quad S_{<_f} \quad \overset f \hookrightarrow \quad S_{\sqsubset}
  $$
  is an embedding. Hence, $S'' = \im(f \circ g)$ induces a subchain of $S_{\sqsubset}$ which is
  isomorphic to $S_{\sqsubset}$. On the other hand, $S'' \subseteq \im(f) = S'$. From \eqref{mnmrf.eq.0} it now follows that
  \begin{equation*}
    \textstyle
    \Big|\big\{\chi(A) : A \in \binom{S''}{n}\big\}\Big| \le t.
  \end{equation*}
  This completes the proof of the first part of the statement.

  \bigskip

  As for the second part of the statement note that $\tilde T(n, S_\sqsubset) = T(n, S_\sqsubset)$ because chains are rigid structures,
  while $T(\calA_n, \calS) = |\Aut(\calA_n)| \cdot \tilde T(\calA_n, \calS)$ holds in general (see~\cite{zucker-brd}).
\end{proof}

\begin{COR}
  Let $L$ be a relational language and let $\calS = (S, L^S)$ be a countable monomorphic structure.
  Then $\spec(\calS)$ is finite if and only if there exists a linear order $<$ on $S$
  which chains $\calS$ and with the property that $\spec(S_<)$ is finite.
\end{COR}

\section{Structures admitting a finite monomorphic decomposition}
\label{mnmrf.sec.fmd}

In this section we extend the above results to first-order structures admitting a finite monomorphic decomposition.
Recall that structures admitting a finite monomorphic decomposition were introduced by Pouzet and Thi\'ery:

\begin{DEF} \cite{pouzet-thiery-I}
  A \emph{monomorphic decomposition} of a relational structure $\calS = (S, L^S)$ 
  is a partition $\{E_i : i \in I\}$ of $S$ satisfying the following:
  for all finite $X, Y \subseteq S$, if $|X \cap E_i| = |Y \cap E_i|$ for all $i \in I$ then $\calS[X] \cong \calS[Y]$.
  A relational structure $\calS$ \emph{admits a finite monomorphic decomposition}
  if there exists a monomorphic decomposition $\{E_i : i \in I\}$ of $\calS$ with $I$ finite.
\end{DEF}

Note that in a monomorphic decomposition each $\calS[E_i]$ is a monomorphic structure, $i \in I$.

\begin{PROP}\cite[Proposition 1.6]{pouzet-thiery-I}\label{mnmrf.prop.min}
  Every relational structure $\calS$ has a monomorphic decomposition $\{B_i : i \in I\}$ such that every other
  monomorphic decomposition of $\calS$ is a refinement of $\{B_i : i \in I\}$.
  This monomorphic decomposition of $\calS$ will be referred to as \emph{minimal}.
\end{PROP}

\begin{LEM}\label{mnmrf.lem.MAIN}
  Let $\calS = (S, L^S)$  be a relational structure that admits a finite monomorphic decomposition, and let $\{B_1, B_2, \ldots, B_q\}$
  be the minimal monomorphic decomposition of $\calS$.
  
  $(a)$ Let $\calS' = (S', L^{S'})$ be a substructure of $\calS$ which is isomorphic to~$\calS$.
  Then $\{S' \cap B_1, S' \cap B_2, \ldots, S' \cap B_q\}$ is a minimal monomorphic decomposition of~$\calS'$.

  $(b)$ For every embedding $f : \calS \hookrightarrow \calS$ there is a permutation $\sigma$ of $\{1, 2, \ldots, q\}$ such that
  $f(B_i) \subseteq B_{\sigma(i)}$ for all $1 \le i \le q$.
\end{LEM}
\begin{proof}
  $(a)$ Let us first show that the non-empty sets among $\{S' \cap B_1, S' \cap B_2, \ldots, S' \cap B_q\}$ 
  form a monomorphic decomposition of $\calS'$. Without loss of generality we can assume that there is a $p \in \{1, 2, \ldots, q\}$
  such that $S' \cap B_i = \0$ for all $i < p$ and $S' \cap B_j \ne \0$ for $j \ge p$.
  Take any finite $X, Y \subseteq S'$ such that $|X \cap (S' \cap B_j)| = |Y \cap (S' \cap B_j)|$ for all $j \ge p$.
  Then $|X \cap B_j| = |Y \cap B_j|$ for all $j \ge p$, while $|X \cap B_i| = 0 = |Y \cap B_i|$ for $i < p$
  because $X, Y \subseteq S'$ and $S' \cap B_i = \0$.
  Therefore $\calS[X] \cong \calS[Y]$ because $\{B_1, B_2, \ldots, B_q\}$ is a monomorphic decomposition of $\calS$. Hence,
  $\calS'[X] = \calS[X] \cong \calS[Y] = \calS'[Y]$.
  
  With this at hand, it is now easy to show that $\{S' \cap B_1, S' \cap B_2, \ldots, S' \cap B_q\}$ is a partition of $S'$:
  if one of $S' \cap B_i$ is empty, the nonempty blocks among $\{S' \cap B_1, S' \cap B_2, \ldots, S' \cap B_q\}$
  would constitute a monomorphic decomposition of $\calS'$ with less than $q$ blocks, so $\calS$,
  being isomorphic to $\calS'$, would also have a monomorphic decomposition with less than $q$ blocks -- a contradiction.

  Now that we know that $\{S' \cap B_1, S' \cap B_2, \ldots, S' \cap B_q\}$ is a partition of $S'$, the argument we started the proof with
  ensures that this is a minimal monomorphic decomposition of $\calS'$.

  \medskip

  $(b)$ Let $f : \calS \hookrightarrow \calS$ be an embedding and let $\calS' = \calS[\im(f)]$ be the image of $\calS$ under $f$.
  Clearly, $\calS' \cong \calS$. Just as a notational convenience let $\calS' = (S', L^{S'})$.
  Then $\{S' \cap B_1, S' \cap B_2, \ldots, S' \cap B_q\}$ is a
  monomorphic decomposition of~$\calS'$ by~$(a)$, so $\{f^{-1}(S' \cap B_1), f^{-1}(S' \cap B_2), \ldots, f^{-1}(S' \cap B_q)\}$
  is a monomorphic decomposition of $\calS$ because the codomain restriction $\restr f {\calS'} : \calS \to \calS'$ is an isomorphism.
  Since $\{B_1, B_2, \ldots, B_q\}$ is the minimal monomorphic decomposition of $\calS$ it follows that
  $\{f^{-1}(S' \cap B_1), f^{-1}(S' \cap B_2), \ldots, f^{-1}(S' \cap B_q)\}$ is finer than $\{B_1, B_2, \ldots, B_q\}$.
  Note that the two partitions of $S$ have the same number of blocks. Therefore,
  $$
    \{B_1, B_2, \ldots, B_q\} = \{f^{-1}(S' \cap B_1), f^{-1}(S' \cap B_2), \ldots, f^{-1}(S' \cap B_q)\}.
  $$
  Because of that, for every $i$ there is a $j$ such that $B_i = f^{-1}(S' \cap B_j)$, or, equivalently,
  $f(B_i) = S' \cap B_j \subseteq B_j$. Moreover, given $i$ this $j$ is unique because
  $\{B_1, B_2, \ldots, B_q\}$ is a partition of $S$. So, define a mapping $\sigma : \{1, 2, \ldots, q\} \to \{1, 2, \ldots, q\}$
  so that $\sigma(i) = j$ if and only if $f(B_i) \subseteq B_j$. To show that $\sigma$ is a bijection it suffices to note that $\sigma$ is surjective,
  which follows from $(a)$.
\end{proof}

\begin{THM}
  Let $\calS$ be a relational structure admitting a finite monomorphic decomposition,
  let $\{B_1, B_2, \ldots, B_q\}$ be the minimal monomorphic decomposition of $\calS$,
  and let $\calB_i = \calS[B_i]$, $1 \le i \le q$.
  If $\calS$ has finite big Ramsey degrees then so does every $\calB_i$, $1 \le i \le q$.
\end{THM}
\begin{proof}
  Suppose, to the contrary, that some $\calB_j$ does not have finite big Ramsey degrees. 
  Then there is a finite relational structure $\calA$ such that $T(\calA, \calB_j) = \infty$.

  Let us show that $T(\calA, \calS) = \infty$. Take any $t \in \NN$. Since $T(\calA, \calB_j) = \infty$
  there is a $k \in \NN$ and a coloring $\chi_j : \Emb(\calA, \calB_j) \to k$
  such that for every embedding $w : \calB_j \hookrightarrow \calB_j$ we have that $|\chi_j(w \circ \Emb(\calA, \calB_j))| \ge t$.
  Let $\iota_j : \calB_j \hookrightarrow \calS$ be the canonical embedding $\iota_j(x) = x$ and
  define $\chi : \Emb(\calA, \calS) \to k$ as follows:
  for a $g \in \Emb(\calA, \calS)$, if $g(A) \subseteq B_j$ then $g = \iota_j \circ f$ for some
  $f \in \Emb(\calA, \calB_j)$ and we put $\chi(g) = \chi_j(f)$; otherwise we let $\chi(g) = 0$.

  Take any $w : \calS \hookrightarrow \calS$. By Lemma~\ref{mnmrf.lem.MAIN}~$(b)$, there is a permutation $\sigma$ such that
  $w(B_i) \subseteq B_{\sigma(i)}$ for all $1 \le i \le q$. Therefore, there is a $k \ge 1$ (the length of the cycle that contains $j$ in the
  cyclic representation of $\sigma$) such that $w^k(B_j) \subseteq B_j$.
  Let $w^* : \calB_j \hookrightarrow \calB_j$ be the restriction of $w^k$, that is, an embedding defined so that $w^*(x) = w^k(x)$ for all $x \in B_j$.
  Note that:
  \begin{equation}\label{mnmrf.eq.S2}
    \iota_j \circ w^* = w^k \circ \iota_j
  \end{equation}
  Let us show that $\chi_j(w^* \circ \Emb(\calA, \calB_j)) \subseteq \chi(w \circ \Emb(\calA, \calS))$:
  \begin{align*}
    \chi_j(w^* \circ \Emb(\calA, \calB_j))
    &= \chi(\iota_j \circ w^* \circ \Emb(\calA, \calB_j)) && [\text{definition of $\chi$}]\\
    &= \chi(w^k \circ \iota_j \circ \Emb(\calA, \calB_j)) && [\text{\eqref{mnmrf.eq.S2}}]\\
    &\subseteq \chi(w^k \circ \Emb(\calA, \calS))\\
    &\subseteq \chi(w \circ \Emb(\calA, \calS))           && [k \ge 1]
  \end{align*}
  The choice of $\chi_j$ ensures that $|\chi_j(w^* \circ \Emb(\calA, \calB_j))| \ge t$, hence
  $$
   |\chi(w \circ \Emb(\calA, \calS))| \ge t.
  $$
  This concludes the proof.
\end{proof}

Let us now list a few more notions and results from \cite{pouzet-thiery-I} that will be needed for the proof of the main result of this section.

\begin{DEF} \cite{pouzet-thiery-I}
    We say that $f$ is a \emph{local automorphism} of $\calS$ if $f$ is an isomorphism
    between two substructures of $\calS$ (finite or infinite).
\end{DEF}

\begin{THM}\cite[Theorem 1.8]{pouzet-thiery-I}\label{mnmrf.thm.local-iso}
  A relational structure $\calS = (S, L^S)$ admits a finite monomorphic decomposition if and only if there exists a linear order
  $<$ on $S$ and a finite partition $\{E_1, \ldots, E_r\}$ of $S$ into intervals of $(S, \Boxed<)$ such that every local isomorphism of
  $(S, \Boxed<)$ which preserves each interval is a local isomorphism of $\calS$.
\end{THM}

We shall also need a special case of the rather technical \cite[Theorem 2.25]{pouzet-thiery-I} (for our purposes only the items $(i)$ and $(ii)$
of \cite[Theorem 2.25]{pouzet-thiery-I} with $F = \0$ suffice:)

\begin{THM}\label{mnmrf.thm.pouzet-thiery-thm2.25} (cf.\ \cite[Theorem 2.25]{pouzet-thiery-I})
  Let $\calE = (E, (\rho_i)_{i \in I})$ be a relational structure. Let us consider the following properties:
  \begin{enumerate}[$(i)$]
  \item $\calE$ is chainable;
  \item $\calE$ is monomorphic.
  \end{enumerate}
  Then $(i) \Rightarrow (ii)$. If $E$ is infinite then $(ii) \Rightarrow (i)$.
\end{THM}

Assume that $\calS = (S, L^S)$ admits a finite monomorphic decomposition. It is actually easy to construct a linear
order on $S$ whose existence Theorem~\ref{mnmrf.thm.local-iso} postulates. Take any finite monomorphic decomposition of $\calS$
and refine its finite blocks to singletons to get a monomorphic decomposition $\{E_1, \ldots, E_r\}$.
The infinite blocks in this decomposition are chainable (Theorem~\ref{mnmrf.thm.pouzet-thiery-thm2.25}), so 
on each $E_i$ there is a linear order $<_i$ such that $<_i$ chains $\calS[E_i]$, $1 \le i \le r$.
Then a lexicographical sum, in any order, of the chains $(E_i, \Boxed{<_i})$ yields a linear order on $S$ for
which the $E_i$'s are intervals and every local isomorphism preserving each of the intervals $E_i$ is a local automorphism of~$\calS$.

\begin{THM}\label{mnmrf.thm.MAIN-MNMRF-2}
  Let $\calS$ be a relational structure admitting a finite monomorphic decomposition,
  let $\{B_1, B_2, \ldots, B_q\}$ be the minimal monomorphic decomposition of $\calS$ and let $\calB_i = \calS[B_i]$, $1 \le i \le q$.
  If every $\calB_i$, $1 \le i \le q$, has finite big Ramsey degrees then so does $\calS$.
\end{THM}
\begin{proof}
  Let $\{E_1, \ldots, E_r\}$ be a finite monomorphic decomposition of $\calS$ obtained from $\{B_1, B_2, \ldots, B_q\}$
  by refining its finite blocks to singletons, and preserving the infinite blocks.
  For convenience, assume that $E_1$, \ldots, $E_t$ are the infinite blocks in the new decomposition, and that $E_{t+1}$, \ldots, $E_r$ are the singletons.
  According to the remark above, the infinite blocks in this decomposition are chainable (Theorem~\ref{mnmrf.thm.pouzet-thiery-thm2.25}), so 
  on each $E_i$ there is a linear order $<_i$ such that $<_i$ chains $\calE_i = \calS[E_i]$, $1 \le i \le t$.
  Without loss of generality we can take $<_i$ to be a minimal linear order (up to bi-embeddability) which chains $\calE_i$.
  Then, according to Theorem~\ref{mnmrf.thm.mnmrf} the fact that each $\calE_i$ has finite big Ramsey degrees implies that each chain
  $(E_i, \Boxed{<_i})$ has finite big Ramsey degrees, $1 \le i \le t$.

  Let $(S, \Boxed\sqsubset)$ be the lexicographical sum of the chains $(E_i, \Boxed{<_i})$, $1 \le i \le r$, where the ordering on the
  singletons is the trivial one:
  $$
    (S, \Boxed\sqsubset) = (E_1, \Boxed{<_1}) \oplus \ldots \oplus (E_r, \Boxed{<_r}).
  $$
  This is a linear order on $S$ in which each $E_i$ is an interval and with the property that every local isomorphism preserving each
  of the intervals $E_i$ is a local automorphism of~$\calS$ (Theorem~\ref{mnmrf.thm.local-iso}).

  For a finite $\calA \le \calS$ and non-negative integers $n_1$, $n_2$, \ldots, $n_r$ let
  $$\textstyle
    {\binom \calS \calA}^{E_1, \ldots, E_r}_{n_1, \ldots, n_r} = \big\{(A', L^{A'}) \in \binom \calS \calA : |A' \cap E_i| = n_i, 1 \le i \le r\big\}.
  $$

  \bigskip

  Claim 1. For every $\calS' = (S', L^{S'}) \in \binom \calS \calS$ and every $1 \le i \le r$ we have that $S' \cap E_i \ne \0$.
  Moreover, $\{S' \cap B_1, \ldots, S' \cap B_q\}$ is a minimal monomorphic decomposition of $\calS'$ and
  $\{S' \cap E_1, \ldots, S' \cap E_r\}$ is a finite monomorphic decomposition of $\calS'$ obtained from $\{S' \cap B_1, \ldots, S' \cap B_q\}$
  by refining its finite blocks to singletons, and preserving the infinite blocks.

  Proof. Since $\calS'$ is an isomorphic copy of $\calS$ there is an embedding $f : \calS \hookrightarrow \calS$
  such that $\im(f) = S'$. Then by Lemma~\ref{mnmrf.lem.MAIN}~$(b)$ there is a permutation $\sigma : \{1, 2, \ldots, q\} \to \{1, 2, \ldots, q\}$
  of the blocks $\{B_1, \ldots, B_q\}$ of the minimal monomorphic decomposition such that
  $f(B_i) \subseteq B_{\sigma(i)}$ for all $1 \le i \le q$. This immediately implies that
  $S'$ intersects every infinite $B_i$, and that $f$ permutes the points that belong to finite blocks.
  Therefore, $S'$ intersects every infinite $E_j$ (because the two decompositions have identical infinite blocks), and $S'$
  contains all the points that belong to finite blocks. 
  
  The second part of the claim follows directly from Lemma~\ref{mnmrf.lem.MAIN}~$(a)$ and the first part of the claim. This proves Claim~1.
  
  \bigskip

  Claim 2. For every finite $\calA \le \calS$ and every choice of non-negative integers $n_1$, $n_2$, \ldots, $n_r$ 
  there exists a positive integer $N$ such that for every $k \ge 1$ and every coloring
  $\chi : {\binom \calS \calA}^{E_1, \ldots, E_r}_{n_1, \ldots, n_r} \to k$ there is a substructure $\calS' \in \binom \calS\calS$
  satisfying
  $$\textstyle
    \Big|\chi\Big({\binom {\calS'} \calA}^{E_1, \ldots, E_r}_{n_1, \ldots, n_r}\Big)\Big| \le N.
  $$

  Proof. Take any $(A', L^{A'}) \in {\binom \calS \calA}^{E_1, \ldots, E_r}_{n_1, \ldots, n_r}$ and let
  $A'_i = A' \cap E_i$, $1 \le i \le r$. For $1 \le i \le r$ we have that $A'_i$ is a subset of the monomorphic structure $\calE_i$
  which is chained by the linear order $<_i$. Therefore, $A'_i$ uniquely determines the embedding
  $f_{A'_i} : n_i \hookrightarrow (E_i, \Boxed{<_i})$ defined so that $\im(f_{A'_i}) = A'_i$.
  Consequently, every $(A', L^{A'}) \in {\binom \calS \calA}^{E_1, \ldots, E_r}_{n_1, \ldots, n_r}$ uniquely determines a tuple of
  embeddings
  $$
    (f_{A'_1}, f_{A'_2}, \ldots, f_{A'_r}) \text{ where } f_{A'_i} : n_i \hookrightarrow (E_i, \Boxed{<_i}).
  $$
  Let $N$ be a positive integer provided by Corollary~\ref{mnmrf.cor.CHAIN-PROD} for the non-negative integers
  $n_1$, \ldots, $n_r$ and chains $(E_1, \Boxed{<_1})$, \ldots, $(E_r, \Boxed{<_r})$.

  Take any coloring $\chi : {\binom \calS \calA}^{E_1, \ldots, E_r}_{n_1, \ldots, n_r} \to k$ and define
  $$
    \gamma : \Emb(n_1, (E_1, \Boxed{<_1})) \times \ldots \times \Emb(n_r, (E_r, \Boxed{<_r})) \to k
  $$
  by $\gamma(f_{A'_1}, f_{A'_2}, \ldots, f_{A'_r}) = \chi(A', L^{A'})$ and
  $\gamma(g_1, \ldots, g_r) = 0$ if $(g_1, \ldots, g_r) \ne (f_{A'_1}, f_{A'_2}, \ldots, f_{A'_r})$
  for all $(A', L^{A'}) \in {\binom \calS \calA}^{E_1, \ldots, E_r}_{n_1, \ldots, n_r}$.
  By Corollary~\ref{mnmrf.cor.CHAIN-PROD} there are embeddings $w_i : (E_i, \Boxed{<_i}) \hookrightarrow (E_i, \Boxed{<_i})$
  $1 \le i \le r$ such that
  $$
    \big|\gamma\big((w_1 \circ \Emb(n_1, (E_1, \Boxed{<_1}))) \times \ldots \times (w_r \circ \Emb(n_r, (E_r, \Boxed{<_r})))\big)\big| \le N.
  $$
  Let
  $
    w^* = w_1 \oplus \ldots \oplus w_r
  $
  be the lexicographic sum of the embeddings $w_1$, \ldots, $w_r$. Clearly, $w^*$ is an embedding $(S, \Boxed\sqsubset) \hookrightarrow (S, \Boxed\sqsubset)$,
  and hence a local isomorphism of $(S, \Boxed\sqsubset)$ which preserves the intervals $E_i$.
  By Theorem~\ref{mnmrf.thm.local-iso} we then know that $w^*$ is a local automorphism of~$\calS$. Moreover,
  $w^*$ is an embedding $\calS \hookrightarrow \calS$. Let $S' = \im(w^*)$ and $\calS' = \calS[S']$.
  Clearly, $\calS' \in \binom \calS \calS$. Let us show that
  $$\textstyle
    \Big|\chi\Big({\binom {\calS'} \calA}^{E_1, \ldots, E_r}_{n_1, \ldots, n_r}\Big)\Big| \le N
  $$
  by showing that
  $$\textstyle
    \chi\Big({\binom {\calS'} \calA}^{E_1, \ldots, E_r}_{n_1, \ldots, n_r}\Big) \subseteq \gamma\big((w_1 \circ \Emb(n_1, (E_1, \Boxed{<_1}))) \times \ldots \times (w_r \circ \Emb(n_r, (E_r, \Boxed{<_r})))\big).
  $$
  Take any $\calA' = (A', L^{A'}) \in {\binom {\calS'} \calA}^{E_1, \ldots, E_r}_{n_1, \ldots, n_r}$, let $A'_i = A' \cap E_i$
  and let $A''_i = w_i^{-1}(A'_i)$, $1 \le i \le r$. The mapping $h : A'_1 \cup \ldots \cup A'_r \to A''_1 \cup \ldots \cup A''_r$
  given by $h(x) = w_i^{-1}(x)$ for $x \in A'_i$ is clearly a local isomorphism of $(S, \Boxed\sqsubset)$ which
  preserves the intervals $E_i$, $1 \le i \le r$. By Theorem~\ref{mnmrf.thm.local-iso} we then know that $h$ is a local automorphism of~$\calS$.
  Therefore, if we let $\calA'' = \calS[A''_1 \cup \ldots \cup A''_r]$, we have that $\calA'' \cong \calA'$,
  $\calA'' \in {\binom {\calS} \calA}^{E_1, \ldots, E_r}_{n_1, \ldots, n_r}$ and $f_{A'_i} = w_i \circ f_{A''_i}$, $1 \le i \le r$.
  Therefore, by definition of $\gamma$:
  \begin{align*}
    \chi(\calA') & = \gamma(f_{A'_1}, \ldots, f_{A'_r})\\
                 & = \gamma(w_1 \circ f_{A''_1}, \ldots, w_r \circ f_{A''_r})\\
                 & \in \gamma\big((w_1 \circ \Emb(n_1, (E_1, \Boxed{<_1}))) \times \ldots \times (w_r \circ \Emb(n_r, (E_r, \Boxed{<_r})))\big).
  \end{align*}
  This concludes the proof of Claim~2.

  \bigskip
  
  Moving on to the proof of the theorem, let $\calA = (A, L^{A}) \le \calS$ be a finite substructure of $\calS$.
  Let $\tau_1$, \ldots, $\tau_m$ be the enumeration of all the possible $r$-tuples of non-negative integers
  $(n_1, \ldots, n_r)$ such that $n_1 + \ldots + n_r = |A|$. Then $\binom \calS \calA$ can be partitioned as:
  \begin{equation}\label{mnmrf.eq.disj-union}
    \textstyle
    \binom \calS \calA = \bigcup\limits_{j=1}^m {\binom \calS \calA}_{\tau_j}^{E_1, \ldots, E_r}.
  \end{equation}
  According to Claim~2 for each $\tau_j$, $1 \le j \le m$, there is a positive integer $N_j$ satisfying the conclusion of
  the claim. Let us show that
  $$
    \tilde T(\calA, \calS) \le N_1 + \ldots + N_m.
  $$
  Take any coloring $\chi : \binom \calS \calA \to k$ and let $\chi_1 : {\binom \calS \calA}^{E_1, \ldots, E_r}_{\tau_1} \to k$
  be the restriction of $\chi$. According to Claim~2 there is an $\calS_1 \in \binom \calS \calS$ such that
  $$\textstyle
    \Big|\chi_1\Big({\binom {\calS_1} \calA}^{E_1, \ldots, E_r}_{\tau_1}\Big)\Big| \le N_1.
  $$
  Let $\chi_2 : {\binom {\calS_1} \calA}^{E_1, \ldots, E_r}_{\tau_2} \to k$ be another restriction of $\chi$.
  Claim~1 ensures that Claim~2 applies to this setting as well, so there is an $\calS_2 \in \binom {\calS_1} \calS$ such that
  $$\textstyle
    \Big|\chi_2\Big({\binom {\calS_2} \calA}^{E_1, \ldots, E_r}_{\tau_2}\Big)\Big| \le N_2.
  $$
  And so on. In the final step we get an $\calS_m \in \binom {\calS_{m-1}} \calS$ such that
  $$\textstyle
    \Big|\chi_m\Big({\binom {\calS_m} \calA}^{E_1, \ldots, E_r}_{\tau_m}\Big)\Big| \le N_m.
  $$
  Let us show that $\Big|\chi\Big(\binom{\calS_m}{\calA}\Big)\Big| \le N_1 + \ldots + N_m$.
  Using \eqref{mnmrf.eq.disj-union} applied to $\binom{\calS_m}{\calA}$, the fact that
  ${\binom{\calS_m}{\calA}}^{E_1, \ldots, E_r}_{\tau_j} \subseteq {\binom{\calS_j}{\calA}}^{E_1, \ldots, E_r}_{\tau_j}$
  for all $1 \le j \le m$ and the fact that $\chi_j$ is an appropriate restriction of $\chi$ we get:
  \begin{align*}
    \textstyle \Big|\chi\Big(\binom{\calS_m}{\calA}\Big)\Big|
    &\textstyle = \Big|\chi\Big(\bigcup_{j=1}^m {\binom{\calS_m}{\calA}}^{E_1, \ldots, E_r}_{\tau_j}\Big)\Big|\\
    &\textstyle = \sum_{j=1}^m \Big|\chi\Big({\binom{\calS_m}{\calA}}^{E_1, \ldots, E_r}_{\tau_j}\Big)\Big|\\
    &\textstyle \le \sum_{j=1}^m \Big|\chi\Big({\binom{\calS_j}{\calA}}^{E_1, \ldots, E_r}_{\tau_j}\Big)\Big|\\
    &\textstyle = \sum_{j=1}^m \Big|\chi_j\Big({\binom{\calS_j}{\calA}}^{E_1, \ldots, E_r}_{\tau_j}\Big)\Big| \le \sum_{j=1}^m N_j.
  \end{align*}
  This completes the proof of the theorem.    
\end{proof}

\section{A product Ramsey theorem for chains}
\label{mnmrf.sec.prod-thm}

In this section we prove the product Ramsey theorem for big Ramsey degrees of countable chains (Theorem~\ref{mnmrf.thm.CHAIN-PROD}),
whose mild modification (Corollary~\ref{mnmrf.cor.CHAIN-PROD}) was instrumental in the proof of Theorem~\ref{mnmrf.thm.MAIN-MNMRF-2}.
We find this product Ramsey theorem particularly intriguing because big Ramsey degrees are known to exhibit irregular behavior when it comes
to general product statements. The proof proceeds in several stages, so let us start building the necessary infrastructure.

\subsection{The tools}

As it turns out, the convenient language for our purposes is the language of category theory.
In order to specify a \emph{category} $\CC$ one has to specify
a class of objects $\Ob(\CC)$, a class of morphisms $\hom_\CC(A, B)$ for all $A, B \in \Ob(\CC)$,
the identity morphism $\id_A$ for all $A \in \Ob(\CC)$, and for every triple $A, B, C \in \Ob(\CC)$
the composition of morphisms~$\Boxed\circ : \hom_\CC(B, C) \times \hom_\CC(A, B) \to \hom_\CC(A, C)$
so that $\id_B \circ f = f = f \circ \id_A$ for all $f \in \hom_\CC(A, B)$, and
$(h \circ g) \circ f = h \circ (g \circ f)$ for all $f \in \hom_\CC(A, B)$, $g \in \hom_\CC(B, C)$ and $h \in \hom_\CC(C, D)$.

A category $\CC$ is \emph{locally small} if $\hom_\CC(A, B)$ is a set for all $A, B \in \Ob(\CC)$.
Sets of the form $\hom_\CC(A, B)$ are then referred to as \emph{homsets}. If $\CC$ can be deduced from the context we simply write $\hom(A, B)$.

We say that objects $X, Y \in \Ob(\CC)$ are \emph{hom-equivalent} if $\hom(X, Y) \ne \0$ and $\hom(Y, X) \ne \0$.

For categories $\CC_1$, \ldots, $\CC_n$, $n \in \NN$,
there is the \emph{product category} $\CC_1 \times \ldots \times \CC_n$ whose objects are
tuples $(A_1, \ldots, A_n)$ where $A_i \in \Ob(\CC_i)$, $1 \le i \le n$,
morphisms are tuples $(f_1, \ldots, f_n) : (A_1, \ldots, A_n) \to (B_1, \ldots, B_n)$, where
$f_i \in \hom_{\CC_i}(A_i, B_i)$, $1 \le i \le n$, and the composition of morphisms
is carried out componentwise: $(g_1, \ldots, g_n) \circ (f_1, \ldots, f_n) = (g_1 \circ f_1, \ldots, g_n \circ f_n)$.
Clearly, if $(A_1, \ldots, A_n)$ and $(B_1, \ldots, B_n)$ are objects of $\CC_1 \times \ldots \times \CC_n$
then
\begin{multline*}
  \hom_{\CC_1 \times \ldots \times \CC_n}\big((A_1, \ldots, A_n), (B_1, \ldots, B_n)\big) = \\
  = \hom_{\CC_1}(A_1, B_1) \times \ldots \times \hom_{\CC_n}(A_n, B_n).
\end{multline*}
We write $\CC^n$ for $\CC \times \ldots \times \CC$ ($n$ times).

The notion of big Ramsey degrees we have seen in previous sections translates to the context of category
theory straightforwardly. Let $\CC$ be a locally small category and $A, S \in \Ob(\CC)$. We say that
\emph{$A$ has finite big Ramsey degree in $S$} if there is a $t \in \NN$ such that for every $k \in \NN$ and
every coloring $\chi : \hom_\CC(A, S) \to k$ there is a $w \in \hom_\CC(S, S)$ such that
$|\chi(w \circ \hom_\CC(A, S))| \le t$. The least such $t$ is referred to as the \emph{big Ramsey degree of $A$ in $S$}
and we write $T_\CC(A, S) = t$. If $A$ does not have a finite big Ramsey degree in $S$ we write $T_\CC(A, S) = \infty$.

Our main proof strategy is based on transporting the Ramsey property from one context to another.
In the setting of finite Ramsey theory similar notions have been proposed in~\cite{masul-preadj} and~\cite{solecki-finite-ramsey-category-theory}.

\begin{DEF}\label{mnmrf.def.piggyback}
  Let $\AAA$ and $\BB$ be locally small categories. For $A, X \in \Ob(\AAA)$ and $B, Y \in \Ob(\BB)$
  we write $(A, X)_\AAA \prec (B, Y)_\BB$ to denote that there is an $M \subseteq \hom(B, Y)$ and a set-function
  $\phi : M \to \hom(A, X)$ such that for every $h \in \hom(Y, Y)$ one can find a $g \in \hom(X, X)$ satisfying:
  $$
    g \circ \hom(A, X) \subseteq \phi(M \cap h \circ \hom(B, Y)).
  $$
  (Note that $\circ$ takes precedence over $\cap$, so $M \cap h \circ \hom(B, Y)$ should be understood as $M \cap (h \circ \hom(B, Y))$.)
\end{DEF}

\begin{LEM}\label{mnmrf.lem.piggy=>brd}
    Let $\AAA$ and $\BB$ be locally small categories, $A, X \in \Ob(\CC)$ and $B, Y \in \Ob(\DD)$.
    If $(A, X)_\AAA \prec (B, Y)_\BB$ then $T_\AAA(A, X) \le T_\BB(B, Y)$.
\end{LEM}
\begin{proof}
    Let $T_\BB(B, Y) = t \in \NN$. Since $(A, X)_\AAA \prec (B, Y)_\BB$, there is an $M \subseteq \hom(B, Y)$ and a set-function
    $\phi : M \to \hom(A, X)$ as in Definition~\ref{mnmrf.def.piggyback}.
    Take any coloring $\chi : \hom(A, X) \to k$.
    Define $\gamma : \hom(B, Y) \to k$ as follows: $\gamma(f) = \chi(\phi(f))$ if $f \in M$
    and $\gamma(f) = 0$ otherwise. Since $T_\BB(B, Y) = t$ there is an $h \in \hom(Y, Y)$ such that
    $
      |\gamma(h \circ \hom(B, Y))| \le t
    $.
    By the choice of $M$ and $\phi$ for this $h$ there is a $g \in \hom(X, X)$
    satisfying
    $
      g \circ \hom(A, X) \subseteq \phi(M \cap h \circ \hom(B, Y))
    $.
    Now,
    $$
      \chi(g \circ \hom(A, X))
      \subseteq \chi(\phi(M \cap h \circ \hom(B, Y)))
      \subseteq \gamma(h \circ \hom(B, Y)),
    $$
    whence $|\chi(g \circ \hom(A, X))| \le |\gamma(h \circ \hom(B, Y))| \le t$.
\end{proof}

\begin{LEM}\label{mnmrf.lem.rev-piggyback}
  Let $\AAA$ and $\BB$ be locally small categories, $A, X \in \Ob(\AAA)$ and $B, Y \in \Ob(\BB)$.
  Suppose that there is an injective function $\psi : \hom(A, X) \to \hom(B, Y)$ such that
  for every $h \in \hom(Y, Y)$ one can find a $g \in \hom(X, X)$ satisfying
  $
    \psi(g \circ \hom(A, X)) \subseteq h \circ \hom(B, Y)
  $.
  Then $(A, X)_\AAA \prec (B, Y)_\BB$.
\end{LEM}
\begin{proof}
  Let $M = \im(\psi) \subseteq \hom(B, Y)$ and note that the codomain restriction $\psi_M : \hom(A, X) \to M$
  defined by $\psi_M(f) = f$ is a bijection. Let $\phi = \psi_M^{-1} : M \to \hom(A, X)$.
  Then it is easy to see that $M$ and $\phi$ satisfy the requirements of the Definition~\ref{mnmrf.def.piggyback}.
\end{proof}

\begin{LEM}\label{mnmrf.lem.piggyback-transitivity}
  Let $\AAA$, $\BB$ and $\CC$ be locally small categories, and let $A, X \in \Ob(\AAA)$, $B, Y \in \Ob(\BB)$ and
  $C, Z \in \Ob(\CC)$ be arbitrary objects.
  If $(A, X)_\AAA \prec (B, Y)_\BB$ and $(B, Y)_\BB \prec (C, Z)_\CC$ then $(A, X)_\AAA \prec (C, Z)_\CC$.
\end{LEM}
\begin{proof}
  Since $(A, X)_\AAA \prec (B, Y)_\BB$ there is a set $M_1 \subseteq \hom(B, Y)$ and a function $\phi_1 : M_1 \to \hom(A, X)$ such that
  for every $g \in \hom(Y, Y)$ there is an $f \in \hom(X, X)$ satisfying
  $$
    f \circ \hom(A, X) \subseteq \phi_1(M_1 \cap g \circ \hom(B, Y)).
  $$
  Analogously, $(B, Y)_\BB \prec (C, Z)_\CC$ means that there is a set $M_2 \subseteq \hom(C, Z)$ and a function $\phi_2 : M_2 \to \hom(B, Y)$ such that
  for every $h \in \hom(Z, Z)$ there is a $g \in \hom(Y, Y)$ satisfying
  $$
    g \circ \hom(B, Y) \subseteq \phi_2(M_2 \cap h \circ \hom(C, Z)).
  $$
  To show that $(A, X)_\AAA \prec (C, Z)_\CC$ let $M = \phi_2^{-1}(M_1) \subseteq \hom(C, Z)$ and define $\phi : M \to \hom(A, X)$ by
  $\phi(f) = \phi_1(\phi_2(f))$. Take any $h \in \hom(Z, Z)$. Then there is a $g \in \hom(Y, Y)$ such that
  $g \circ \hom(B, Y) \subseteq \phi_2(M_2 \cap h \circ \hom(C, Z))$.
  For this $g$ there is an $f \in \hom(X, X)$ such that $f \circ \hom(A, X) \subseteq \phi_1(M_1 \cap g \circ \hom(B, Y))$. Now,
  \begin{align*}
    f \circ \hom(A, X)
    &\subseteq \phi_1(M_1 \cap g \circ \hom(B, Y)) \\
    &\subseteq \phi_1(M_1 \cap \phi_2(M_2 \cap h \circ \hom(C, Z))) \\
    &\subseteq \phi_1(\phi_2(\phi_2^{-1}(M_1) \cap h \circ \hom(C, Z)))\\
    &= \phi(M \cap h \circ \hom(C, Z)).
  \end{align*}
  This completes the proof.
\end{proof}

\begin{LEM}\label{mnmrf.lem.hom-equiv}
  Let $\AAA$ and $\BB$ be locally small categories, $A, X, X' \in \Ob(\AAA)$ and $B, Y \in \Ob(\BB)$. If $X$ and $X'$ are hom-equivalent and
  $(A, X)_\AAA \prec (B, Y)_\BB$  then $(A, X')_\AAA \prec (B, Y)_\BB$.
\end{LEM}
\begin{proof}
  Since $(A, X)_\AAA \prec (B, Y)_\BB$ there is an $M \subseteq \hom(B, Y)$ and a function $\phi : M \to \hom(A, X)$
  as in Definition~\ref{mnmrf.def.piggyback}. Fix a pair of morphisms $p \in \hom(X, X')$ and $q \in \hom(X', X)$.
  Define $\phi' : M \to \hom(A, X')$ by $\phi'(f) = p \circ \phi(f)$ and take any $h \in \hom(Y, Y)$.
  Then there is a $g \in \hom(X, X)$ satisfying
  $$
    g \circ \hom(A, X) \subseteq \phi(M \cap h \circ \hom(B, Y)).
  $$
  Since $q \circ \hom(A, X') \subseteq \hom(A, X)$ we get
  $$
    p \circ g \circ q \circ \hom(A, X') \subseteq p \circ g \circ \hom(A, X) \subseteq p \circ \phi(M \cap h \circ \hom(B, Y)).
  $$
  Therefore, for $g' = p \circ g \circ q \in \hom(X', X')$ we have that
  $$
    g' \circ \hom(A, X') \subseteq \phi'(M \cap h \circ \hom(B, Y)).
  $$
  This completes the proof.
\end{proof}

\begin{LEM}\label{mnmrf.lem.prod-pigyback}
  Let $\AAA_1$, \ldots, $\AAA_n$, $\BB_1$, \ldots, $\BB_n$ be locally small categories,
  and let $A_i, X_i \in \Ob(\AAA_i)$, $B_i, Y_i \in \Ob(\BB)$, $1 \le i \le n$, be arbitrary.
  If $(A_i, X_i)_{\AAA_i} \prec (B_i, Y_i)_{\BB_i}$ for all $1 \le i \le n$, then
  $$
    \big((A_1, \ldots, A_n),  (X_1, \ldots, X_n)\big)_{\AAA_1 \times \ldots \times \AAA_n} \prec
    \big((B_1, \ldots, B_n),  (Y_1, \ldots, Y_n)\big)_{\BB_1 \times \ldots \times \BB_n}.
  $$
\end{LEM}
\begin{proof}
    Since $(A_i, X_i)_{\AAA_i} \prec (B_i, Y_i)_{\BB_i}$, $1 \le i \le n$,
    there is an $M_i \subseteq \hom(B_i, Y_i)$ and a function $\phi_i : M_i \to \hom(A_i, X_i)$
    as in Definition~\ref{mnmrf.def.piggyback}. Then
    $$
      M^* = M_1 \times \ldots \times M_n \subseteq \hom\big((B_1, \ldots, B_n),  (Y_1, \ldots, Y_n)\big)
    $$
    and
    $$
      \phi^* : M^* \to \hom\big((A_1, \ldots, A_n),  (X_1, \ldots, X_n)\big)
    $$
    given by $\phi^*(f_1, \ldots, f_n) = (\phi_1(f_1), \ldots, \phi_n(f_n))$ satisfy the requirements of Definition~\ref{mnmrf.def.piggyback}.
\end{proof}

\subsection{The non-scattered case}

Let $\ChEmb$ be the category of all chains and embeddings between them (so that $\hom_{\ChEmb}(A, B) = \Emb(A, B)$).
Although the main result of this section, Theorem~\ref{mnmrf.thm.CHAIN-PROD}, is a statement about chains as first-order structures,
the computations are much easier in an auxiliary category $\PQ$ defined as follows.
The objects of $\PQ$ are $\NN \cup \{\QQ\}$, that is, all finite chains 1, 2, 3, \ldots, $n$, \ldots\ together with the chain $\QQ$;
the morphisms in $\PQ$ are defined as follows:
\begin{itemize}
  \item $\hom_\PQ(n, n) = \{\id_n\}$, and $\hom_\PQ(m, n) = \0$ for $n \ne m$ ($m, n \in \NN$),
  \item $\hom_\PQ(\QQ, \QQ) = \Emb(\QQ, \QQ)$ and $\hom_\PQ(\QQ, n) = \0$ ($n \in \NN$),
  \item $\hom_\PQ(n, \QQ)$ contains all partial maps $n \rightharpoonup \QQ$ including the empty map $\0$,
        that is,
        all set-functions of the form $f : A  \to \QQ$ where $A \subseteq n$ ($n \in \NN$);
\end{itemize}
and the composition is the usual composition of (partial) functions. In particular,
if $f : n \rightharpoonup \QQ$ is a partial function with $\dom(f) = A \subseteq n$
and $h : \QQ \hookrightarrow \QQ$ is an embedding, then $h \circ f : n \rightharpoonup \QQ$ is a partial function
with $A$ as its domain defined so that $(h \circ f)(x) = h(f(x))$ for all $x \in A$.

\begin{LEM}\label{mnmrf.lem.PQ-fbrd}
  In the category $\PQ$ every finite chain has finite big Ramsey degree in $\QQ$, that is, $T_\PQ(n, \QQ) < \infty$ for all $n \in \NN$.
\end{LEM}
\begin{proof}
  Fix an $n \in \NN$. An \emph{$n$-type} is either the empty tuple $\0$, or a tuple $\tau = (A_0, A_1, \ldots, A_{m-1})$ such that
  $\0 \ne A = A_0 \cup A_1 \cup \ldots \cup A_{m-1} \subseteq n$ and $\{A_0, A_1, \ldots, A_{m-1}\}$ is a partition of~$A$.
  We say that a partial map $f : n \rightharpoonup \QQ$ \emph{is of type $\tau$} and write $\tp(f) = \tau$ if
  either $f = \0$ and $\tau = \0$, or
  \begin{itemize}
    \item $\dom(f) = A \ne \0$,
    \item $(\forall j < m)(\forall x, y \in A_j) f(x) = f(y)$, and
    \item $(\forall i < j < m)(\forall x \in A_i)(\forall y \in A_j) f(x) < f(y)$.
  \end{itemize}
  We say that $m$ is the length of $\tau$ and write $m = |\tau|$ with $|\0| = 0$. Let
  $$
    \hom_\tau(n, \QQ) = \{f \in \hom_\PQ(n, \QQ) : \tp(f) = \tau\}.
  $$

  \medskip

  Claim. For every coloring $\chi : \hom_\tau(n, \QQ) \to k$ there is an embedding $w : \QQ \hookrightarrow \QQ$ such that
  $|\chi(w \circ \hom_\tau(n, \QQ))| \le T(m, \QQ)$ where the big Ramsey degree is computed in $\ChEmb$.
  By convention we take $T(0, \QQ) = 1$.

  Proof. The statement trivially holds for $\tau = \0$. Assume, therefore, that $\tau \ne \0$.
  There is a bijective correspondence $\Phi : \Emb(m, \QQ) \to \hom_\tau(n, \QQ)$ which assigns to each
  $g : m \hookrightarrow \QQ$ a partial map $f = \Phi(g) : n \rightharpoonup \QQ$ such that $\dom(f) = A$ and
  for every $j < m$ and every $a \in A_j$ we have that $f(a) = g(j)$. Now, define $\gamma : \Emb(m, \QQ) \to k$
  by $\gamma(g) = \chi(\Phi(g))$. Then there is an embedding $w : \QQ \hookrightarrow \QQ$ such that
  $|\gamma(w \circ \Emb(m, \QQ))| \le T(m, \QQ)$. Therefore,
  $|\chi(\Phi(w \circ \Emb(m, \QQ)))| \le T(m, \QQ)$. To finish the proof of the claim it suffices to note that $\Phi(w \circ g) = w \circ \Phi(g)$
  for every $g \in \Emb(m, \QQ)$, and that $\Phi(\Emb(m, \QQ)) = \hom_\tau(n, \QQ)$.

  \bigskip

  Take any coloring $\chi : \hom_\PQ(n, \QQ) \to k$. Let us enumerate all $n$-types of all lengths as $\tau_1$, \ldots, $\tau_s$. Note that
  $\hom_\PQ(n, \QQ) = \bigcup_{j=1}^s \hom_{\tau_j}(n, \QQ)$ and that this is a disjoint union.
  We shall now inductively construct a sequence of colorings $\chi_1$, \ldots, $\chi_s$ and a sequence of embeddings $w_1, \ldots, w_s \in \Emb(\QQ, \QQ)$.
  To start the induction define $\chi_1 : \hom_{\tau_1}(n, \QQ) \to k$ by $\chi_1(f) = \chi(f)$. Then by the Claim there is a $w_1 \in \Emb(\QQ, \QQ)$
  such that
  $$
    |\chi_1(w_1 \circ \hom_{\tau_1}(n, \QQ))| \le T_1,
  $$
  where $T_1 = T_{\ChEmb}(|\tau_1|, \QQ)$. Assume, now, that $\chi_1$, \ldots, $\chi_{j-1}$ and
  embeddings $w_1, \ldots, w_{j-1} \in \Emb(\QQ, \QQ)$ have been constructed. Define  $\chi_j : \hom_{\tau_j}(n, \QQ) \to k$ by
  $$
    \chi_j(f) = \chi(w_1 \circ \ldots \circ w_{j-1} \circ f).
  $$
  By the Claim there is a $w_j \in \Emb(\QQ, \QQ)$ such that
  $$
    |\chi_j(w_j \circ \hom_{\tau_j}(n, \QQ))| \le T_j,
  $$
  where $T_j = T_{\ChEmb}(|\tau_j|, \QQ)$.
  Let $w = w_1 \circ w_2 \circ \ldots \circ w_s$. Then
  \begin{gather*}\textstyle
    |\chi(w \circ \hom_\PQ(n, \QQ))| = \sum_{j=1}^s |\chi(w \circ \hom_{\tau_j}(n, \QQ))| = \\
    \textstyle = \sum_{j=1}^s |\chi(w_1 \circ \ldots \circ w_{j-1} \circ w_j \circ w_{j+1} \circ \ldots \circ w_s \circ \hom_{\tau_j}(n, \QQ))| \le \\
    \textstyle \le \sum_{j=1}^s |\chi_j(w_j \circ \hom_{\tau_j}(n, \QQ))| \le \sum_{j=1}^s T_j,
  \end{gather*}
  having in mind the fact that $w_{j+1} \circ \ldots \circ w_s \circ \hom_{\tau_j}(n, \QQ) \subseteq \hom_{\tau_j}(n, \QQ)$,
  the definition of~$\chi_j$ and the choice of $w_j$. This completes the proof.
\end{proof}

For partial functions $f_1 : n_1 \rightharpoonup \QQ$, \ldots, $f_s : n_s \rightharpoonup \QQ$ let
$$
  f_1 \oplus \ldots \oplus f_s : n_1 + \ldots + n_s \rightharpoonup \QQ
$$
denote the partial function constructed as follows: $f_j(i)$ is defined if and only if $(f_1 \oplus \ldots \oplus f_s)(n_1 + \ldots + n_{j-1} + i)$ is defined, and then
$f_j(i) = (f_1 \oplus \ldots \oplus f_s)(n_1 + \ldots + n_{j-1} + i)$. In other words, $f_1 \oplus \ldots \oplus f_s$ is
constructed by ``concatenating'' the partial functions $f_1$, \ldots, $f_s$.

\begin{LEM}\label{mnmrf.lem.PQprod-to-PQ}
    Let $s \in \NN$ be a positive integer and let $n_1, \ldots, n_s \in \NN$ be finite chains. Then
    $
      \big((n_1, \ldots, n_s), (\QQ, \ldots, \QQ)\big)_{\PowPQ{s}} \prec (n_1 + \ldots + n_s, \QQ)_\PQ
    $,
    where $\PowPQ{s} = \PQ \times \ldots \times \PQ$ ($s$ times).
\end{LEM}
\begin{proof}
  Define
  $$
    \psi : \hom_{\PowPQ{s}}\big((n_1, \ldots, n_s), (\QQ, \ldots, \QQ)\big) \to \hom_\PQ(n_1 + \ldots + n_s, \QQ)
  $$
  by $\psi(f_1, \ldots, f_s) = f_1 \oplus \ldots \oplus f_s$.
  Take any $h \in \Emb(\QQ, \QQ)$ and let $g = (\underbrace{h, \ldots, h}_{\text{$s$ times}})$.
  Then it is easy to check that
  $$
    \psi\Big(g \circ \hom_{\PowPQ{s}} \big((n_1, \ldots, n_s), (\QQ, \ldots, \QQ)\big)\Big) \subseteq
    h \circ \hom_\PQ (n_1 + \ldots + n_s, \QQ)
  $$
  because
  \begin{multline*}
    \psi\big((h, \ldots, h) \circ (f_1, \ldots, f_s)\big) = \psi\big((h \circ f_1, \ldots, h \circ f_s)\big) =\\
    = (h \circ f_1) \oplus \ldots \oplus (h \circ f_s) = h \circ (f_1 \oplus \ldots \oplus f_s).
  \end{multline*}
  The claim now follows from Lemma~\ref{mnmrf.lem.rev-piggyback}.
\end{proof}

\begin{PROP}\label{mnmrf.prop.non-scattered-chain-lemma}
    Let $\calC$ be a non-scattered countable chain and let $n \in \NN$ be a finite chain.
    Then $(n, \calC)_{\ChEmb} \prec (n, \QQ)_{\PQ}$.
\end{PROP}
\begin{proof}
    Note first that $(n, \QQ)_{\ChEmb} \prec (n, \QQ)_{\ChEmb}$ trivially.
    Since $\calC$ is a non-scattered countable chain it is bi-embeddable with $\QQ$.
    In other words, $\calC$ and $\QQ$ are hom-equivalent in $\ChEmb$,
    so $(n, \calC)_{\ChEmb} \prec (n, \QQ)_{\ChEmb}$ by Lemma~\ref{mnmrf.lem.hom-equiv}.
    It is easy to see that $(n, \QQ)_\ChEmb \prec (n, \QQ)_\PQ$,
    so the statement follows by transitivity of~$\prec$ (Lemma~\ref{mnmrf.lem.piggyback-transitivity}).
\end{proof}

\subsection{The scattered case}

Let us recall some notions and adapt some facts from~\cite{masul-fbrd-chains}.
A \emph{rooted tree} is a triple $\tau = (T, \Boxed\le, v_0)$ where $(T, \Boxed\le)$
is a partially ordered set, $v_0 \in T$ is the \emph{root} of $T$ and for every $x \in T$ the interval
$[v_0, x]_T = \{a \in T : v_0 \le a \le x\}$ is nonempty and well-ordered.
Maximal chains in~$(T, \Boxed\le)$ are called the \emph{branches} of $\tau$.

Since we are interested in trees coding countable scattered chains of finite Hausdorff rank,
the following notion will be convenient: we shall say that a rooted tree is \emph{small} if all of
its branches are finite and every vertex in the tree has at most countably many immediate successors.
Note that every tree that codes a countable scattered chain of finite Hausdorff rank is small,
so \emph{all the trees in this subsection are small.}

A vertex $x \in T$ is a \emph{leaf of $\tau$} if it has no immediate successors.
Every branch in a small rooted tree starts at the root of the tree and ends in a leaf.

Let $\{b_\xi : \xi < \alpha\}$ be a set containing some branches of a small rooted tree $\tau = (T, \Boxed\le, v_0)$. The \emph{subtree of $\tau$
induced by branches $b_\xi$, $\xi < \alpha$,} is the subtree of $\tau$ induced by the set of vertices $\UNION_{\xi < \alpha} b_\xi$.

A small rooted tree $\tau = (T, \Boxed\le, v_0)$ is \emph{ordered} if we are given a linear order on each of the successor sets of the tree.
Let $\tau = (T, \Boxed\le, v_0)$ be a small ordered small rooted tree.
Since every branch in $\tau$ is a finite set of vertices leading to the root of the tree,
every vertex $x \in T$ has a finite \emph{height} $\high_\tau(x)$, and
every pair of vertices $x, y \in T$ has a unique \emph{meet} $x \wedge y$ in $\tau$.
The linear orders of the successor sets in $\tau$ uniquely determine a linear
ordering $\preccurlyeq_{\mathit{BFS}}$ on the vertices of $T$ which we refer to as the \emph{BFS-ordering of $\tau$}:
just traverse the tree using the breadth-first-search strategy. This means that we start
with the root $v_0$, then list the immediate successors of $v_0$ in the prescribed order, and so on.
More precisely, we let $x \preccurlyeq_{\mathit{BFS}} y$ if:
\begin{itemize}
    \item $x \le y$; or
    \item $x$ and $y$ are incomparable with respect to $\le$, but $\high_\tau(x) < \high_\tau(y)$; or
    \item $x$ and $y$ are incomparable with respect to $\le$, $\high_\tau(x) = \high_\tau(y)$, and for
          $z = x \wedge y$ there exist immediate successors $x'$ and $y'$ of $z$ such that $x' \le x$, $y' \le y$
          and $x'$ is smaller than $y'$ with respect to the linear order imposed on the immediate successors of $z$.
\end{itemize}

A \emph{small labeled ordered rooted tree} is a small ordered rooted tree whose vertices are labeled by the elements of some set $L_v$,
and edges are labeled by the elements of some set $L_e$.
For a small labeled ordered rooted tree $\tau$ by $L_v(\tau)$ we denote the set of vertex labels that appear in~$\tau$,
and by $L_e(\tau)$ we denote the set of edge labels that appear in~$\tau$.

Let us now define a family of sets $\frakA_n$, $n \in \omega$, of small labeled ordered rooted trees and the scattered chains they encode.
Let $L_v = \{0, 1, \Boxed+, \omega, \omega^*\}$ be the set of vertex labels and let
$L_e = \omega \union \{\iota_n : n \in \omega\}$ be the set of edge labels.
Let $\frakA_0 = \{ \bullet0, \bullet1 \}$ be the set whose elements are
single-vertex trees $\bullet0$ (a vertex labeled by $0$) and $\bullet1$ (a vertex labeled by $1$);
the chains these trees encode are $\|\bullet0\| = \0$ -- the empty chain, and $\|\bullet1\| = 1$ -- the trivial one-element chain.

\begin{figure}
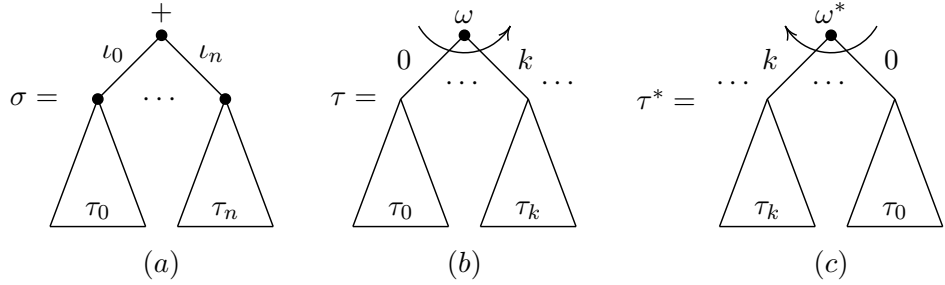

  \centering
\begin{pgfpicture}
  \pgfsetxvec{\pgfpoint{\acadpgfunit}{0pt}}
  \pgfsetyvec{\pgfpoint{0pt}{\acadpgfunit}}
  \pgfsetlinewidth{\acadpgflinewidth}
  \pgftransformshift{\pgfpointxy{800.0}{125.0}}

  \begin{pgfscope}
    \pgfpathmoveto{\pgfpointxy{75.0}{250.0}}
    \pgfpathlineto{\pgfpointxy{150.0}{50.0}}
    \pgfusepath{stroke}
  \end{pgfscope}
  \begin{pgfscope}
    \pgfpathmoveto{\pgfpointxy{150.0}{50.0}}
    \pgfpathlineto{\pgfpointxy{0.0}{50.0}}
    \pgfusepath{stroke}
  \end{pgfscope}
  \begin{pgfscope}
    \pgfpathmoveto{\pgfpointxy{0.0}{50.0}}
    \pgfpathlineto{\pgfpointxy{75.0}{250.0}}
    \pgfusepath{stroke}
  \end{pgfscope}
  \begin{pgfscope}
    \pgfpathmoveto{\pgfpointxy{275.0}{250.0}}
    \pgfpathlineto{\pgfpointxy{200.0}{50.0}}
    \pgfusepath{stroke}
  \end{pgfscope}
  \begin{pgfscope}
    \pgfpathmoveto{\pgfpointxy{200.0}{50.0}}
    \pgfpathlineto{\pgfpointxy{350.0}{50.0}}
    \pgfusepath{stroke}
  \end{pgfscope}
  \begin{pgfscope}
    \pgfpathmoveto{\pgfpointxy{350.0}{50.0}}
    \pgfpathlineto{\pgfpointxy{275.0}{250.0}}
    \pgfusepath{stroke}
  \end{pgfscope}
  \begin{pgfscope}
    \pgfpathmoveto{\pgfpointxy{175.0}{350.0}}
    \pgfpathlineto{\pgfpointxy{75.0}{250.0}}
    \pgfusepath{stroke}
  \end{pgfscope}
  \begin{pgfscope}
    \pgfpathmoveto{\pgfpointxy{175.0}{350.0}}
    \pgfpathlineto{\pgfpointxy{275.0}{250.0}}
    \pgfusepath{stroke}
  \end{pgfscope}
  \begin{pgfscope}
    \pgfpathmoveto{\pgfpointxy{103.106}{364.461}}
    \pgfpatharcaxes{209.62}{330.38}{\pgfpointxy{82.7012}{0.0}}{\pgfpointxy{0.0}{82.7012}}
    \pgfusepath{stroke}
  \end{pgfscope}
  \begin{pgfscope}
    \pgfpathmoveto{\pgfpointxy{228.88}{350.178}}
    \pgfpatharcaxes{286.44}{330.38}{\pgfpointxy{30.7241}{0.0}}{\pgfpointxy{0.0}{30.7241}}
    \pgfusepath{stroke}
  \end{pgfscope}
  \begin{pgfscope}
    \pgfpathmoveto{\pgfpointxy{246.894}{364.461}}
    \pgfpatharcaxes{150.38}{194.319}{\pgfpointxy{30.7241}{0.0}}{\pgfpointxy{0.0}{30.7241}}
    \pgfusepath{stroke}
  \end{pgfscope}
  \begin{pgfscope}
    \pgfpathmoveto{\pgfpointxy{650.0}{250.0}}
    \pgfpathlineto{\pgfpointxy{725.0}{50.0}}
    \pgfusepath{stroke}
  \end{pgfscope}
  \begin{pgfscope}
    \pgfpathmoveto{\pgfpointxy{725.0}{50.0}}
    \pgfpathlineto{\pgfpointxy{575.0}{50.0}}
    \pgfusepath{stroke}
  \end{pgfscope}
  \begin{pgfscope}
    \pgfpathmoveto{\pgfpointxy{575.0}{50.0}}
    \pgfpathlineto{\pgfpointxy{650.0}{250.0}}
    \pgfusepath{stroke}
  \end{pgfscope}
  \begin{pgfscope}
    \pgfpathmoveto{\pgfpointxy{850.0}{250.0}}
    \pgfpathlineto{\pgfpointxy{775.0}{50.0}}
    \pgfusepath{stroke}
  \end{pgfscope}
  \begin{pgfscope}
    \pgfpathmoveto{\pgfpointxy{775.0}{50.0}}
    \pgfpathlineto{\pgfpointxy{925.0}{50.0}}
    \pgfusepath{stroke}
  \end{pgfscope}
  \begin{pgfscope}
    \pgfpathmoveto{\pgfpointxy{925.0}{50.0}}
    \pgfpathlineto{\pgfpointxy{850.0}{250.0}}
    \pgfusepath{stroke}
  \end{pgfscope}
  \begin{pgfscope}
    \pgfpathmoveto{\pgfpointxy{750.0}{350.0}}
    \pgfpathlineto{\pgfpointxy{650.0}{250.0}}
    \pgfusepath{stroke}
  \end{pgfscope}
  \begin{pgfscope}
    \pgfpathmoveto{\pgfpointxy{750.0}{350.0}}
    \pgfpathlineto{\pgfpointxy{850.0}{250.0}}
    \pgfusepath{stroke}
  \end{pgfscope}
  \begin{pgfscope}
    \pgfpathmoveto{\pgfpointxy{678.106}{364.461}}
    \pgfpatharcaxes{209.62}{330.38}{\pgfpointxy{82.7012}{0.0}}{\pgfpointxy{0.0}{82.7012}}
    \pgfusepath{stroke}
  \end{pgfscope}
  \begin{pgfscope}
    \pgfpathmoveto{\pgfpointxy{678.106}{364.461}}
    \pgfpatharcaxes{209.62}{253.56}{\pgfpointxy{30.7241}{0.0}}{\pgfpointxy{0.0}{30.7241}}
    \pgfusepath{stroke}
  \end{pgfscope}
  \begin{pgfscope}
    \pgfpathmoveto{\pgfpointxy{681.166}{341.677}}
    \pgfpatharcaxes{-14.3193}{29.6201}{\pgfpointxy{30.7241}{0.0}}{\pgfpointxy{0.0}{30.7241}}
    \pgfusepath{stroke}
  \end{pgfscope}
  \begin{pgfscope}
    \pgfpathmoveto{\pgfpointxy{-400.0}{250.0}}
    \pgfpathlineto{\pgfpointxy{-325.0}{50.0}}
    \pgfusepath{stroke}
  \end{pgfscope}
  \begin{pgfscope}
    \pgfpathmoveto{\pgfpointxy{-325.0}{50.0}}
    \pgfpathlineto{\pgfpointxy{-475.0}{50.0}}
    \pgfusepath{stroke}
  \end{pgfscope}
  \begin{pgfscope}
    \pgfpathmoveto{\pgfpointxy{-475.0}{50.0}}
    \pgfpathlineto{\pgfpointxy{-400.0}{250.0}}
    \pgfusepath{stroke}
  \end{pgfscope}
  \begin{pgfscope}
    \pgfpathmoveto{\pgfpointxy{-200.0}{250.0}}
    \pgfpathlineto{\pgfpointxy{-275.0}{50.0}}
    \pgfusepath{stroke}
  \end{pgfscope}
  \begin{pgfscope}
    \pgfpathmoveto{\pgfpointxy{-275.0}{50.0}}
    \pgfpathlineto{\pgfpointxy{-125.0}{50.0}}
    \pgfusepath{stroke}
  \end{pgfscope}
  \begin{pgfscope}
    \pgfpathmoveto{\pgfpointxy{-125.0}{50.0}}
    \pgfpathlineto{\pgfpointxy{-200.0}{250.0}}
    \pgfusepath{stroke}
  \end{pgfscope}
  \begin{pgfscope}
    \pgfpathmoveto{\pgfpointxy{-300.0}{350.0}}
    \pgfpathlineto{\pgfpointxy{-400.0}{250.0}}
    \pgfusepath{stroke}
  \end{pgfscope}
  \begin{pgfscope}
    \pgfpathmoveto{\pgfpointxy{-300.0}{350.0}}
    \pgfpathlineto{\pgfpointxy{-200.0}{250.0}}
    \pgfusepath{stroke}
  \end{pgfscope}
  \begin{pgfscope}
    \pgfsetfillcolor{black}
    \pgfpathellipse{\pgfpointxy{175.0}{350.0}}{\pgfpointxy{8.0}{0.0}}{\pgfpointxy{0.0}{8.0}}
    \pgfusepath{fill,stroke}
  \end{pgfscope}
  \begin{pgfscope}
    \pgfsetfillcolor{black}
    \pgfpathellipse{\pgfpointxy{750.0}{350.0}}{\pgfpointxy{8.0}{0.0}}{\pgfpointxy{0.0}{8.0}}
    \pgfusepath{fill,stroke}
  \end{pgfscope}
  \begin{pgfscope}
    \pgfsetfillcolor{black}
    \pgfpathellipse{\pgfpointxy{-300.0}{350.0}}{\pgfpointxy{8.0}{0.0}}{\pgfpointxy{0.0}{8.0}}
    \pgfusepath{fill,stroke}
  \end{pgfscope}
  \begin{pgfscope}
    \pgfsetfillcolor{black}
    \pgfpathellipse{\pgfpointxy{-400.0}{250.0}}{\pgfpointxy{8.0}{0.0}}{\pgfpointxy{0.0}{8.0}}
    \pgfusepath{fill,stroke}
  \end{pgfscope}
  \begin{pgfscope}
    \pgfsetfillcolor{black}
    \pgfpathellipse{\pgfpointxy{-200.0}{250.0}}{\pgfpointxy{8.0}{0.0}}{\pgfpointxy{0.0}{8.0}}
    \pgfusepath{fill,stroke}
  \end{pgfscope}
  \pgftext[bottom,at={\pgfpointxy{175.0}{370.0}}]{$\omega$}
  \pgftext[bottom,at={\pgfpointxy{75.0}{62.0}}]{$\tau_0$}
  \pgftext[bottom,at={\pgfpointxy{275.0}{62.0}}]{$\tau_k$}
  \pgftext[right,at={\pgfpointxy{38.0}{250.0}}]{$\tau =$}
  \pgftext[at={\pgfpointxy{175.0}{275.0}}]{$\cdots$}
  \pgftext[at={\pgfpointxy{325.0}{275.0}}]{$\cdots$}
  \pgftext[bottom,right,at={\pgfpointxy{92.0}{297.461}}]{$0$}
  \pgftext[bottom,left,at={\pgfpointxy{258.0}{297.461}}]{$k$}
  \pgftext[bottom,at={\pgfpointxy{750.0}{370.0}}]{$\omega^*$}
  \pgftext[bottom,at={\pgfpointxy{650.0}{62.0}}]{$\tau_k$}
  \pgftext[bottom,at={\pgfpointxy{850.0}{62.0}}]{$\tau_0$}
  \pgftext[right,at={\pgfpointxy{538.0}{250.0}}]{$\tau^* =$}
  \pgftext[at={\pgfpointxy{750.0}{275.0}}]{$\cdots$}
  \pgftext[at={\pgfpointxy{600.0}{275.0}}]{$\cdots$}
  \pgftext[bottom,right,at={\pgfpointxy{667.0}{297.461}}]{$k$}
  \pgftext[bottom,left,at={\pgfpointxy{833.0}{297.461}}]{$0$}
  \pgftext[bottom,at={\pgfpointxy{-300.0}{370.0}}]{$+$}
  \pgftext[bottom,right,at={\pgfpointxy{-358.0}{308.0}}]{$\iota_0$}
  \pgftext[bottom,left,at={\pgfpointxy{-242.0}{308.0}}]{$\iota_{n}$}
  \pgftext[bottom,at={\pgfpointxy{-400.0}{62.0}}]{$\tau_0$}
  \pgftext[bottom,at={\pgfpointxy{-200.0}{62.0}}]{$\tau_{n}$}
  \pgftext[right,at={\pgfpointxy{-462.0}{250.0}}]{$\sigma =$}
  \pgftext[at={\pgfpointxy{-300.0}{250.0}}]{$\cdots$}
  \pgftext[top,at={\pgfpointxy{-300.0}{13.0}}]{$(a)$}
  \pgftext[top,at={\pgfpointxy{175.0}{13.0}}]{$(b)$}
  \pgftext[top,at={\pgfpointxy{750.0}{13.0}}]{$(c)$}
\end{pgfpicture}
  \caption{The three summations on trees}
  \label{mnmrf.fig.trees}
\end{figure}

Assume that $\frakA_i$ have been defined for all $i < m$ and let us define three operations on trees as follows:
\begin{itemize}
\item
  For $n \in \NN$ and $\tau_0, \ldots, \tau_n \in \bigcup_{i < m} \frakA_i$ let
  $\sigma$ be the tree whose root is labeled by $+$,
  edges going out of the root are labeled by $\iota_0, \ldots, \iota_{n}$ and are ordered that way, and each edge $\iota_k$
  leads to a subtree isomorphic to $\tau_k$, $0 \le k \le n$, Fig.~\ref{mnmrf.fig.trees}~$(a)$.
  Let us denote this tree as $\sigma = \tau_0 + \ldots + \tau_n$;
  the chain it encodes is $\|\sigma\| = \|\tau_0\| + \ldots + \|\tau_{n}\|$.
\item
  For $\tau_k \in \bigcup_{i < m} \frakA_i$, $k \in \omega$,
  let $\tau$, resp.\ $\tau^*$, be a tree whose root is labeled by $\omega$, resp.\ $\omega^*$,
  edges going out of the root are labeled by and ordered as $\omega$, resp.\ $\omega^*$,
  and each edge labeled by $k \in \omega$ leads to a subtree isomorphic to $\tau_k$, $k \in \omega$,
  Fig.~\ref{mnmrf.fig.trees}~$(b)$ and~$(c)$.
  Let us denote the tree as $\tau = \sum_{k \in \omega} \tau_k$, resp.\ $\tau^* = \sum_{k \in \omega^*} \tau_k$;
  the chain it encodes is $\|\tau\| = \sum_{k \in \omega} \|\tau_k\|$, resp.\ $\|\tau^*\| = \sum_{k \in \omega^*} \|\tau_k\|$.
\end{itemize}
Then put
\begin{align*}
  \frakA_m =
  &\textstyle \Big\{\sum_{k \in \omega} \tau_k : \tau_k \in \bigcup_{i < m} \frakA_i \text{ and } \|\tau_0\| \hookrightarrow \|\tau_1\| \hookrightarrow \ldots\Big\}\\
  &\textstyle \cup \Big\{\sum_{k \in \omega^*} \tau_k : \tau_k \in \bigcup_{i < m} \frakA_i \text{ and } \|\tau_0\| \hookrightarrow \|\tau_1\| \hookrightarrow \ldots\Big\}\\
  &\textstyle \cup
    \begin{array}[t]{@{}l@{\,}l}
      \Big\{\sum_{k \in \omega} (\tau_{k0} + \ldots +  \tau_{kn}) : & n \in \NN, \tau_{kj} \in \bigcup_{i < m} \frakA_i \text{ and } \\
                                                                    & \|\tau_{0j}\| \hookrightarrow \|\tau_{1j}\| \hookrightarrow \ldots \text{ for all } j\Big\}
    \end{array}\\
  &\textstyle \cup
    \begin{array}[t]{@{}l@{\,}l}
      \Big\{\sum_{k \in \omega^*} (\tau_{k0} + \ldots +  \tau_{kn}) : & n \in \NN, \tau_{kj} \in \bigcup_{i < m} \frakA_i \text{ and } \\
                                                                      & \|\tau_{0j}\| \hookrightarrow \|\tau_{1j}\| \hookrightarrow \ldots \text{ for all } j\Big\}.
    \end{array}
\end{align*}
and let 
$
  \frakA = \bigcup_{m \in \omega} \frakA_m
$.
Furthermore, let $\frakS$ be the set of trees defined as ``finite sums of trees from~$\frakA$'':
$$
  \frakS = \frakA \union \{\tau_0 + \ldots + \tau_n : n \in \NN \text{ and } \tau_0, \ldots, \tau_n \in \frakA\}.
$$
Then building on Laver's results from~\cite{laver-fraisse-conj} it is easy to show:

\begin{LEM} \cite[Lemma 5.4]{masul-fbrd-chains} 
    A chain $\calS$ is a countable scattered chain of finite Hausdorff rank if and only if there is a tree $\sigma \in \frakS$ such that
    $\calS$ and $\|\sigma\|$ are bi-embeddable.
\end{LEM}

\begin{LEM}\label{mnmrf.lem.no-leaves=0}
  If $\sigma \in \frakS$ encodes a nonempty chain then there is a
  $\sigma' \in \frakS$ such that $\|\sigma\|$ and $\|\sigma'\|$ are bi-embeddable and no leaf in $\sigma'$ is labeled by~0.
\end{LEM}
\begin{proof}
    Clearly, it suffices to show that for every $m \in \omega$ and every $\tau \in \frakA_m$, either $\|\tau\| = \0$, or there is a
    $\tau' \in \bigcup_{i \le m} \frakA_i$ such that $\|\tau\|$ and $\|\tau'\|$ are bi-embeddable
    and no leaf in $\tau'$ is labeled by~0. The proof is by induction on $m$. The claim trivially holds for $\tau \in \frakA_0$.
    Assume that the claim is true for all $i < m$ and take any $\tau \in \frakA_m$.

    \medskip

    Case 1: $\tau = \sum_{k \in \omega} \tau_k$ where $\tau_k \in \bigcup_{i < m} \frakA_i$ and $\|\tau_0\| \hookrightarrow \|\tau_1\| \hookrightarrow \ldots$:
    
    It $\|\tau_k\| = \0$ for all $k$ then $\|\tau\| = \0$. Assume, therefore, that there is a $k$ such that $\|\tau_k\| \ne \0$
    and let $k_0$ be the least such $k$. Since $\|\tau_{k_0}\| \hookrightarrow \|\tau_{k_0+1}\| \hookrightarrow \ldots$
    we know that $\|\tau_k\| \ne \0$ for all $k \ge k_0$. By the induction hypothesis, for each $k \ge k_0$
    there is a $\tau'_k \in \bigcup_{i < m} \frakA_i$ such that $\|\tau_k\|$ and $\|\tau'_k\|$ are bi-embeddable
    and no leaf in $\tau'_k$ is labeled by~0. Then put $\tau' = \sum_{k \ge k_0} \tau'_k$.

    \medskip
    
    Case 2: $\tau = \sum_{k \in \omega^*} \tau_k$ where $\tau_k \in \bigcup_{i < m} \frakA_i$ and $\|\tau_0\| \hookrightarrow \|\tau_1\| \hookrightarrow \ldots$:

    Analogous to Case 1.

    \medskip

    Case 3: $\tau = \sum_{k \in \omega} (\tau_{k0} + \ldots +  \tau_{kn})$ where 
    $n \in \NN$, $\tau_{kj} \in \bigcup_{i < m} \frakA_i$ for all $k$ and $j$, and 
    $\|\tau_{0j}\| \hookrightarrow \|\tau_{1j}\| \hookrightarrow \ldots $ for all $j$:

    If $\|\tau_{kj}\| = \0$ for all $k$ and $j$ then $\|\tau\| = \0$. Assume, therefore, that $\|\tau_{kj}\| \ne \0$ for some
    $k$ and $j$. Without loss of generality we may assume that there is an integer $n' \in \{0, 1, \ldots, n\}$ such that
    \begin{itemize}
      \item for each $j \in \{0, \ldots, n'\}$ there is a $k$ with $\|\tau_{kj}\| \ne \0$; and
      \item for each $j > n'$ and for each $k \ge 0$ we have $\|\tau_{kj}\| = \0$.
    \end{itemize}
    If $n' = 0$ then $\|\tau_{k0} + \ldots +  \tau_{kn}\| = \|\tau_{k0}\|$ for all $k$, reducing the case to Case~1.
    Assume, therefore, that $n' \in \NN$. Choose $k_0 \in \omega$ with the property that
    $\|\tau_{k_0j}\| \ne \0$ for all $0 \le j \le n'$. Since $\|\tau_{0j}\| \hookrightarrow \|\tau_{1j}\| \hookrightarrow \ldots $ for all $j$,
    it follows that $\|\tau_{kj}\| \ne \0$ for all $0 \le j \le n'$ and $k \ge k_0$. By the induction hypothesis
    for each $0 \le j \le n'$ and $k \ge k_0$ there is a $\tau'_{kj} \in \bigcup_{i < m} \frakA_i$ such that
    $\|\tau_{kj}\|$ and $\|\tau'_{kj}\|$ are bi-embeddable and no leaf in $\tau'_{kj}$ is labeled by~0.
    Then put $\tau' = \sum_{k \ge k_0} (\tau'_{k0} + \ldots +  \tau'_{kn'})$.
    
    \medskip

    Case 4: $\tau = \sum_{k \in \omega^*} (\tau_{k0} + \ldots +  \tau_{kn})$ where 
    $n \in \NN$, $\tau_{kj} \in \bigcup_{i < m} \frakA_i$ for all $k$ and $j$, and 
    $\|\tau_{0j}\| \hookrightarrow \|\tau_{1j}\| \hookrightarrow \ldots $ for all $j$:

    Analogous to Case 3.
\end{proof}

Take any $\sigma \in \frakS$ and assume that no leaves in $\sigma$ are labeled by~0.
Each branch in $\sigma$ can be represented as a string of symbols from
$$
  \Lambda = \{\iota_0, \iota_1, \ldots \} \cup \{(\omega0), (\omega1), \ldots \} \cup \{(\omega^*0), (\omega^*1), \ldots \}
$$
by recording every label we encounter while traversing the branch from the root, see Fig.~\ref{mnmrf.fig.branches}.
To save space we shall skip labels $+$ preceding the $\iota_j$'s and labels $1$ that are mandatory labels of leaves.
Let $\Br(\sigma)$ denote the set of thus generated strings of elements of $\Lambda$.
Let $<_\Lambda$ denote the lexicographic ordering of strings of symbols from $\Lambda$
generated by the following ordering of~$\Lambda$:
$$
  \iota_0 < \iota_1 < \ldots < (\omega0) < (\omega1) < \ldots < \ldots < (\omega^*2) < (\omega^*1) < (\omega^*0).
$$

\begin{figure}
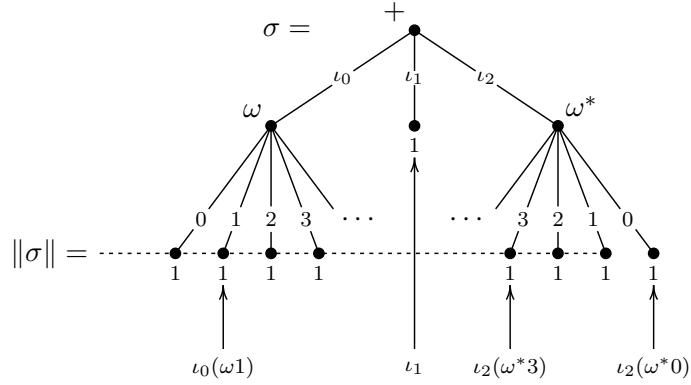

  \centering
  \def\.{\scriptsize}
  \input branches.pgf
  \caption{Encoding branches in a small labeled tree}
  \label{mnmrf.fig.branches}
\end{figure}

\begin{LEM}\label{mnmrf.lem.new-1}
  Take any $\sigma \in \frakS$ such that no leaves in $\sigma$ are labeled by~0.
  Then $\|\sigma\|$ is isomorphic to $\big(\Br(\sigma), \Boxed{<_\Lambda}\big)$.
\end{LEM}
\begin{proof}
  Clearly, the elements of $\|\sigma\|$ correspond to the elements of $\Br(\sigma)$, and
  the ordering of the elements of $\|\sigma\|$ corresponds to the ordering $<_\Lambda$ of $\Br(\sigma)$.
\end{proof}

Let $\tau$ be a small labeled ordered rooted tree whose vertices are labeled by elements of $L_v$
and edges are labeled by elements of $L_e$, and let $U \subseteq L_e$.
By $\restr \tau U$ we denote the subtree of $\tau$ induced by all of its branches whose edge labels belong to~$U$.

Let $V \subseteq \omega$ be an infinite subset of $\omega$, let $U = V \union \{\iota_n : n \in \omega\}$ and let
$\tau \in \frakS$ be arbitrary. The particular structure of $U$ ensures that the infinite sums in $\tau$ are restricted
so that $\sum_{k \in \omega}$ becomes $\sum_{k \in V}$, and similarly for $\sum_{k \in \omega^*}$.
Thus, we define $\|\restr \tau U\|$ as follows:
\begin{itemize}
\item
  if $\tau = \tau_0 + \ldots + \tau_n$ then $\|\restr\tau U\| = \|\restr{\tau_0}{U}\| + \ldots + \|\restr{\tau_n}{U}\|$;
\item
  if $\tau = \sum_{k \in \omega} \tau_k$ then $\|\restr\tau U\| = \sum_{k \in V} \|\restr{\tau_k}{U}\|$, and
  analogously in case $\tau = \sum_{k \in \omega^*} \tau_k$; and
\item
  if $\tau = \sum_{k \in \omega} (\tau_{k0} + \ldots +  \tau_{kn})$ then
  $\|\restr\tau U\| = \sum_{k \in V} (\|\restr{\tau_{k0}}{U}\| + \ldots +  \|\restr{\tau_{kn}}{U}\|)$,
  and analogously in case $\tau = \sum_{k \in \omega^*} (\tau_{k0} + \ldots +  \tau_{kn})$.
\end{itemize}

\begin{LEM}\label{mnmrf.lem.new-2}
  Let $V \subseteq \omega$ be an infinite subset of $\omega$, let $U = V \union \{\iota_n : n \in \omega\}$ and let
  $\tau \in \frakS$ be a tree such that none of its leaves is labeled by~0. Then
  $\big(\Br(\tau), \Boxed{<_\Lambda}\big)$ and $\big(\Br(\restr \tau U), \Boxed{<_\Lambda}\big)$
  are bi-embeddable.
\end{LEM}
\begin{proof}
  By Lemma~\ref{mnmrf.lem.new-1} we have that $\|\tau\|$ is isomorphic to $\big(\Br(\tau), \Boxed{<_\Lambda}\big)$
  and $\|\restr\tau U\|$ is isomorphic to $\big(\Br(\tau), \Boxed{<_\Lambda}\big)$.
  It is obvious that $\|\restr\tau U\| \hookrightarrow \|\tau\|$, so $\big(\Br(\restr \tau U), \Boxed{<_\Lambda}\big)$
  embeds into $\big(\Br(\tau), \Boxed{<_\Lambda}\big)$.

  For the other direction, enumerate the elements of $V$ as $V = \{v_0 < v_1 < \ldots \}$. Define
  $$
    g_U : \Lambda \to \Lambda
  $$
  so that $g_U(\iota_n) = \iota_n$, $n \in \omega$, while for $i \in \omega$ we put $g_U\big((\omega i)\big) = (\omega v_i)$ and 
  $g_U\big((\omega^* i)\big) = (\omega^* v_i)$. Clearly, $g_U$ expands to strings of elements of $\Lambda$
  in the obvious way:
  $$
    g_U(\lambda_1 \ldots \lambda_k) = g_U(\lambda_1) \ldots g_U(\lambda_k),
  $$
  $\lambda_1, \ldots, \lambda_k \in \Lambda$.
  Then $g_U$ is clearly an embedding of $\big(\Br(\tau), \Boxed{<_\Lambda}\big)$ into $\big(\Br(\restr \tau U), \Boxed{<_\Lambda}\big)$.
\end{proof}

A tree $\sigma \in \frakS$ has \emph{bounded finite sums} \cite{masul-fbrd-chains} if there is an
integer $b \in \NN$ such that $L_e(\sigma) \subseteq \omega \union \{\iota_0, \ldots, \iota_b\}$.
In other words, $\sigma$ is a tree whose finite sums have at most $b + 1$ summands.
A countable scattered chain $\calS$ of finite Hausdorff rank has \emph{bounded finite sums} \cite{masul-fbrd-chains}
if there is a tree $\sigma \in \frakS$ with bounded finite sums such that $\calS$ and $\|\sigma\|$ are bi-embeddable.

\begin{LEM}\label{mnmrf.lem.bounded-fin-sums-no-0-leaves}
  For every scattered countable chain $\calS$ of finite Hausdorff rank there is a $\sigma \in \frakS$
  such that $\sigma$ has bounded finite sums, none of its leaves are labeled by~0 and $\|\sigma\|$ is bi-embeddable with~$\calS$.
\end{LEM}
\begin{proof}
  This follows from Lemma~\ref{mnmrf.lem.no-leaves=0} and
  the main argument used to prove \cite[Theorem 3.1]{masul-fbrd-finale}.
  Namely, the crucial insight in the proof of \cite[Theorem 3.1]{masul-fbrd-finale} is the following claim
  which is not explicit in the paper:
  every countable scattered chain of finite Hausdorff rank is bi-embeddable
  with a countable scattered chain of finite Hausdorff rank and with bounded finite sums.
  Lemma~\ref{mnmrf.lem.no-leaves=0} then ensures that we can find another such tree with no leaves
  labeled by~0.
\end{proof}

Take any tree $\sigma \in \frakS$.
Every embedding $f : n \hookrightarrow \Br(\sigma)$, $n \in \NN$, corresponds to a subtree of $\sigma$
induced by the branches $\{f(i) : i < n\}$. Let us denote this subtree of $\sigma$ by $\tree{f}_\sigma$.
Assume, now, that $\tree{f}_\sigma$ has $p$ vertices. If we replace the vertex set of
$\tree{f}_\sigma$ by $\{0, 1, \ldots, p - 1\}$ so that the usual ordering of the integers agrees with the BFS-ordering of the
new tree, and then erase only those edge labels that come from $\omega$ or $\omega^*$, the resulting partially
labeled ordered rooted tree on the set of vertices $\{0, 1, \ldots, p-1\}$
will be referred to as the \emph{type of $f$}~\cite{masul-fbrd-chains} and will be denoted by $\tp_\sigma(f)$.
Embeddings $g_U$ introduced in the proof of Lemma~\ref{mnmrf.lem.new-2}
above are of particular interest because they preserve the types:

\begin{LEM}\label{mnmrf.lem.gU-type-reserving}
  Let $V \subseteq \omega$ be an infinite subset of $\omega$, let $U = V \union \{\iota_n : n \in \omega\}$ and let
  $\sigma \in \frakS$ be a tree such that none of its leaves is labeled by~0. Let $g_U : \Br(\sigma) \hookrightarrow \Br(\restr \sigma U)$
  be the embedding introduced in the proof of Lemma~\ref{mnmrf.lem.new-2}. Then for any embedding $f : n \hookrightarrow \Br(\sigma)$ we have that
  $\tp_\sigma(f) = \tp_\sigma(g_U \circ f)$.
\end{LEM}
\begin{proof}
  Note that the embeddings $g_U$ change labels of the form $(\omega i)$ and $(\omega^* i)$, and preserve the other labels,
  while in order to obtain the type of an embedding we erase precisely those labels. Therefore, it must be the case that
  $\tp_\sigma(f) = \tp_\sigma(g_U \circ f)$.
\end{proof}

A finite labeled ordered rooted tree $\tau$ is an
\emph{$(n, \sigma)$-type} if $\tau = \tp_\sigma(f)$ for some embedding $f : n \hookrightarrow \Br(\sigma)$.
For an $(n, \sigma)$-type $\tau$ let
$$
  \Emb_\tau(n, \Br(\sigma)) = \{ f \in \Emb(n, \Br(\sigma)) : \tp_\sigma(f) = \tau \}.
$$
The following is a simple, but important observation:

\begin{LEM}\cite{masul-fbrd-chains}
  Given an $n \in \NN$ and a $\sigma \in \calS$ with bounded finite sums, there are only finitely many $(n, \sigma)$-types.
\end{LEM}

\begin{PROP}\label{mnmrf.prop.scattered-chain-lemma}
    Let $\calS$ be a scattered countable chain of finite Hausdorff rank and let $n \in \NN$ be a finite chain.
    There is a positive integer $m$ such that $(n, \calS)_{\ChEmb} \prec (m, \QQ)_\PQ$.
\end{PROP}
\begin{proof}
  Let $\calS$ be a scattered countable chain of finite Hausdorff rank and let $n \in \NN$ be a finite chain.
  Thanks to Lemmas~\ref{mnmrf.lem.bounded-fin-sums-no-0-leaves} and~\ref{mnmrf.lem.hom-equiv}
  without loss of generality we may assume that $\calS = \|\sigma\| \cong \Br(\sigma)$ where
  $\sigma \in \frakS$ has bounded finite sums and none of its vertices labeled by~0.
  Let $\tau_1$, \ldots, $\tau_p$ be an enumeration of all $(n, \sigma)$-types. Then
  $$\textstyle
    \Emb(n, \Br(\sigma)) = \bigcup_{j=1}^p \Emb_{\tau_j}(n, \Br(\sigma))
  $$
  and this is a disjoint union. For each $\tau_j$ we shall now present a convenient encoding of embeddings from
  $\Emb_{\tau_j}(n, \Br(\sigma))$ in $\PQ$. In order to do so let us fix a bijection $\xi : \omega \to \QQ$.
  Let $\0_k : k \rightharpoonup \QQ$ denote the empty partial map. 

  Take an $(n, \sigma)$-type $\tau_j$. Assume, first, that no vertex of $\tau_j$ is labeled either by $\omega$ or by $\omega^*$.
  Then $|\Emb_{\tau_j}(n, \Br(\sigma))| = 1$, so let $m_j = 1$, let $M_j = \hom_\PQ(m_j, Q) \setminus \{\0_{m_j}\}$ and let
  $\phi_j : M_j \to \Emb_{\tau_j}(n, \Br(\sigma))$ be the constant map.
  
  Assume, now, that at least one vertex of $\tau_j$ is labeled by $\omega$ or $\omega^*$.
  Let $\ell_{j1} < \ell_{j2} < \ldots < \ell_{js_j}$ be all the vertices of~$\tau_j$ labeled by $\omega$ or $\omega^*$ and let
  let $m_{ji}$ be the number of immediate successors of $\ell_{ji}$ in $\tau_j$, $1 \le i \le s_j$.

  Take any $f \in \Emb_{\tau_j}(n, \Br(\sigma))$ and let $(v_1, v_2, \ldots, v_{p})$ be the vertex
  set of $\tree{f}_{\Br(\sigma)}$ ordered by the BFS-order of $\tree{f}_{\Br(\sigma)}$.
  Since $\tp_\sigma(f) = {\tau_j}$, the only vertices
  in $\tree{f}_{\Br(\sigma)}$ labeled by $\omega$ or $\omega^*$ are $v_{\ell_1}$, $v_{\ell_2}$, \ldots, $v_{\ell_{s}}$.
  Let $L_{ji}(f) \subseteq \omega$ be the set of all the labels used to label the edges to the immediate successors of $v_{\ell_i}$
  in $\tree{f}_{\Br(\sigma)}$, $1 \le i \le s_j$. Clearly, $|L_{ji}(f)| = m_{ji}$ for all $1 \le i \le s_j$.
  We can represent subsets $L_{ji}(f)$ of $\omega$ as embeddings
  $E_{ji}(f) : m_{ji} \hookrightarrow \omega$ so that $\im(E_{ji}(f)) = L_{ji}(f)$, $1 \le i \le s_j$.
  By construction, each embedding $f \in \Emb_{\tau_j}(n, \Br(\sigma))$ is uniquely determined by the
  sequence $(E_{j1}(f), E_{j2}(f), \ldots, E_{js_j}(f))$. Therefore,
  $$
    \psi_j : \Emb_{\tau_j}(n, \Br(\sigma)) \to \hom_\PQ(m_{j1} + \ldots + m_{js_j}, \QQ)
  $$
  given by
  $$
    \psi_j(f) = \xi \circ (E_{j1}(f) \oplus E_{j2}(f) \oplus \ldots \oplus E_{js_j}(f))
  $$
  is injective. Let $m_j = m_{j1} + m_{j2} + \ldots + m_{js_j}$, let $M_j = \im(\psi_j) \subseteq \hom_\PQ(m_j, \QQ)$ and let
  $\phi_j : M_j \to \Emb_{\tau_j}(n, \Br(\sigma))$ be the inverse of the codomain restriction of $\psi_j$.
  Note that
  $$
    M_j = \im(\psi_j) = \xi \circ (\Emb(m_{j1}, \omega) \oplus \ldots \oplus \Emb(m_{js_j}, \omega)).
  $$

  So, for each $(n, \sigma)$-type $\tau_j$, $1 \le j \le p$, we have constructed a positive integer $m_j$,
  a set $M_j \subseteq \hom_\PQ(m_j, \QQ)$ and a surjective function
  $$
    \phi_j : M_j \to \Emb_{\tau_j}(n, \Br(\sigma)).
  $$
  Let $m = m_1 + m_2 + \ldots + m_p$ and let $M \subseteq \hom_\PQ(m, \QQ)$ be the following set:
  $$
    M = \bigcup_{j=1}^p \{\0_{m_1 + \ldots + m_{j-1}}  \oplus f \oplus \0_{m_{j+1} + \ldots + m_{p}} : f \in M_j\}.
  $$
  Define $\phi : M \to \Emb(n, \Br(\sigma))$ by
  $$
    \phi(\0_{m_1 + \ldots + m_{j - 1}} \oplus f \oplus \0_{m_{j + 1} + \ldots + m_p}) = \phi_j(f).
  $$

  Take any $h \in \Emb(\QQ, \QQ)$. Then $h_0 = \xi^{-1} \circ h \circ \xi : \omega \to \omega$ is an injective map (which is not necessarily an
  embedding). Let $V = \im(h_0)$ and $U = V \cup \{\iota_k : k \in \omega\}$. Then
  $$
    g_U \circ \Emb(n, \Br(\sigma)) \subseteq \phi(M \cap h \circ \hom_\PQ(m, \QQ)).
  $$
  To see why this is indeed the case, take any $f \in \Emb(n, \Br(\sigma))$ and let $\tau_j = \tp_\sigma(f)$.

  If no vertex of $\tau_j$ is labeled either by $\omega$ or by $\omega^*$ then
  $\Emb_{\tau_j}(n, \Br(\sigma)) = \{f\}$, $m_j = 1$ and $M_j = \hom_\PQ(m_j, Q) \setminus \{\0_{m_j}\}$.
  Take any $w \in \hom_\PQ(m_1, \QQ) \setminus \{\0_{m_j}\}$ and note that
  $$
    w' = \0_{m_1 + \ldots + m_{j - 1}} \oplus (h \circ w) \oplus \0_{m_{j + 1} + \ldots + m_p} \in M \cap h \circ \hom_\PQ(m, \QQ)
  $$
  and that
  $$
    \phi(w') = \phi_j(h \circ w) = f,
  $$
  because $\phi_j$ is the constant map $M_j \to \Emb_{\tau_j}(n, \Br(\sigma)) = \{f\}$. On the other hand, $g_U \circ f = f$ by the construction of $g_U$.
  
  Assume, now, that at least one vertex of $\tau_j$ is labeled by $\omega$ or $\omega^*$.
  Then there exist embeddings $e_i : m_{ji} \hookrightarrow \omega$, $1 \le i \le s_j$, such that
  $$
    \psi_j(g_U \circ f) = \xi \circ g_U \circ (e_1 \oplus \ldots \oplus e_{s_j}).
  $$
  Let
  $
    e_1 \oplus \ldots \oplus e_{s_j} = \begin{pmatrix}
        0 & 1 & \ldots & m_j - 1\\
        k_0 & k_1 & \ldots & k_{m_j - 1}
    \end{pmatrix}
  $.
  Then
  $$
    \psi_j(g_U \circ f) = \begin{pmatrix}
        0 & 1 & \ldots & m_j - 1\\
        \xi \circ g_U(k_0) & \xi \circ g_U(k_1) & \ldots & \xi \circ g_U(k_{m_j - 1})
    \end{pmatrix} \in M_j.
  $$
  On the other hand, by the construction of $g_U$ we know that $g_U(i) \in \im(h_0) = \im(\xi^{-1} \circ h \circ \xi)$ for all $i \in \omega$.
  Therefore, for every $i \in \omega$ there is a $q \in \omega$ such that $g_U(i) = \xi^{-1} \circ h \circ \xi(q)$. Hence,
  $$
    \psi_j(g_U \circ f) = \begin{pmatrix}
        0 & 1 & \ldots & m_j - 1\\
        h \circ \xi(q_0) & h \circ \xi(q_1) & \ldots & h \circ \xi(q_{m_j - 1})
    \end{pmatrix} \in h \circ \hom_\PQ(m_j, \QQ).
  $$
  This suffices to conclude that
  $$
    g_U \circ f \in \phi(M \cap h \circ \hom_\PQ(m, \QQ)),
  $$
  and thus to conclude the proof.
\end{proof}

\subsection{The product Ramsey theorem}

We are now ready to prove the main result of the section.

\begin{THM}\label{mnmrf.thm.CHAIN-PROD}
  Let $\calC_1$, \ldots, $\calC_r$ be countable chains each with finite big Ramsey spectrum.
  For every choice of finite chains $n_1$, \ldots, $n_r$ there is a positive integer $N$
  such that for every $k \ge 1$ and every coloring
  $$
    \gamma : \Emb(n_1, \calC_1) \times \ldots \times \Emb(n_r, \calC_r) \to k
  $$
  there are embeddings $w_i : \calC_i \hookrightarrow \calC_i$, $1 \le i \le r$, such that
  $$
    \big|\gamma\big((w_1 \circ \Emb(n_1, \calC_1)) \times \ldots \times (w_r \circ \Emb(n_r, \calC_r))\big)\big| \le N.
  $$
\end{THM}
\begin{proof}
    Let $\calC_1$, \ldots, $\calC_r$ be countable chains each with finite big Ramsey spectrum, and let
    $n_1$, \ldots, $n_r$ be finite chains. We have to show that:
    $$
      T_{\PowChEmb{r}}\big((n_1, \ldots, n_r), (\calC_1, \ldots, \calC_r)\big) < \infty,
    $$
    where $\PowChEmb{r} = \ChEmb \times \ldots \times \ChEmb$ ($r$ times). Countable chains with finite
    big Ramsey spectra have been characterized in Theorem~\ref{fbrd-finale.thm.MAIN}:
    a countable chain $\calC$ has finite big Ramsey spectrum if and only if
    $\calC$ is non-scattered, or $\calC$ is a scattered chain of finite Hausdorff rank.

    If $\calC_i$ is non-scattered then for $m_i = n_i$ we have that
    $(n_i, \calC_i)_{\ChEmb} \prec (m_i, \QQ)_{\PQ}$ by Proposition~\ref{mnmrf.prop.non-scattered-chain-lemma}.
    If, however, $\calC_i$ is a scattered chain of finite Hausdorff rank then by Proposition~\ref{mnmrf.prop.scattered-chain-lemma}
    there is a positive integer $m_i \in \NN$ such that
    $(n_i, \calS)_{\ChEmb} \prec (m_i, \QQ)_{\PQ}$. Therefore, Lemma~\ref{mnmrf.lem.prod-pigyback} yields:
    $$
      \big((n_1, \ldots, n_r), (\calC_1, \ldots, \calC_r)\big)_{\PowChEmb{r}} \prec 
      \big((m_1, \ldots, m_r), (\QQ, \ldots, \QQ)\big)_{\PowPQ{r}},
    $$
    where $\PowPQ{r} = \PQ \times \ldots \times \PQ$ ($r$ times). By Lemma~\ref{mnmrf.lem.PQprod-to-PQ}
    there is a positive integer $p \in \NN$ such that
    $$
      \big((m_1, \ldots, m_r), (\QQ, \ldots, \QQ)\big)_{\PowPQ{r}} \prec (p, \QQ)_\PQ.
    $$
    Therefore,
    $$
      T_{\PowChEmb{r}}\big((n_1, \ldots, n_r), (\calC_1, \ldots, \calC_r)\big) \le T_\PQ(p, \QQ) < \infty
    $$
    by Lemma~\ref{mnmrf.lem.piggy=>brd} and Lemma~\ref{mnmrf.lem.PQ-fbrd}.
    This concludes the proof.
\end{proof}

The intricacies of relational structures admitting a finite monomorphic decomposition require
the following slight generalization:

\begin{COR}\label{mnmrf.cor.CHAIN-PROD}
  Let $\calC_1$, \ldots, $\calC_t$ be countable chains each with finite big Ramsey spectrum, and let
  $\calC_{t+1}$, \ldots, $\calC_r$ be singletons, $t \le r$. For every choice of non-negative integers
  $n_1$, \ldots, $n_r \ge 0$ where $n_{j} \le 1$ for $t+1 \le j \le r$ there is a positive integer $N$
  such that for every $k \ge 1$ and every coloring
  $$
    \gamma : \Emb(n_1, \calC_1) \times \ldots \times \Emb(n_r, \calC_r) \to k
  $$
  there are embeddings $w_i : \calC_i \hookrightarrow \calC_i$, $1 \le i \le r$, such that
  $$
    \big|\gamma\big((w_1 \circ \Emb(n_1, \calC_1)) \times \ldots \times (w_r \circ \Emb(n_r, \calC_r))\big)\big| \le N.
  $$  
\end{COR}
\begin{proof}
  It is easy to see that this is nothing but Theorem~\ref{mnmrf.thm.CHAIN-PROD} sprinkled with trivialities:
  if $n_i = 0$ then we take $\Emb(n_i, \calC_i) = \{\0\}$ and $w_i \circ \0 = \0$ for any $w_i : \calC_i \hookrightarrow \calC_i$;
  on the other hand, if $\calC_i$ is a singleton and $n_i = 1$ then $|\Emb(n_i, \calC_i)| = 1$ and the only embedding
  $\calC_i \hookrightarrow \calC_i$ is the identity.
\end{proof}

\section{Big Ramsey degrees for the generic \textcolor{blue}{2-di\-men\-si\-o\-nal} partial order}
\label{mnmrf.sec.brd-gen-poset}

In this section we \sout{provide an alternative proof of the result of Hubi\v cka from~\cite{hubicka-param-spaces}}
\textcolor{blue}{prove} that the generic \textcolor{blue}{2-dimensional}
partial order has big Ramsey degrees. \sout{Instead of Voigt's infinite $\star$-version of the
Graham-Rothschild's theorem (see \cite[Theorem~A]{promel-voigt-1985}), our} \textcolor{blue}{Our} proof relies on the tools developed in
Section~\ref{mnmrf.sec.prod-thm} and in particular on Theorem~\ref{mnmrf.thm.CHAIN-PROD}.

Let us recall some basic facts of \Fraisse's theory of relational structures~\cite{Fraisse-ThRel}.
Fix a finite relational language.
An \emph{age} of a countable relational structure $\calS$, denoted by $\Age(\calS)$, is the class of all finite structures that
$\calS$ embeds. We say that a structure $\calS'$ is \emph{younger} that $\calS$ if $\Age(\calS)$ contains $\Age(\calS')$.
A countable relational structure $\calS$ is \emph{ultrahomogeneous} if every isomorphism between finite substructures of $\calS$
extends to an automorphism of $\calS$. We shall say that a class $\KK$ of finite structures is a \emph{\Fraisse\ age}
if there exists a countable ultrahomogeneous relational structure $\calS$ such that $\KK = Age(\calS)$.
Given a \Fraisse\ age $\KK$ there is, up to isomorphism, precisely one countable ultrahomogeneous relational structure
$\calS$ with $\KK = \Age(\calS)$ and it is usually referred to as the \emph{\Fraisse\ limit of $\KK$}.
If $\calS$ is an ultrahomogeneous countable relational structure and if $\calS'$ is younger than $\calS$ then $\calS$ embeds $\calS'$.
For example, the class of all finite linear orders is a \Fraisse\ age and its \Fraisse\ limit is the order of the rationals~\cite{Fraisse1,Fraisse2}.
The order of the rationals embeds all finite and countably infinite chains.
The class of all finite partial orders is a also \Fraisse\ age and its \Fraisse\ limit is the \emph{generic partial order}~\cite{Schmerl}.
The generic partial order embeds all finite and countably infinite partial orders.

Our main goal in this section is to transport the big Ramsey degrees from the rationals $\QQ$ to the
generic \textcolor{blue}{2-dimensional} partial order via
an intermediary generic structure -- the generic permutation. From a traditional point of view
a permutation of a set $A$ is any bijection $f : A \to A$. If $A$ is finite, say $A = \{a_1, a_2, \ldots, a_n\}$,
then each permutation $f : A \to A$ can be represented as
$
    f = \left(\begin{smallmatrix}
      a_1 & a_2 & \ldots & a_n \\
      a_{i_1} & a_{i_2} & \ldots & a_{i_n}
    \end{smallmatrix}\right)
$.
So, in order to specify a permutation it suffices to specify two linear orders on $A$: the ``standard'' order
$a_1 < a_2 < \ldots < a_n$ on $A$, and the permuted order $a_{i_1} \sqsubset a_{i_2} \sqsubset \ldots \sqsubset a_{i_n}$.
In this paper we adopt P.~J.~Cameron's reinterpretation of permutations in model-theoretic terms~\cite{cameron-perm}
and say that a \emph{permutation} is a triple $(A, \Boxed<, \Boxed\sqsubset)$ where $<$ and $\sqsubset$ are linear orders on~$A$.
Cameron then shows in~\cite{cameron-perm} that the class of all finite permutations is a \Fraisse\ age and constructs the
generic permutation. It has the form $(\QQ, \Boxed{\prec_1}, \Boxed{\prec_2})$ and it is easy to see that
$(\QQ, \Boxed{\prec_1}) \cong (\QQ, \Boxed{\prec_2}) \cong \QQ$. Let $\Perm$ denote the category of all finite and countably
infinite permutations together with embeddings.

We start the proof by showing that the generic permutation has big Ramsey degrees.
We shall derive this fact from 
$$
  T_{\ChEmb \times \ChEmb}\big((\calA, \calB), (\QQ, \QQ)\big) < \infty,
$$
where $\calA$ and $\calB$ are finite chains and $\QQ$ is the chain of the rationals ordered in the usual way.
This is an obvious consequence of Theorem~\ref{mnmrf.thm.CHAIN-PROD}.
The main idea of the proof is to transfer the big Ramsey degrees
from the category $\ChEmb \times \ChEmb$ to the category $\Perm$ along
a functor $G : \Perm \to \ChEmb \times \ChEmb$. We then use the fact that the generic \textcolor{blue}{2-dimensional} partial order is
quantifier-free definable in the generic permutation to transport the property of having big Ramsey degrees
from the generic permutation to the generic \textcolor{blue}{2-dimensional} partial order.

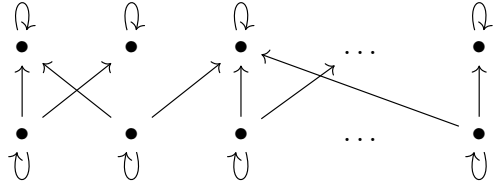
\begin{figure}
  \centering
  \begin{tikzcd}
    {\bullet} \arrow[loop above] & {\bullet} \arrow[loop above] & {\bullet} \arrow[loop above] & \ldots & {\bullet} \arrow[loop above] \\
    {\bullet} \arrow[loop below] \arrow[u] \arrow[ur] & {\bullet} \arrow[loop below] \arrow[ur] \arrow[ul] & {\bullet} \arrow[loop below] \arrow[u] \arrow[ur] & \ldots & {\bullet} \arrow[loop below] \arrow[u] \arrow[ull]
  \end{tikzcd}
  \caption{A binary category}
  \label{mnmrf.fig.bin-cat}
\end{figure}

Consider a finite, bipartite digraph with loops 
where all the arrows go from one class of vertices into the other,
and the out-degree of all the vertices in the first class is~2 (modulo loops), see Fig.~\ref{mnmrf.fig.bin-cat}.
Such a diagraph can be thought of as a category where the loops represent the identity morphisms, and will be referred to as
a \emph{binary category}. Note that all the compositions in a binary category
are trivial since no nonidentity morphisms are composable.
An \emph{amalgamation problem} in a category $\CC$ is a functor $F : \Delta \to \CC$ where $\Delta$ is a binary category,
$F$ takes the top row of $\Delta$ to the same object, and takes the bottom of $\Delta$ to the same object,
see Fig.~\ref{nrt.fig.2}. If $F$ takes the bottom row of $\Delta$ to an object $A$ and the top
row to an object $B$ then the functor $F : \Delta \to \CC$ will be referred to as the \emph{$(A, B)$-diagram in $\CC$}.
An amalgamation problem $F : \Delta \to \CC$ \emph{has a solution in $\CC$} if $F$ has a \emph{compatible cocone in $\CC$}
in the following sense: there is a $C \in \Ob(\CC)$ and for each $\delta \in \Ob(\Delta)$ there is a morphism $f_\delta \in \hom_\CC(F(\delta), C)$ 
such that for all $\gamma, \delta \in \Ob(\Delta)$ and all $e \in \hom_\Delta(\gamma, \delta)$ the following diagram commutes:
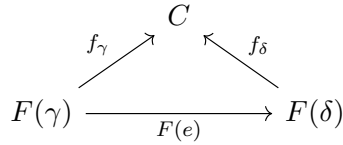
\begin{center}
  \begin{tikzcd}
     & C & \\
     F(\gamma) \arrow[rr, "F(e)"'] \arrow[ur, "f_\gamma"] & & F(\delta) \arrow[ul, "f_\delta"']
  \end{tikzcd}
\end{center}
We then say that $C$ is the \emph{tip of the cocone}.
For a functor $G : \BB \to \CC$ let $G(\BB)$ denote the \emph{image of $G$}, that is, a subcategory of $\CC$
whose objects are of the form $G(B)$, $B \in \Ob(\CC)$, and whose morphisms are of the form $G(f)$, $f \in \hom_\BB(A, B)$
for some $A, B \in \Ob(\BB)$. Note that $G(\BB)$ need not be a full subcategory of $\CC$.

\begin{figure}
\centering
\begin{tikzcd}
  {\bullet} & {\bullet} & {\bullet}
  & & B & B & B
\\
  {\bullet} \arrow[u] \arrow[ur] & {\bullet} \arrow[ur] \arrow[ul] & {\bullet} \arrow[ul] \arrow[u]
  & & A \arrow[u] \arrow[ur] & A \arrow[ur] \arrow[ul] & A \arrow[ul] \arrow[u]
\\
  & \Delta \arrow[rrrr, "F"]  & & & & \CC  
\end{tikzcd}
\caption{An $(A,B)$-diagram in $\CC$ (of shape $\Delta$)}
\label{nrt.fig.2}
\end{figure}
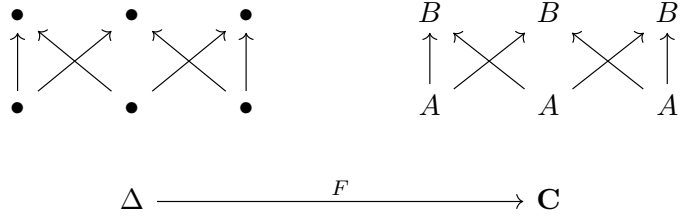

Our main transferring tool is the following statement:

\begin{THM}\label{bigrd.thm.1} (cf.~\cite{masul-bigrd})
  Let $\BB$ and $\CC$ be categories whose every morphism is mono,
  and let $G : \BB \to \CC$ be a faithful functor.
  Let $B \in \Ob(\BB)$ be universal for $\BB$ and let $C \in \Ob(\CC)$ be universal for $G(\BB)$.
  Take any $A \in \Ob(\BB)$ and assume that for every $(A, B)$-diagram $F : \Delta \to \BB$ in $\BB$
  the following holds: if the amalgamation problem $GF : \Delta \to \CC$ has a solution in $\CC$ whose tip is $C$,
  then $F$ has a solution in~$\BB$. Then $T_\BB(A, B) \le T_\CC(G(A), C)$.
\end{THM}

We can now execute the first step of the plan.

\begin{THM}\label{mnmrf.thm.gen-perm-fbrd}
  The generic permutation has big Ramsey degrees.
\end{THM}
\begin{proof}
  Let $\BB = \Perm$ and $\CC = \ChEmb \times \ChEmb$.
  Let $\calQ = (\QQ, \Boxed{\prec_1}, \Boxed{\prec_2})$ be the generic permutation, and
  let $\calQ_1 = (\QQ, \Boxed{\prec_1})$ and $\calQ_2 = (\QQ, \Boxed{\prec_2})$. Clearly,
  $\calQ_1 \cong \calQ_2 \cong \QQ$ so by Theorem~\ref{mnmrf.thm.CHAIN-PROD}:
  \begin{equation}\label{mnmrf.eq.ABQQ}
    T_{\CC}\big((\calA, \calB), (\calQ_1, \calQ_2)\big) < \infty,
  \end{equation}
  for all finite chains $\calA$ and $\calB$. Define $G : \BB \to \CC$ as follows:
  $$
    G\big((A, \Boxed{\sqsubset_1}, \Boxed{\sqsubset_2})\big) = \big((A, \Boxed{\sqsubset_1}), (A, \Boxed{\sqsubset_2})\big)
  $$
  on objects, and $G(f) = (f, f)$ on morphisms. As a notational convenience, if $\calA = (A, \Boxed{\sqsubset_1}, \Boxed{\sqsubset_2})$
  is a permutation then $\calA_1 = (A, \Boxed{\sqsubset_1})$ and $\calA_2 = (A, \Boxed{\sqsubset_2})$ are the corresponding chains.
  Note that $G(\calQ) = (\calQ_1, \calQ_2)$.

  By construction, $G$ is a faithful functor. It is also clear that $(\calQ_1, \calQ_2)$ is universal for $\CC$
  and that $\calQ$ is universal for $\BB$. Take any finite permutation $\calA = (A, \Boxed{\sqsubset_1}, \Boxed{\sqsubset_2})$ and any
  $\big(\calA, \calQ)$-diagram $F : \Delta \to \BB$, and assume that
  the amalgamation problem $GF : \Delta \to \CC$ has a solution in $\CC$ whose tip is $(\calQ_1, \calQ_2)$ and morphisms are of the
  form $(f_i, g_i)$, $i \in I$, for some index set $I$:
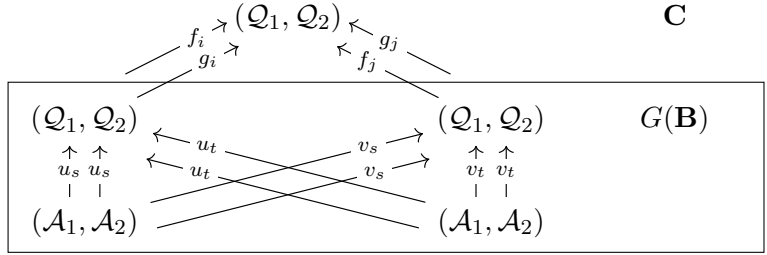
  \begin{center}
    \begin{tikzcd}[execute at end picture={\draw (-5.25,-1.75) rectangle (4.75,0.5);}]
         & (\calQ_1, \calQ_2)  &  &  \CC & \\
         (\calQ_1, \calQ_2) \arrow[ur, shift left=3mm, pos=0.7, "f_i" description] \arrow[ur, shift right=0.5, pos=0.7, "g_i"' description]
         & & 
         (\calQ_1, \calQ_2)  \arrow[ul, shift left=0.5mm, pos=0.7, "f_j" description] \arrow[ul, shift right=3mm, pos=0.6, "g_j"' description]
         & G(\BB)
         \\
         (\calA_1, \calA_2) \arrow[u, shift left=2mm, "u_s" description] \arrow[u, shift right=2mm, "u_s"' description] \arrow[urr, shift left=0.5, pos=0.8, "v_s" description] \arrow[urr, shift right=3mm, pos=0.8, "v_s"' description]
         & & 
         (\calA_1, \calA_2) \arrow[u, shift left=2mm, "v_t" description] \arrow[u, shift right=2mm, "v_t"' description] \arrow[ull, shift left=3mm, pos=0.8, "u_t" description] \arrow[ull, shift right=0.5mm, pos=0.8, "u_t"' description]
    \end{tikzcd}
  \end{center}
  In particular, this means that:
  \begin{equation}\label{mnmrf.eq.comp-cocone}
    f_i \circ u_s = f_j \circ v_s \text{ and } g_i \circ u_s = g_j \circ v_s,
  \end{equation}
  for all $i, j \in I$ and every~$s$.

  Let $\prec_{12}$ denote the lexicographic product of $\prec_1$ and $\prec_2$ on $\QQ \times \QQ$ and let
  $\prec_{21}$ denote the antilexicographic product of $\prec_1$ and $\prec_2$ on $\QQ \times \QQ$. In other words:
  \begin{itemize}
    \item $(a, b) \mathrel{\prec_{12}} (c, d)$ if $a \mathrel{\prec_1} c$, or $a = c$ and $b \mathrel{\prec_2} d$; and
    \item $(a, b) \mathrel{\prec_{21}} (c, d)$ if $b \mathrel{\prec_2} d$, or $b = d$ and $a \mathrel{\prec_1} c$.
  \end{itemize}
  Then $\calQ^* = (\QQ \times \QQ, \Boxed{\prec_{12}}, \Boxed{\prec_{21}}) \in \Ob(\BB)$.
  For each $i \in I$ let $(f_i, g_i) : \QQ \to \QQ \times \QQ$ denote the obvious mapping $(f_i, g_i)(x) = (f_i(x), g_i(x))$.
  It is easy to see that $(f_i, g_i) \in \hom_\BB(\calQ, \calQ^*)$ for all $i \in I$:
  \begin{itemize}
    \item if $a \mathrel{\prec_1} b$ then $f_i(a) \mathrel{\prec_1} f_i(b)$ so $(f_i(a), g_i(a)) \mathrel{\prec_{12}} (f_i(b), g_i(b))$; and
    \item if $a \mathrel{\prec_2} b$ then $g_i(a) \mathrel{\prec_2} g_i(b)$ so $(f_i(a), g_i(a)) \mathrel{\prec_{21}} (f_i(b), g_i(b))$.
  \end{itemize}
  Finally, let us show that $\calQ^*$ together with morphisms $(f_i, g_i)$, $i \in I$, is a compatible cocone for $F$ in $\BB$:
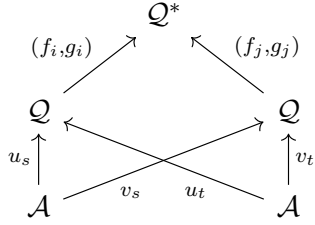
  \begin{center}
    \begin{tikzcd}
      & \calQ^* & \\
      \calQ \arrow[ur, "{(f_i, g_i)}"] & & \calQ \arrow[ul, "{(f_j, g_j)}"'] \\
      \calA \arrow[u, "u_s"] \arrow[urr, "v_s"' near start] & & \calA \arrow[u, "v_t"'] \arrow[ull, "u_t" near start]
    \end{tikzcd}
  \end{center}
  This follows immediately from \eqref{mnmrf.eq.comp-cocone} because $(f_i, g_i) \circ u_s = (f_j, g_j) \circ v_s$
  is nothing but a reformulation of \eqref{mnmrf.eq.comp-cocone}, having in mind that $(f_i, g_i) \circ u_s = (f_i \circ u_s, g_i \circ u_s)$.

  Therefore, $T_\BB(\calA, \calQ) < T_\CC\big((\calA_1, \calA_2), (\calQ_1, \calQ_2)\big) < \infty$
  by Theorem~\ref{bigrd.thm.1} and~\eqref{mnmrf.eq.ABQQ}.
\end{proof}

Let $L$ be a relational language. Recall that a class $\KK$ of $L$-structures is called hereditary if the following holds:
if $\calA \in \KK$ and $\calB$ is an $L$-structure which embeds into $\calA$, then $\calB \in \KK$. 
An $L$-structure $\calU$ is universal for $\KK$ if $\calU$ embeds every $\calA \in \KK$.

\begin{THM}\label{mnmrf.thm.qf-definable} \cite[Theorem~7.1]{masul-rdbas}
  Let $L = \{ R_1, \ldots, R_n \}$ be a finite relational language, let
  $M = \{ S_j : j \in J \}$ be a relational language and let
  $\Phi = \{ \phi_j : j \in J \}$ be a set of quantifier-free $L$-formulas.
  Let $\KK^*$ be a hereditary class of at most countably infinite $L$-structures and let
  $\KK$ be the class of all the $M$-structures which are definable by $\Phi$ in~$\KK^*$.
  Let $\calS^* \in \KK^*$ be universal for $\KK^*$ and let $\calS \in \KK$ be the $M$-structure
  definable in $\calS^*$ by $\Phi$. Then $\calS$ is universal for $\KK$, and
  if $\calS^*$ has finite big Ramsey degrees then so does $\calS$.
\end{THM}

\begin{THM} \sout{\cite{hubicka-param-spaces}}
  The generic \textcolor{blue}{2-dimensional} partial order has big Ramsey degrees.
\end{THM}
\begin{proof}
  Let $\calQ = (\QQ, \Boxed{\prec_1}, \Boxed{\prec_2})$ be the generic permutation and define
  $\preccurlyeq$ on $\QQ$ as follows:
  $$
    a \preccurlyeq b \text{ iff } a = b \lor (a \mathrel{\prec_1} b \land a \mathrel{\prec_2} b),
  $$
  and let us denote by $\phi(a, b)$ the quantifier-free formula on the right of iff.
  
  \medskip

  \sout{Claim. Let $A, B, C \subseteq \QQ$ be three finite (possibly empty) pairwise disjoint sets such that}
  \begin{itemize}
    \item \sout{$a \preccurlyeq c$ for all $a \in A$ and $c \in C$;}
    \item \sout{$b \not\preccurlyeq a$ for all $a \in A$ and $b \in B$; and}
    \item \sout{$c \not\preccurlyeq b$ for all $c \in C$ and $b \in B$.}
  \end{itemize}
  \sout{Then there exists an $s \in \QQ \setminus (A \cup B \cup C)$ such that}
  \begin{itemize}
    \item \sout{$a \preccurlyeq s \preccurlyeq c$ for all $a \in A$ and $c \in C$; and}
    \item \sout{$s \not\preccurlyeq b$ and $b \not\preccurlyeq s$ for all $b \in B$.}
  \end{itemize}
  
  \sout{Proof. Straightforward.}

  \medskip
  
  \sout{Therefore, if} \textcolor{blue}{If} $\KK^*$ denotes the class of all finite or countably infinite permutations,
  then the class of all structures that are definable by $\phi$ in $\KK^*$
  is precisely the class of all finite or countably infinite \textcolor{blue}{2-dimensional} partial orders.
  \sout{Moreover,} \textcolor{blue}{In what follows, we are going to prove that} $(\QQ, \Boxed{\preccurlyeq})$ is
  the generic \textcolor{blue}{2-dimensional} partial order. Since the generic permutation $\calQ$ has big Ramsey degrees, Theorem~\ref{mnmrf.thm.qf-definable}
  immediately yields that the generic \textcolor{blue}{2-dimensional} poset $(\QQ, \Boxed{\preccurlyeq})$ has big Ramsey degrees.
\end{proof}

\textcolor{blue}\bgroup

The \emph{dimension} of a finite partially ordered set $(P, \Boxed{\preceq})$ was introduced in 1941
by Dushnik and Miller~\cite{dushnik-miller-1941} as the least number of linear extensions of $\preceq$
whose intersection is $\preceq$. Equivalently, the dimension of $(P, \Boxed{\preceq})$ is the least
number $d$ of finite linear orders $(L_1, \le_1)$, \ldots, $(L_d, \le_d)$ such that
$(P, \Boxed{\preceq})$ embeds into $(L_1, \le_1) \times \ldots \times (L_d, \le_d)$ (see \cite{ore-1962}).
A finite partially ordered set $(P, \Boxed{\preceq})$ is \emph{at most 2-dimensional} if its dimension is $\le 2$.

\begin{figure}
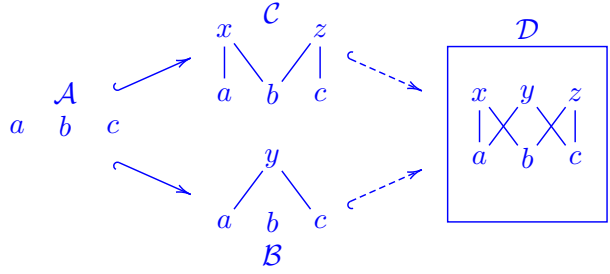

  \centering
  \textcolor{blue}\bgroup
\begin{pgfpicture}
  \pgfsetxvec{\pgfpoint{\acadpgfunit}{0pt}}
  \pgfsetyvec{\pgfpoint{0pt}{\acadpgfunit}}
  \pgfsetlinewidth{\acadpgflinewidth}
  \pgftransformshift{\pgfpointxy{-237.5}{-212.5}}

  \begin{pgfscope}
    \pgfpathmoveto{\pgfpointxy{1050.0}{475.0}}
    \pgfpathlineto{\pgfpointxy{1050.0}{525.0}}
    \pgfusepath{stroke}
  \end{pgfscope}
  \begin{pgfscope}
    \pgfpathmoveto{\pgfpointxy{1065.0}{530.0}}
    \pgfpathlineto{\pgfpointxy{1110.0}{470.0}}
    \pgfusepath{stroke}
  \end{pgfscope}
  \begin{pgfscope}
    \pgfpathmoveto{\pgfpointxy{1140.0}{470.0}}
    \pgfpathlineto{\pgfpointxy{1185.0}{530.0}}
    \pgfusepath{stroke}
  \end{pgfscope}
  \begin{pgfscope}
    \pgfpathmoveto{\pgfpointxy{1200.0}{525.0}}
    \pgfpathlineto{\pgfpointxy{1200.0}{475.0}}
    \pgfusepath{stroke}
  \end{pgfscope}
  \begin{pgfscope}
    \pgfpathmoveto{\pgfpointxy{1185.0}{470.0}}
    \pgfpathlineto{\pgfpointxy{1140.0}{530.0}}
    \pgfusepath{stroke}
  \end{pgfscope}
  \begin{pgfscope}
    \pgfpathmoveto{\pgfpointxy{1110.0}{530.0}}
    \pgfpathlineto{\pgfpointxy{1065.0}{470.0}}
    \pgfusepath{stroke}
  \end{pgfscope}
  \begin{pgfscope}
    \pgfpathmoveto{\pgfpointxy{785.0}{370.0}}
    \pgfpathlineto{\pgfpointxy{740.0}{430.0}}
    \pgfusepath{stroke}
  \end{pgfscope}
  \begin{pgfscope}
    \pgfpathmoveto{\pgfpointxy{710.0}{430.0}}
    \pgfpathlineto{\pgfpointxy{665.0}{370.0}}
    \pgfusepath{stroke}
  \end{pgfscope}
  \begin{pgfscope}
    \pgfpathmoveto{\pgfpointxy{650.0}{575.0}}
    \pgfpathlineto{\pgfpointxy{650.0}{625.0}}
    \pgfusepath{stroke}
  \end{pgfscope}
  \begin{pgfscope}
    \pgfpathmoveto{\pgfpointxy{665.0}{630.0}}
    \pgfpathlineto{\pgfpointxy{710.0}{570.0}}
    \pgfusepath{stroke}
  \end{pgfscope}
  \begin{pgfscope}
    \pgfpathmoveto{\pgfpointxy{740.0}{570.0}}
    \pgfpathlineto{\pgfpointxy{785.0}{630.0}}
    \pgfusepath{stroke}
  \end{pgfscope}
  \begin{pgfscope}
    \pgfpathmoveto{\pgfpointxy{800.0}{625.0}}
    \pgfpathlineto{\pgfpointxy{800.0}{575.0}}
    \pgfusepath{stroke}
  \end{pgfscope}
  \begin{pgfscope}
    \pgfpathmoveto{\pgfpointxy{1000.0}{625.0}}
    \pgfpathlineto{\pgfpointxy{1250.0}{625.0}}
    \pgfpathlineto{\pgfpointxy{1250.0}{350.0}}
    \pgfpathlineto{\pgfpointxy{1000.0}{350.0}}
    \pgfpathclose
    \pgfusepath{stroke}
  \end{pgfscope}
  \begin{pgfscope}
    \pgfpathmoveto{\pgfpointxy{484.487}{556.936}}
    \pgfpathlineto{\pgfpointxy{596.987}{606.936}}
    \pgfusepath{stroke}
  \end{pgfscope}
  \begin{pgfscope}
    \pgfpathmoveto{\pgfpointxy{479.802}{567.478}}
    \pgfpatharcaxes{113.962}{293.962}{\pgfpointxy{5.7681}{0.0}}{\pgfpointxy{0.0}{5.7681}}
    \pgfusepath{stroke}
  \end{pgfscope}
  \begin{pgfscope}
    \pgfpathmoveto{\pgfpointxy{573.978}{603.307}}
    \pgfpatharcaxes{263.962}{293.962}{\pgfpointxy{45.0}{0.0}}{\pgfpointxy{0.0}{45.0}}
    \pgfusepath{stroke}
  \end{pgfscope}
  \begin{pgfscope}
    \pgfpathmoveto{\pgfpointxy{596.987}{606.936}}
    \pgfpatharcaxes{113.962}{143.962}{\pgfpointxy{45.0}{0.0}}{\pgfpointxy{0.0}{45.0}}
    \pgfusepath{stroke}
  \end{pgfscope}
  \begin{pgfscope}
    \pgfpathmoveto{\pgfpointxy{484.487}{443.064}}
    \pgfpathlineto{\pgfpointxy{596.987}{393.064}}
    \pgfusepath{stroke}
  \end{pgfscope}
  \begin{pgfscope}
    \pgfpathmoveto{\pgfpointxy{484.487}{443.064}}
    \pgfpatharcaxes{66.0375}{246.038}{\pgfpointxy{5.7681}{0.0}}{\pgfpointxy{0.0}{5.7681}}
    \pgfusepath{stroke}
  \end{pgfscope}
  \begin{pgfscope}
    \pgfpathmoveto{\pgfpointxy{596.987}{393.064}}
    \pgfpatharcaxes{66.0375}{96.0375}{\pgfpointxy{45.0}{0.0}}{\pgfpointxy{0.0}{45.0}}
    \pgfusepath{stroke}
  \end{pgfscope}
  \begin{pgfscope}
    \pgfpathmoveto{\pgfpointxy{578.875}{407.711}}
    \pgfpatharcaxes{216.038}{246.038}{\pgfpointxy{45.0}{0.0}}{\pgfpointxy{0.0}{45.0}}
    \pgfusepath{stroke}
  \end{pgfscope}
  \begin{pgfscope}
    \pgfsetdash{{2.5pt}{1.5pt}}{0pt}
    \pgfpathmoveto{\pgfpointxy{846.987}{605.564}}
    \pgfpathlineto{\pgfpointxy{959.487}{555.564}}
    \pgfusepath{stroke}
  \end{pgfscope}
  \begin{pgfscope}
    \pgfpathmoveto{\pgfpointxy{851.672}{616.106}}
    \pgfpatharcaxes{66.0375}{246.038}{\pgfpointxy{5.7681}{0.0}}{\pgfpointxy{0.0}{5.7681}}
    \pgfusepath{stroke}
  \end{pgfscope}
  \begin{pgfscope}
    \pgfpathmoveto{\pgfpointxy{959.487}{555.564}}
    \pgfpatharcaxes{66.0375}{96.0375}{\pgfpointxy{45.0}{0.0}}{\pgfpointxy{0.0}{45.0}}
    \pgfusepath{stroke}
  \end{pgfscope}
  \begin{pgfscope}
    \pgfpathmoveto{\pgfpointxy{941.375}{570.211}}
    \pgfpatharcaxes{216.038}{246.038}{\pgfpointxy{45.0}{0.0}}{\pgfpointxy{0.0}{45.0}}
    \pgfusepath{stroke}
  \end{pgfscope}
  \begin{pgfscope}
    \pgfsetdash{{2.5pt}{1.5pt}}{0pt}
    \pgfpathmoveto{\pgfpointxy{846.987}{381.936}}
    \pgfpathlineto{\pgfpointxy{959.487}{431.936}}
    \pgfusepath{stroke}
  \end{pgfscope}
  \begin{pgfscope}
    \pgfpathmoveto{\pgfpointxy{846.987}{381.936}}
    \pgfpatharcaxes{113.962}{293.962}{\pgfpointxy{5.7681}{0.0}}{\pgfpointxy{0.0}{5.7681}}
    \pgfusepath{stroke}
  \end{pgfscope}
  \begin{pgfscope}
    \pgfpathmoveto{\pgfpointxy{936.478}{428.307}}
    \pgfpatharcaxes{263.962}{293.962}{\pgfpointxy{45.0}{0.0}}{\pgfpointxy{0.0}{45.0}}
    \pgfusepath{stroke}
  \end{pgfscope}
  \begin{pgfscope}
    \pgfpathmoveto{\pgfpointxy{959.487}{431.936}}
    \pgfpatharcaxes{113.962}{143.962}{\pgfpointxy{45.0}{0.0}}{\pgfpointxy{0.0}{45.0}}
    \pgfusepath{stroke}
  \end{pgfscope}
  \pgftext[at={\pgfpointxy{1050.0}{450.0}}]{$a$}
  \pgftext[at={\pgfpointxy{1125.0}{450.0}}]{$b$}
  \pgftext[at={\pgfpointxy{1200.0}{450.0}}]{$c$}
  \pgftext[at={\pgfpointxy{1050.0}{550.0}}]{$x$}
  \pgftext[at={\pgfpointxy{1125.0}{550.0}}]{$y$}
  \pgftext[at={\pgfpointxy{1200.0}{550.0}}]{$z$}
  \pgftext[at={\pgfpointxy{650.0}{350.0}}]{$a$}
  \pgftext[at={\pgfpointxy{725.0}{350.0}}]{$b$}
  \pgftext[at={\pgfpointxy{800.0}{350.0}}]{$c$}
  \pgftext[at={\pgfpointxy{725.0}{450.0}}]{$y$}
  \pgftext[at={\pgfpointxy{650.0}{550.0}}]{$a$}
  \pgftext[at={\pgfpointxy{725.0}{550.0}}]{$b$}
  \pgftext[at={\pgfpointxy{800.0}{550.0}}]{$c$}
  \pgftext[at={\pgfpointxy{650.0}{650.0}}]{$x$}
  \pgftext[at={\pgfpointxy{800.0}{650.0}}]{$z$}
  \pgftext[at={\pgfpointxy{325.0}{500.0}}]{$a$}
  \pgftext[at={\pgfpointxy{400.0}{500.0}}]{$b$}
  \pgftext[at={\pgfpointxy{475.0}{500.0}}]{$c$}
  \pgftext[bottom,at={\pgfpointxy{1125.0}{637.0}}]{$\calD$}
  \pgftext[top,at={\pgfpointxy{725.0}{313.0}}]{$\calB$}
  \pgftext[bottom,at={\pgfpointxy{725.0}{662.0}}]{$\calC$}
  \pgftext[bottom,at={\pgfpointxy{400.0}{537.0}}]{$\calA$}
\end{pgfpicture}
    \caption{A counterexample for the amalgamation property for the class of all at most 2-dimensional finite linear orders}
  \egroup
  \label{mnmrf-c.fig.no-amalg}
\end{figure}

The class of all at most 2-dimensional finite linear orders is not a \Fraisse\ age. Namely,
it is a well-known fact that one of the fundamental properties of a \Fraisse\ age is the
\emph{amalgamation property (AP)}, which is a requirement that
every span $\calB \overset f \hookleftarrow \calA \overset g \hookrightarrow \calC$ with $\calA, \calB, \calC \in \KK$
completes to a square with $\calD \in \KK$:
\begin{center}
  \begin{tikzcd}
    \calC \arrow[r, hookrightarrow, "g'"] & \calD\\
    \calA \arrow[r, hookrightarrow, "f"'] \arrow[u, hookrightarrow, "g"] & \calB \arrow[u, hookrightarrow, "f'"']
  \end{tikzcd}
\end{center}
This property fails for the class of all at most 2-dimensional finite linear orders: the span
$\calB \hookleftarrow \calA \hookrightarrow \calC$ in Fig.~\ref{mnmrf-c.fig.no-amalg}
can easily be completed to a square in the class of \emph{all} partially ordered sets;
however, every candidate for $\calD$ embeds a crown on six elements, so the dimension of each such $\calD$ is at least~3.

Nevertheless, we shall prove that the reduct of the generic permutation we are interested in is the \emph{weak \Fraisse\ limit} of this class, and hence a generic structure.

A class $\KK$ of finite relational structures over the same relational language has the \emph{joint embedding property (JEP)}
if every pair of structures from $\KK$ embeds into a structure from $\KK$; and it has the \emph{weak
amalgamation property (WAP)} if for every $\calA \in \KK$ there is a $\overline\calA \in \KK$ and
an embedding $e : \calA \hookrightarrow \overline\calA$ such that for every choice of $\calB, \calC \in \KK$
and embeddings $f : \overline\calA \hookrightarrow \calB$ and $g : \overline\calA \hookrightarrow \calC$
there exist a $\calD \in \KK$ and embeddings $f' : \calB \hookrightarrow \calD$ and
$g' : \calC \hookrightarrow \calD$ satisfying $f' \circ f \circ e = g' \circ g \circ e$:
\begin{center}
  \begin{tikzcd}
                          & \calC \arrow[r, hookrightarrow, "g'"] & \calD\\
    \calA \arrow[r, hookrightarrow, "e"'] & \overline \calA \arrow[r, hookrightarrow, "f"'] \arrow[u, hookrightarrow, "g"] & \calB \arrow[u, hookrightarrow, "f'"']
  \end{tikzcd}
\end{center}

The importance of the weak amalgamation property was first recognized by Ivanov \cite{ivanov-1999}
and later independently by Kechris and Rosendal \cite{kechris-rosendal-2007} in their study of generic automorphisms.
The importance of the weak amalgamation property with respect to the existence of generic structures was
recognized independently by Kubi\'s \cite{kubis-2022}, and Panagiotopoulos and Tent \cite{panag-tent-2022}.
Namely, if $\KK$ is an age of a countable relational structure with (JEP) and (WAP), then there is a
unique (up to isomorphism) countable structure $\Omega$ on $\omega = \{0, 1, 2, \ldots\}$ with $\Age(\Omega) = \KK$
whose isomorphism class is comeager in a naturally defined Baire space of all structures on $\omega$ whose age is $\KK$,
(see~\cite{panag-tent-2022} for details). We call $\Omega$ the \emph{weak \Fraisse\ limit of $\KK$}
and say that $\Omega$ is \emph{generic} among countable structures whose age is $\KK$.

In order to verify the weak amalgamation property,
it suffices to show that for every structure $\calA$ there is a well-behaved structure $\overline\calA$ into which
$\calA$ embeds, and over which we can weakly amalgamate. As we shall se below, in the case of finite posets of dimension $\le 2$
the well-behaved structures $\overline\calA$ are of the form $(A, \Boxed{\le^A}) \times (B, \Boxed{\le^B})$
where $(A, \Boxed{\le^A})$ and $(B, \Boxed{\le^B})$ are finite linear orders. 
We shall now introduce a bit of notation and establish several elementary facts about posets of this form,
which will be used in the main argument below.

Let $(A, \Boxed{\le^A})$ and $(B, \Boxed{\le^B})$ be finite linear orders. Then $(A, \Boxed{\le^A}) \times (B, \Boxed{\le^B})$
is a finite partially ordered set $(A \times B, \Boxed{\sqsubseteq^{A \times B}})$ where
$$
  (a_1, b_1) \mathrel{\sqsubseteq^{A \times B}} (a_2, b_2) \text{ iff } a_1 \le^A a_2 \land b_1 \le^B b_2.
$$
Let $\lex^{A \times B}$ and $\alex^{A \times B}$ denote the \emph{lexicographic} and \emph{antilexicographic} order
on $A \times B$, respectively, which are linear orders defined as follows:
\begin{align*}
  (a_1, b_1) \mathrel{\lex^{A \times B}} (a_2, b_2) \text{ iff }
  & a_1 <^A a_2 \lor (a_1 = a_2 \land b_1 \le^B b_2), \text{ and}\\
  (a_1, b_1) \mathrel{\alex^{A \times B}} (a_2, b_2) \text{ iff }
  & b_1 <^B b_2 \lor (b_1 = b_2 \land a_1 \le^A a_2).
\end{align*}
Then it is easy to show that:
$$
  \Boxed{\sqsubseteq^{A \times B}} = (\Boxed{\lex^{A \times B}}) \cap (\Boxed{\alex^{A \times B}}).
$$

\begin{LEMA}
  Let $(A, \Boxed{\le^A})$ and $(B, \Boxed{\le^B})$ be nonempty finite linear orders, and let
  $(A \times B, \Boxed{\sqsubseteq^{A \times B}}) = (A, \Boxed{\le^A}) \times (B, \Boxed{\le^B})$.
  Let $\le_1^{A \times B}$ and $\le_2^{A \times B}$ be linear orders on $A \times B$ such that
  $\Boxed{\sqsubseteq^{A \times B}} = (\Boxed{\le_1^{A \times B}}) \cap (\Boxed{\le_2^{A \times B}})$. Then
  $\{\Boxed{\le_1^{A \times B}}, \Boxed{\le_2^{A \times B}}\} = \{\Boxed{\lex^{A \times B}}, \Boxed{\alex^{A \times B}}\}$.
\end{LEMA}
\begin{proof}
  The proof is by induction on $|B|$. If $|B| = 1$ then $(A \times B, \Boxed{\sqsubseteq^{A \times B}}) \cong (A, \Boxed{\le^A})$
  and:
  $$
    (\le_1^{A \times B}) = (\le_2^{A \times B}) = (\lex^{A \times B}) = (\alex^{A \times B}).
  $$
  For the induction step, assume that the statement is true whenever $|B| = n - 1$ and let $(B, \Boxed{\le^B})$ be an $n$-element
  linearly ordered set, say $B = \{1, 2, \ldots, n\}$ with the usual ordering of the integers, where $n \ge 2$.
  For the sake of simplicity, let $A = \{a, b, c, \ldots, z\}$ so that
  $$
    A \times B = \{a_1, b_1, \ldots, z_1, \, a_2, b_2, \ldots, z_2, \, \ldots, \, a_n, b_n, \ldots, z_n\},
  $$
  see Fig.~\ref{corrigendum.fig.AxB}.

  \begin{figure}
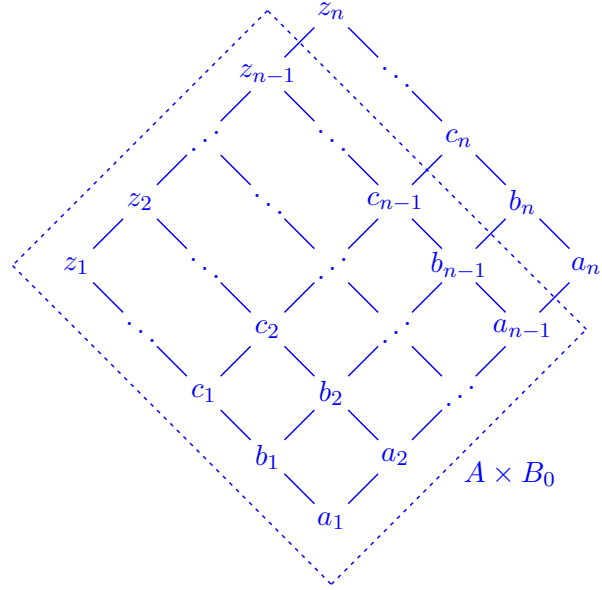

    \centering
	\textcolor{blue}\bgroup
\begin{pgfpicture}
  \pgfsetxvec{\pgfpoint{\acadpgfunit}{0pt}}
  \pgfsetyvec{\pgfpoint{0pt}{\acadpgfunit}}
  \pgfsetlinewidth{\acadpgflinewidth}
  \pgftransformshift{\pgfpointxy{50.0}{-112.5}}

  \begin{pgfscope}
    \pgfpathmoveto{\pgfpointxy{126.517}{676.517}}
    \pgfpathlineto{\pgfpointxy{173.483}{723.483}}
    \pgfusepath{stroke}
  \end{pgfscope}
  \begin{pgfscope}
    \pgfpathmoveto{\pgfpointxy{226.517}{776.517}}
    \pgfpathlineto{\pgfpointxy{273.483}{823.483}}
    \pgfusepath{stroke}
  \end{pgfscope}
  \begin{pgfscope}
    \pgfpathmoveto{\pgfpointxy{326.517}{876.517}}
    \pgfpathlineto{\pgfpointxy{373.483}{923.483}}
    \pgfusepath{stroke}
  \end{pgfscope}
  \begin{pgfscope}
    \pgfpathmoveto{\pgfpointxy{426.517}{976.517}}
    \pgfpathlineto{\pgfpointxy{473.483}{1023.48}}
    \pgfusepath{stroke}
  \end{pgfscope}
  \begin{pgfscope}
    \pgfpathmoveto{\pgfpointxy{526.517}{1023.48}}
    \pgfpathlineto{\pgfpointxy{573.483}{976.517}}
    \pgfusepath{stroke}
  \end{pgfscope}
  \begin{pgfscope}
    \pgfpathmoveto{\pgfpointxy{226.517}{523.483}}
    \pgfpathlineto{\pgfpointxy{273.483}{476.517}}
    \pgfusepath{stroke}
  \end{pgfscope}
  \begin{pgfscope}
    \pgfpathmoveto{\pgfpointxy{326.517}{476.517}}
    \pgfpathlineto{\pgfpointxy{373.483}{523.483}}
    \pgfusepath{stroke}
  \end{pgfscope}
  \begin{pgfscope}
    \pgfpathmoveto{\pgfpointxy{426.517}{576.517}}
    \pgfpathlineto{\pgfpointxy{473.483}{623.483}}
    \pgfusepath{stroke}
  \end{pgfscope}
  \begin{pgfscope}
    \pgfpathmoveto{\pgfpointxy{526.517}{676.517}}
    \pgfpathlineto{\pgfpointxy{573.483}{723.483}}
    \pgfusepath{stroke}
  \end{pgfscope}
  \begin{pgfscope}
    \pgfpathmoveto{\pgfpointxy{626.517}{776.517}}
    \pgfpathlineto{\pgfpointxy{673.483}{823.483}}
    \pgfusepath{stroke}
  \end{pgfscope}
  \begin{pgfscope}
    \pgfpathmoveto{\pgfpointxy{726.517}{823.483}}
    \pgfpathlineto{\pgfpointxy{773.483}{776.517}}
    \pgfusepath{stroke}
  \end{pgfscope}
  \begin{pgfscope}
    \pgfpathmoveto{\pgfpointxy{773.483}{723.483}}
    \pgfpathlineto{\pgfpointxy{726.517}{676.517}}
    \pgfusepath{stroke}
  \end{pgfscope}
  \begin{pgfscope}
    \pgfpathmoveto{\pgfpointxy{673.483}{623.483}}
    \pgfpathlineto{\pgfpointxy{626.517}{576.517}}
    \pgfusepath{stroke}
  \end{pgfscope}
  \begin{pgfscope}
    \pgfpathmoveto{\pgfpointxy{573.483}{523.483}}
    \pgfpathlineto{\pgfpointxy{526.517}{476.517}}
    \pgfusepath{stroke}
  \end{pgfscope}
  \begin{pgfscope}
    \pgfpathmoveto{\pgfpointxy{473.483}{423.483}}
    \pgfpathlineto{\pgfpointxy{426.517}{376.517}}
    \pgfusepath{stroke}
  \end{pgfscope}
  \begin{pgfscope}
    \pgfpathmoveto{\pgfpointxy{426.517}{323.483}}
    \pgfpathlineto{\pgfpointxy{473.483}{276.517}}
    \pgfusepath{stroke}
  \end{pgfscope}
  \begin{pgfscope}
    \pgfpathmoveto{\pgfpointxy{526.517}{276.517}}
    \pgfpathlineto{\pgfpointxy{573.483}{323.483}}
    \pgfusepath{stroke}
  \end{pgfscope}
  \begin{pgfscope}
    \pgfpathmoveto{\pgfpointxy{626.517}{376.517}}
    \pgfpathlineto{\pgfpointxy{673.483}{423.483}}
    \pgfusepath{stroke}
  \end{pgfscope}
  \begin{pgfscope}
    \pgfpathmoveto{\pgfpointxy{726.517}{476.517}}
    \pgfpathlineto{\pgfpointxy{773.483}{523.483}}
    \pgfusepath{stroke}
  \end{pgfscope}
  \begin{pgfscope}
    \pgfpathmoveto{\pgfpointxy{826.517}{576.517}}
    \pgfpathlineto{\pgfpointxy{873.483}{623.483}}
    \pgfusepath{stroke}
  \end{pgfscope}
  \begin{pgfscope}
    \pgfpathmoveto{\pgfpointxy{873.483}{676.517}}
    \pgfpathlineto{\pgfpointxy{826.517}{723.483}}
    \pgfusepath{stroke}
  \end{pgfscope}
  \begin{pgfscope}
    \pgfpathmoveto{\pgfpointxy{673.483}{876.517}}
    \pgfpathlineto{\pgfpointxy{626.517}{923.483}}
    \pgfusepath{stroke}
  \end{pgfscope}
  \begin{pgfscope}
    \pgfpathmoveto{\pgfpointxy{373.483}{376.517}}
    \pgfpathlineto{\pgfpointxy{326.517}{423.483}}
    \pgfusepath{stroke}
  \end{pgfscope}
  \begin{pgfscope}
    \pgfpathmoveto{\pgfpointxy{173.483}{576.517}}
    \pgfpathlineto{\pgfpointxy{126.517}{623.483}}
    \pgfusepath{stroke}
  \end{pgfscope}
  \begin{pgfscope}
    \pgfpathmoveto{\pgfpointxy{226.517}{723.483}}
    \pgfpathlineto{\pgfpointxy{273.483}{676.517}}
    \pgfusepath{stroke}
  \end{pgfscope}
  \begin{pgfscope}
    \pgfpathmoveto{\pgfpointxy{326.517}{623.483}}
    \pgfpathlineto{\pgfpointxy{373.483}{576.517}}
    \pgfusepath{stroke}
  \end{pgfscope}
  \begin{pgfscope}
    \pgfpathmoveto{\pgfpointxy{426.517}{523.483}}
    \pgfpathlineto{\pgfpointxy{473.483}{476.517}}
    \pgfusepath{stroke}
  \end{pgfscope}
  \begin{pgfscope}
    \pgfpathmoveto{\pgfpointxy{526.517}{423.483}}
    \pgfpathlineto{\pgfpointxy{573.483}{376.517}}
    \pgfusepath{stroke}
  \end{pgfscope}
  \begin{pgfscope}
    \pgfpathmoveto{\pgfpointxy{326.517}{823.483}}
    \pgfpathlineto{\pgfpointxy{373.483}{776.517}}
    \pgfusepath{stroke}
  \end{pgfscope}
  \begin{pgfscope}
    \pgfpathmoveto{\pgfpointxy{426.517}{723.483}}
    \pgfpathlineto{\pgfpointxy{473.483}{676.517}}
    \pgfusepath{stroke}
  \end{pgfscope}
  \begin{pgfscope}
    \pgfpathmoveto{\pgfpointxy{526.517}{623.483}}
    \pgfpathlineto{\pgfpointxy{573.483}{576.517}}
    \pgfusepath{stroke}
  \end{pgfscope}
  \begin{pgfscope}
    \pgfpathmoveto{\pgfpointxy{626.517}{523.483}}
    \pgfpathlineto{\pgfpointxy{673.483}{476.517}}
    \pgfusepath{stroke}
  \end{pgfscope}
  \begin{pgfscope}
    \pgfpathmoveto{\pgfpointxy{773.483}{576.517}}
    \pgfpathlineto{\pgfpointxy{726.517}{623.483}}
    \pgfusepath{stroke}
  \end{pgfscope}
  \begin{pgfscope}
    \pgfpathmoveto{\pgfpointxy{673.483}{676.517}}
    \pgfpathlineto{\pgfpointxy{626.517}{723.483}}
    \pgfusepath{stroke}
  \end{pgfscope}
  \begin{pgfscope}
    \pgfpathmoveto{\pgfpointxy{573.483}{776.517}}
    \pgfpathlineto{\pgfpointxy{526.517}{823.483}}
    \pgfusepath{stroke}
  \end{pgfscope}
  \begin{pgfscope}
    \pgfpathmoveto{\pgfpointxy{473.483}{876.517}}
    \pgfpathlineto{\pgfpointxy{426.517}{923.483}}
    \pgfusepath{stroke}
  \end{pgfscope}
  \begin{pgfscope}
    \pgfsetdash{{1.5pt}{2pt}}{0pt}
    \pgfpathmoveto{\pgfpointxy{900.0}{550.0}}
    \pgfpathlineto{\pgfpointxy{400.0}{1050.0}}
    \pgfusepath{stroke}
  \end{pgfscope}
  \begin{pgfscope}
    \pgfsetdash{{1.5pt}{2pt}}{0pt}
    \pgfpathmoveto{\pgfpointxy{400.0}{1050.0}}
    \pgfpathlineto{\pgfpointxy{0.0}{650.0}}
    \pgfusepath{stroke}
  \end{pgfscope}
  \begin{pgfscope}
    \pgfsetdash{{1.5pt}{2pt}}{0pt}
    \pgfpathmoveto{\pgfpointxy{0.0}{650.0}}
    \pgfpathlineto{\pgfpointxy{500.0}{150.0}}
    \pgfusepath{stroke}
  \end{pgfscope}
  \begin{pgfscope}
    \pgfsetdash{{1.5pt}{2pt}}{0pt}
    \pgfpathmoveto{\pgfpointxy{500.0}{150.0}}
    \pgfpathlineto{\pgfpointxy{900.0}{550.0}}
    \pgfusepath{stroke}
  \end{pgfscope}
  \pgftext[at={\pgfpointxy{500.0}{250.0}}]{$a_1$}
  \pgftext[at={\pgfpointxy{400.0}{350.0}}]{$b_1$}
  \pgftext[at={\pgfpointxy{300.0}{450.0}}]{$c_1$}
  \pgftext[at={\pgfpointxy{100.0}{650.0}}]{$z_1$}
  \pgftext[at={\pgfpointxy{201.525}{559.913}}]{$\ddots$}
  \pgftext[at={\pgfpointxy{600.0}{350.0}}]{$a_2$}
  \pgftext[at={\pgfpointxy{500.0}{450.0}}]{$b_2$}
  \pgftext[at={\pgfpointxy{400.0}{550.0}}]{$c_2$}
  \pgftext[at={\pgfpointxy{200.0}{750.0}}]{$z_2$}
  \pgftext[at={\pgfpointxy{301.525}{659.913}}]{$\ddots$}
  \pgftext[at={\pgfpointxy{700.763}{459.151}}]{$\iddots$}
  \pgftext[at={\pgfpointxy{600.763}{559.151}}]{$\iddots$}
  \pgftext[at={\pgfpointxy{500.763}{659.151}}]{$\iddots$}
  \pgftext[at={\pgfpointxy{300.763}{859.151}}]{$\iddots$}
  \pgftext[at={\pgfpointxy{800.0}{550.0}}]{$a_{n-1}$}
  \pgftext[at={\pgfpointxy{700.0}{650.0}}]{$b_{n-1}$}
  \pgftext[at={\pgfpointxy{600.0}{750.0}}]{$c_{n-1}$}
  \pgftext[at={\pgfpointxy{400.0}{950.0}}]{$z_{n-1}$}
  \pgftext[at={\pgfpointxy{501.525}{859.913}}]{$\ddots$}
  \pgftext[at={\pgfpointxy{900.0}{650.0}}]{$a_n$}
  \pgftext[at={\pgfpointxy{800.0}{750.0}}]{$b_n$}
  \pgftext[at={\pgfpointxy{700.0}{850.0}}]{$c_n$}
  \pgftext[at={\pgfpointxy{500.0}{1050.0}}]{$z_n$}
  \pgftext[at={\pgfpointxy{601.525}{959.913}}]{$\ddots$}
  \pgftext[at={\pgfpointxy{401.525}{759.913}}]{$\ddots$}
  \pgftext[top,left,at={\pgfpointxy{708.0}{342.0}}]{$A \times B_0$}
\end{pgfpicture}
	  \caption{$(A \times B, \Boxed{\sqsubseteq^{A \times B}})$ in the proof of Lemma A}
	\egroup
	\label{corrigendum.fig.AxB}
  \end{figure}

  Let $\le_1^{A \times B}$ and $\le_2^{A \times B}$ be linear orders on $A \times B$ such that
  $$
    \Boxed{\sqsubseteq^{A \times B}} = (\Boxed{\le_1^{A \times B}}) \cap (\Boxed{\le_2^{A \times B}}).
  $$
  Put $B_0 = \{1, 2, \ldots, n-1\}$ and let $\le_1^{A \times B_0}$ and $\le_2^{A \times B_0}$ be the restrictions of
  $\le_1^{A \times B}$ and $\le_2^{A \times B}$, respectively, to $A \times B_0$. Then 
  $$
    \Boxed{\sqsubseteq^{A \times B_0}} = (\Boxed{\le_1^{A \times B_0}}) \cap (\Boxed{\le_2^{A \times B_0}}),
  $$
  so, by the induction hypothesis, 
  $$
    \{\Boxed{\le_1^{A \times B_0}}, \Boxed{\le_2^{A \times B_0}}\} = \{\Boxed{\lex^{A \times B_0}}, \Boxed{\alex^{A \times B_0}}\}.
  $$
  Without loss of generality we can assume that:
  \begin{align*}
    \Boxed{\le_1^{A \times B_0}} = \Boxed{\alex^{A \times B_0}} = a_1 b_1 \ldots z_1 \; a_2 b_2 \ldots z_2 \; \ldots \; a_{n-1} b_{n-1} \ldots z_{n-1};\\
    \Boxed{\le_2^{A \times B_0}} = \Boxed{\lex^{A \times B_0}} = a_1 a_2 \ldots a_{n-1} \; b_1 b_2 \ldots b_{n-1} \; \ldots \; z_1 z_2 \ldots z_{n-1}.
  \end{align*}
  Let us show that there is exactly one way to reconstruct $\Boxed{\le_1^{A \times B}}$ and $\Boxed{\le_2^{A \times B}}$
  from $\Boxed{\le_1^{A \times B_0}}$ and $\Boxed{\le_2^{A \times B_0}}$.
  
  Let us start with $a_n$. Since $a_{n-1} \mathrel{\sqsubseteq^{A \times B}} a_n$ it follows that
  $a_{n-1} \mathrel{\le_1^{A \times B}} a_n$ and $a_{n-1} \mathrel{\le_2^{A \times B}} a_n$.
  By the induction hypothesis we have that $b_1 \mathrel{\le_1^{A \times B_0}} a_{n-1}$, so
  $b_1 \mathrel{\le_1^{A \times B}} a_{n}$. Then $a_n \mathrel{\le_2^{A \times B}} b_1$ because
  $a_n$ and $b_1$ are incomparable with respect to $\Boxed{\sqsubseteq^{A \times B}} = (\Boxed{\le_1^{A \times B}}) \cap (\Boxed{\le_2^{A \times B}})$.
  Therefore,
  $$
    a_{n-1} \mathrel{\le_2^{A \times B}} a_n \mathrel{\le_2^{A \times B}} b_1,
  $$
  whence follows that $\Boxed{\le_2^{A \times B}}$ contains the following linear order:
  $$
    \Boxed{\le_2^{A \times B}}: \; a_1 a_2 \ldots a_{n-1} \underset\uparrow{a_n} \; b_1 b_2 \ldots b_{n-1} \; \ldots \; z_1 z_2 \ldots z_{n-1}
  $$
  Consequently, $a_n \mathrel{\le_2^{A \times B}} z_{n-1}$, so $z_{n-1} \mathrel{\le_1^{A \times B}} a_n$ because
  $a_n$ and $z_{n-1}$ are incomparable with respect to $\Boxed{\sqsubseteq^{A \times B}}$. This implies that
  $\Boxed{\le_1^{A \times B}}$ contains the following linear order:
  $$
    \Boxed{\le_1^{A \times B}}: \; a_1 b_1 \ldots z_1 \; a_2 b_2 \ldots z_2 \; \ldots \; a_{n-1} b_{n-1} \ldots z_{n-1} \; a_n.
  $$
  But, $a_n \Boxed{\sqsubseteq^{A \times B}} b_n \Boxed{\sqsubseteq^{A \times B}} \ldots \Boxed{\sqsubseteq^{A \times B}} z_n$, so
  we can fully describe $\Boxed{\le_1^{A \times B}}$ as:
  $$
    \Boxed{\le_1^{A \times B}} = a_1 b_1 \ldots z_1 \; a_2 b_2 \ldots z_2 \; \ldots \; a_{n-1} b_{n-1} \ldots z_{n-1} \; a_n b_n \ldots z_n = \Boxed{\alex^{A \times B}}.
  $$

  Let us complete the proof that $\Boxed{\le_2^{A \times B}} = \Boxed{\lex^{A \times B}}$. We have already placed $a_n$, so let
  us now move on to $b_n$. Since $b_{n-1} \mathrel{\sqsubseteq^{A \times B}} b_n$ it follows that
  $b_{n-1} \mathrel{\le_1^{A \times B}} b_n$ and $b_{n-1} \mathrel{\le_2^{A \times B}} b_n$.
  By the induction hypothesis we have that $c_1 \mathrel{\le_1^{A \times B_0}} b_{n-1}$, so
  $c_1 \mathrel{\le_1^{A \times B}} b_{n}$. Then $b_n \mathrel{\le_2^{A \times B}} c_1$ because
  $b_n$ and $c_1$ are incomparable with respect to $\Boxed{\sqsubseteq^{A \times B}}$. Therefore,
  $$
    b_{n-1} \mathrel{\le_2^{A \times B}} b_n \mathrel{\le_2^{A \times B}} c_1,
  $$
  whence follows that $\Boxed{\le_2^{A \times B}}$ contains the following linear order:
  $$
    \Boxed{\le_2^{A \times B}}: \; a_1 a_2 \ldots a_{n-1} a_n \; b_1 b_2 \ldots b_{n-1} \underset\uparrow{b_n} \; c_1 \ldots \; z_1 z_2 \ldots z_{n-1}.
  $$
  We can now repeat the argument to place $c_n$ immediately after $c_{n-1}$, \ldots, $y_n$ immediately after $y_{n-1}$
  showing, thus, that $\Boxed{\le_2^{A \times B}}$ contains the following linear order:
  $$
    \Boxed{\le_2^{A \times B}}: \; a_1 a_2 \ldots a_n \; b_1 b_2 \ldots b_n \; c_1 c_2 \ldots c_n \; \ldots \; 
	y_1 y_2 \ldots y_n \; z_1 z_2 \ldots z_{n-1}.
  $$
  The fact that $z_n$ is the largest element in $(A \times B, \Boxed{\sqsubseteq^{A \times B}})$ places $z_n$ at the very end of the above list,
  showing that $\Boxed{\le_2^{A \times B}} = \Boxed{\lex^{A \times B}}$.
\end{proof}

\begin{LEMB}
  Let $(A, \Boxed{\le^A})$, $(B, \Boxed{\le^B})$, $(C, \Boxed{\le^C})$ and $(D, \Boxed{\le^D})$
  be nonempty finite linear orders, and let $f : (A, \Boxed{\le^A}) \times (B, \Boxed{\le^B}) \hookrightarrow
  (C, \Boxed{\le^C}) \times (D, \Boxed{\le^D})$ be an embedding. Then
  \begin{itemize}
  \item $f : (A \times B, \Boxed{\lex^{A \times B}}, \Boxed{\alex^{A \times B}}) \hookrightarrow (C \times D, \Boxed{\lex^{C \times D}}, \Boxed{\alex^{C \times D}})$ is an embedding, or
  \item $f : (A \times B, \Boxed{\lex^{A \times B}}, \Boxed{\alex^{A \times B}}) \hookrightarrow (C \times D, \Boxed{\alex^{C \times D}}, \Boxed{\lex^{C \times D}})$ is an embedding.
  \end{itemize}
\end{LEMB}
\begin{proof}
  Let $f : (A, \Boxed{\le^A}) \times (B, \Boxed{\le^B}) \hookrightarrow
  (C, \Boxed{\le^C}) \times (D, \Boxed{\le^D})$ be an embedding. In other words,
  $f : (A \times B, \Boxed{\sqsubseteq^{A \times B}}) \hookrightarrow (C \times D, \Boxed{\sqsubseteq^{C \times D}})$ is an embedding.
  Define linear orders $\le_1^{A \times B}$ and $\le_2^{A \times B}$ on $A \times B$ as follows:
  \begin{align*}
    (a_1, b_1) \mathrel{\le_1^{A \times B}} (a_2, b_2) &\text{ iff } f(a_1, b_1) \mathrel{\lex^{C \times D}} f(a_2, b_2), \text{ and}\\
    (a_1, b_1) \mathrel{\le_2^{A \times B}} (a_2, b_2) &\text{ iff } f(a_1, b_1) \mathrel{\alex^{C \times D}} f(a_2, b_2).
  \end{align*}
  Then $f : (A \times B, \Boxed{\le_1^{A \times B}}, \Boxed{\le_2^{A \times B}}) \hookrightarrow (C \times D, \Boxed{\lex^{C \times D}}, \Boxed{\alex^{C \times D}})$ is an embedding.
  On the other hand, it is easy to check that:
  $$
    \Boxed{\sqsubseteq^{A \times B}} = (\Boxed{\le_1^{A \times B}}) \cap (\Boxed{\le_2^{A \times B}}),
  $$
  so, by Lemma~A, it must be the case that
  $\{\Boxed{\le_1^{A \times B}}, \Boxed{\le_2^{A \times B}}\} = \{\Boxed{\lex^{A \times B}}, \Boxed{\alex^{A \times B}}\}$.
\end{proof}

Following \cite{kubis-2022}, in order to show that a countable structure $\Omega$ is a weak \Fraisse\ limit
of an age $\KK$, and hence a generic structure, it suffices to show that $\Omega$ embeds every structure from $\KK$,
and that it is \emph{weakly injective} in the following sense: for every $\calA \in \KK$ and every embedding
$f : \calA \hookrightarrow \Omega$ there is an embedding $e : \calA \hookrightarrow \overline\calA$ with $\overline \calA \in \KK$
such that for every embedding $g : \overline\calA \hookrightarrow \calB$ with $\calB \in \KK$ there is an
embedding $h : \calB \hookrightarrow \Omega$ such that $h \circ g \circ e = f$:
\begin{center}
  \begin{tikzcd}
    \calA \arrow[d, hookrightarrow, "f"'] \arrow[r, hookrightarrow, "e"] & \overline\calA \arrow[r, hookrightarrow, "g"] & \calB \arrow[dll, hookrightarrow, dashed, "h"]\\
	\Omega
  \end{tikzcd}
\end{center}

Instead of the strict version $\calQ = (\QQ, \Boxed{\prec_1}, \Boxed{\prec_2})$ of the generic permutation,
for the constructions that follow it will be more convenient to work with the reflexive version
$\calQ_\refl = (\QQ, \Boxed{\preceq_1}, \Boxed{\preceq_2})$. Define $\preceq$ on $\QQ$ as follows:
$$
  a \preceq b \text{ iff } a \preceq_1 b \land a \preceq_2 b.
$$
Note that $\preceq_1$ and $\preceq_2$ are linear orders, while $\preceq$ is a 2-dimensional partial order on $\QQ$.
So, $(\QQ, \Boxed\preceq)$ is a 2-dimensional poset and Theorem~6.4 shows that this poset has big Ramsey degrees.

\begin{THMC}
  $(\QQ, \Boxed\preceq)$ is the generic 2-dimensional poset.
\end{THMC}
\begin{proof}
  As we have argued above, it suffices to show that $(\QQ, \Boxed\preceq)$ is a weak \Fraisse\ limit of the class $\KK_{\le2}$
  of all finite at most 2-dimensional posets.
  
  \bigskip
  
  Claim 1. $\Age(\QQ, \Boxed\preceq) = \KK_{\le2}$.
  
  Proof. $(\subseteq)$ Let $(A, \Boxed{\preceq^A})$ be a finite poset and $f : (A, \Boxed{\preceq^A}) \hookrightarrow (\QQ, \Boxed\preceq)$
  an embedding. Define linear orders $\le^A_1$ and $\le^A_2$ on $A$ as follows:
  $$
    a \mathrel{\le_i^A} b \text{ iff } f(a) \mathrel{\preceq_i} f(b) \text{ in } \calQ_\refl, \quad i \in \{1, 2\}.
  $$
  Then it is easy to check that $\Boxed{\preceq^A} = (\Boxed{\le_1^A}) \cap (\Boxed{\le_2^A})$ proving that
  $(A, \Boxed{\preceq^A}) \in \KK_{\le 2}$.
  
  $(\supseteq)$ Take any $(A, \Boxed{\preceq^A}) \in \KK_{\le 2}$ and let $\le^A_1$ and $\le^A_2$ be linear orders on $A$
  such that $\Boxed{\preceq^A} = (\Boxed{\le_1^A}) \cap (\Boxed{\le_2^A})$. Then $(A, \Boxed{\le_1^A}, \Boxed{\le_2^A})$
  is a finite permutation, so there is an embedding $f : (A, \Boxed{\le_1^A}, \Boxed{\le_2^A}) \hookrightarrow (\QQ, \Boxed{\preceq_1}, \Boxed{\preceq_2})$
  because $(\QQ, \Boxed{\preceq_1}, \Boxed{\preceq_2})$ is the \Fraisse\ limit of the class of all finite permutations.
  Then it is easy to check that $f$ is also an embedding $f : (A, \Boxed{\preceq^A}) \hookrightarrow (\QQ, \Boxed\preceq)$,
  so $(A, \Boxed{\preceq^A}) \in \Age(\QQ, \Boxed\preceq)$.
  
  \bigskip
  
  Claim 2. $(\QQ, \Boxed\preceq)$ is weakly injective for the class $\KK_{\le2}$.
  
  Proof. Take any $(A, \Boxed{\preceq^A}) \in K_{\le 2}$ and any embedding $f : (A, \Boxed{\preceq^A}) \hookrightarrow (\QQ, \Boxed{\preceq})$.
  Define linear orders $\le_1^A$ and $\le_2^A$ on $A$ as follows:
  $$
    a \mathrel{\le_i^A} b \text{ iff } f(a) \mathrel{\preceq_i} f(b) \text{ in } \calQ_\refl, \quad i \in \{1, 2\}.
  $$
  Then $\Boxed{\preceq^A} = (\Boxed{\le_1^A}) \cap (\Boxed{\le_2^A})$, so
  $e : (A, \Boxed{\preceq^A}) \to (A, \Boxed{\le_1^A}) \times (A, \Boxed{\le_2^A})$ given by $e(a) = (a, a)$ is an embedding.
  
  Now, take any $(B, \Boxed{\preceq^B}) \in \KK_{\le 2}$ and any embedding
  $g : (A, \Boxed{\le_1^A}) \times (A, \Boxed{\le_2^A}) \hookrightarrow (B, \Boxed{\preceq^B})$.
  Since the dimension of $(B, \Boxed{\preceq^B})$ is $\le 2$, there are finite linear orders $\le^B_1$ and $\le^B_2$ on $B$
  such that $\Boxed{\preceq^B} = (\Boxed{\le_1^B}) \cap (\Boxed{\le_2^B})$ and
  $\ell : (B, \Boxed{\preceq^B}) \to (B, \Boxed{\le_1^B}) \times (B, \Boxed{\le_2^B})$ given by $\ell(b) = (b,b)$ is an embedding.
  Then
  $$
    \ell \circ g : (A, \Boxed{\le_1^A}) \times (A, \Boxed{\le_2^A}) \hookrightarrow (B, \Boxed{\le_1^B}) \times (B, \Boxed{\le_2^B})
  $$
  is an embedding, so by Lemma~B one of the two maps bellow is an embedding:
  \begin{itemize}
  \item $\ell \circ g : (A \times A, \Boxed{\lex^{A \times A}}, \Boxed{\alex^{A \times A}}) \to (B \times B, \Boxed{\lex^{B \times B}}, \Boxed{\alex^{B \times B}})$ or
  \item $\ell \circ g : (A \times A, \Boxed{\lex^{A \times A}}, \Boxed{\alex^{A \times A}}) \to (B \times B, \Boxed{\alex^{B \times B}}, \Boxed{\lex^{B \times B}})$.
  \end{itemize}
  Without loss of generality we can assume that $\ell \circ g : (A \times A, \Boxed{\lex^{A \times A}}, \Boxed{\alex^{A \times A}})
  \hookrightarrow (B \times B, \Boxed{\alex^{B \times B}}, \Boxed{\lex^{B \times B}})$ is an embedding.
  
  To conclude the proof, note that $e$ given as above is also an embedding
  $e : (A, \Boxed{\le_1^A}, \Boxed{\le_2^A}) \hookrightarrow (A \times A, \Boxed{\lex^{A \times A}}, \Boxed{\alex^{A \times A}})$,
  and that, by construction of $\Boxed{\le_1^A}$ and $\Boxed{\le_2^A}$, $f$ is also an embedding
  $f : (A, \Boxed{\le_1^A}, \Boxed{\le_2^A}) \hookrightarrow (\QQ, \Boxed{\preceq_1}, \Boxed{\preceq_2})$. Since
  $(\QQ, \Boxed{\preceq_1}, \Boxed{\preceq_2})$ is the \Fraisse\ limit of the class of all finite permutations, there is
  an embedding $k : (B \times B, \Boxed{\alex^{B \times B}}, \Boxed{\lex^{B \times B}}) \hookrightarrow (\QQ, \Boxed{\preceq_1}, \Boxed{\preceq_2})$
  such that $k \circ (\ell \circ g \circ e) = f$:
  \begin{center}
    \begin{tikzcd}[column sep=large]
	  (A, \Boxed{\le_1^A}, \Boxed{\le_2^A}) \arrow[d, hookrightarrow, "f"'] \arrow[r, hookrightarrow, "\ell \circ g \circ e"] & (B \times B, \Boxed{\alex^{B \times B}}, \Boxed{\lex^{B \times B}}) \arrow[dl, hookrightarrow, dashed, "k"]\\
	  (\QQ, \Boxed{\preceq_1}, \Boxed{\preceq_2})
	\end{tikzcd}
  \end{center}
  Again, it is easy to verify that $k$ is also an embedding
  $$
    k : (B \times B, \Boxed{\alex^{B \times B}} \cap \Boxed{\lex^{B \times B}}) \hookrightarrow (\QQ, \Boxed{\preceq}),
  $$
  so recalling that $(B \times B, \Boxed{\alex^{B \times B}} \cap \Boxed{\lex^{B \times B}}) = (B \times B, \Boxed{\sqsubseteq^{B \times B}})
  = (B, \Boxed{\le_1^B}) \times (B, \Boxed{\le_2^B})$, we conclude that the following holds in the category of posets of dimension
  at most~2:
  \begin{center}
    \begin{tikzcd}
	  (A, \Boxed{\preceq^A}) \arrow[d, hookrightarrow, "f"'] \arrow[r, hookrightarrow, "e"] & (A, \Boxed{\le_1^A}) \times (A, \Boxed{\le_2^A}) \arrow[r, hookrightarrow, "g"] & (B, \Boxed{\preceq^B}) \arrow[dll, hookrightarrow, dashed, "k \circ \ell"]\\
	  (\QQ, \Boxed{\preceq})
	\end{tikzcd}
  \end{center}
  This concludes the proof.
\end{proof}

\paragraph{Open problem 1.}
Extend the main results of Section~6 and this Corrigendum to classes of relational
structures of the form $(A, \Boxed{<_1}, \ldots, \Boxed{<_n})$ where
$\Boxed{<_1}$, \ldots, $\Boxed{<_n}$ are linear orders on $A$, $n \ge 3$.

\bigskip

The occurrence of the weak amalgamation property in this setting is not incidental.
Indeed, suppose that $\calK$ is a countable relational structure with big Ramsey degrees.
It then follows by compactness that the finite structures in $\Age(\calK)$ have small Ramsey degrees.
Since the joint embedding property is immediate for $\Age(\calK)$, Theorem 3.2 of \cite{masul-zucker-2024}
implies that $\Age(\calK)$ has the weak amalgamation property and
we are precisely in the framework considered here.

\paragraph{Open problem 2.} Find more examples of weak \Fraisse\ limits with big Ramsey degrees.

\section*{Addendum}

The \emph{profile} of a class $\KK$ of finite relational structures is the integer-valued function $\phi_{\KK} : \NN \to \NN$
that assigns to each nonnegative integer $n$ the number of structures in $\KK$ with $n$ elements, counted up to isomorphism.
The behavior of this function has been extensively studied, particularly in the case where $\KK$ is \emph{hereditary}
(i.e., closed under taking substructures) and consists of objects such as graphs (directed or undirected), tournaments,
ordered sets, ordered graphs, or ordered hypergraphs. Moreover, a result of Cameron~\cite{Cameron2003} shows that the study
of permutations—originating from the Stanley--Wilf conjecture and its resolution by Marcus and Tardos~\cite{MarcusTardos2004}—fits naturally into the framework of profiles of hereditary classes of ordered relational structures (see~\cite{OudrarPouzet2011,OudrarPouzet2016}).

These results demonstrate that the profile cannot exhibit arbitrary growth: there are ``jumps'' in the possible growth rates. Typically, the growth is either polynomial or faster than any polynomial (see~\cite{Pouzet1978,Pouzet2006}). For several classes of structures, the profile grows at least exponentially (e.g., tournaments~\cite{BaloghBollobasMorris2007,BoudabousPouzet2010}, ordered graphs and hypergraphs~\cite{BaloghBollobasMorris2006a,Klazar2008}, and permutations~\cite{KaiserKlazar2003}) or at least as fast as the partition function (e.g., graphs~\cite{BaloghBollobasSaksSos2009}). For further details, see the survey by Klazar~\cite{Klazar2010}.  

As noted in~\cite{oudrar-pouzet-2026}, there are very few cases for which the profile $\phi_\KK$ of a hereditary class
$\KK$ is actually a polynomial. This is why we introduce the following relaxation: we say that $\phi_\KK$ is
\emph{eventually polynomial} if for some non-negative integer $n$, the restriction of $\phi_\KK$ to
$\{n, n+1, n+2, \ldots\}$ is a polynomial function.
In particular, Pouzet and Oudrar proved the following in \cite{oudrar-pouzet-2026}, where an \emph{interval decomposition}
of an ordered relational structure is a monomorphic decomposition of the structure into intervals:

\begin{THMD} \cite[cf.~Theorem 1.1]{oudrar-pouzet-2026}
Let $\KK$ be a hereditary class of finite ordered relational structures over a finite relational language. Then:
\begin{itemize}
    \item
      either there exists an integer $k$ such that every member of $\KK$ has an interval decomposition into at most $k+1$
      blocks, in which case $\KK$ is a finite union of ages of ordered relational structures, each having an interval decomposition
      into at most $k+1$ blocks, and the profile of $\KK$ is eventually polynomial (of degree at most $k$);
    \item
      or the profile of\/ $\KK$ is at least exponential.
\end{itemize}
\end{THMD}

\begin{THME}
  Let $\calS$ be a countable relational structure with a total order relation $<$ in its language.
  If the profile of $\Age(\calS)$ is eventually polynomial, then:
  
  $(a)$ $\calS$ has a finite monomorphic decomposition, and
  
  $(b)$ $\calS$ has finite big Ramsey degrees if and only if every substructure of $\calS$ induced by a block in its
  minimal monomorphic decomposition is chained by a chain having finite big Ramsey degrees.
\end{THME}
\begin{proof}
  $(b)$ follows from $(a)$ by Theorems~4.4 and~4.8, while $(a)$ follows from Theorem~D by an easy application of K\H onig's lemma.
  Let us quickly sketch the argument for $(a)$.
  
  Let $\calS$ be a countable first-order structure with a total order relation $<$ in its language,
  and assume that the profile of $\KK = \Age(\calS)$ is eventually polynomial. Then by Theorem~D we know that
  there exists an integer $k$ such that every member of $\KK$ has an interval decomposition into at most $k+1$
  blocks. Note that if $\calB \le \calC$ are two finite substructures of $\calS$,
  then every interval decomposition of $\calC$ into at most $k+1$ blocks has a restriction that is
  an interval decomposition of $\calB$ into at most $k+1$ blocks.
  Now, take any sequence $\calS_1 \le \calS_2 \le \calS_3 \le \ldots$ of finite substructures of $\calS$
  such that $\bigcup_{i \in \NN} \calS_i = \calS$ and construct a finitely branching tree as follows:
  the root of the tree is the trivial (formal) interval decomposition $\{\0\}$ of the empty set;
  the level $i$ of the tree consists of all interval decompositions of $\calS_i$;
  each interval decomposition of $\calS_i$ is joined by an edge to its restriction to $\calS_{i-1}$.
  By K\H onig's lemma there is an infinite branch. Since each decomposition extends the one that
  precedes it, it is clear how to take the supremum along this branch to get a
  valid interval decomposition of $\calS$ into at most $k+1$ blocks.
  Finally, recall that an interval decomposition into finitely many blocks is a special finite monomorphic decomposition.
\end{proof}

\egroup

\section*{Acknowledgements}

The first author was supported by the Science Fund of the Republic of Serbia, Grant No.~7750027:
Set-theoretic, model-theoretic and Ramsey-theoretic phenomena in mathematical structures: similarity and diversity -- SMART.
The second author would like to thank A. D\v zuklevski for many useful discussions during the preparation of this paper.

\textcolor{blue}{The authors are grateful to Jan Hubi\v cka for bringing to our attention a mistake
in the earlier version of the paper.}

\end{document}